\documentclass[12pt]{amsart}
\usepackage[utf8]{inputenc}

\title[$p$-adic iterated integration on semistable curves]{$p$-adic iterated integration and the Frobenius and monodromy operators on semistable curves}
\author{Eric Katz and Daniel Litt}
\date{\today}

\usepackage{natbib}
\usepackage{graphicx}
\usepackage{amsmath}
\usepackage{amssymb}
\usepackage{amsthm}
\usepackage{mathrsfs}    
\usepackage{subcaption}
\usepackage{stmaryrd}
\usepackage{mathtools}

\usepackage{url}
\usepackage[all]{xy}

\usepackage[ left = 1.2in, right = 1.2in, 
             top = 1.2in, bottom = 1.2in ]{geometry}

\numberwithin{equation}{subsection}

\newcommand{\nc}[2]{\newcommand{#1}{#2}}
\nc{\on}{\operatorname}

\nc{\Berk}{\on{Berk}}
\nc{\End}{\on{End}}
\nc{\Gm}{{\mathbf{G}_m}}
\nc{\Ga}{\mathbf{G}_a}
\nc{\Log}{\on{Log}}
\nc{\Spec}{\on{Spec}}
\nc{\Vect}{\on{Vect}}

\nc{\abs}{\on{abs}}
\nc{\an}{\on{an}}
\nc{\conv}{{\on{conv}}}
\nc{\dR}{\on{dR}}
\nc{\gp}{\on{gp}}
\nc{\rig}{\on{rig}}
\nc{\st}{\on{st}}
\nc{\un}{\on{un}}
\nc{\nr}{\on{nr}}

\nc{\cA}{\mathcal{A}}
\nc{\cC}{\mathcal{C}}
\nc{\cD}{\mathcal{D}}
\nc{\cE}{{\mathcal{E}}}
\nc{\cF}{\mathcal{F}} 
\nc{\cG}{\mathcal{G}}
\nc{\cO}{\mathcal{O}}
\nc{\cP}{\mathcal{P}}
\nc{\cQ}{\mathcal{Q}}
\nc{\cU}{\mathcal{U}}
\nc{\cV}{\mathcal{V}}
\nc{\cW}{\mathcal{W}}
\nc{\cX}{\mathcal{X}}
\nc{\cY}{\mathcal{Y}}

\nc{\C}{\mathbb{C}}
\nc{\M}{{\mathbb{M}}}
\nc{\N}{{\mathbb{N}}}
\renewcommand{\P}{\mathbb{P}}
\nc{\Q}{\mathbb{Q}}
\nc{\R}{{\mathbb{R}}}
\nc{\Z}{\mathbb{Z}}

\nc{\sC}{{\mathscr{C}}}

\nc{\val}{{\on{val}}}

\newcommand{\<}{\langle}
\renewcommand{\>}{\rangle}
\nc{\checkthis}{[C]}
\nc{\rfthis}{[R]}
\DeclareMathOperator{\Spf}{Spf}
\DeclareMathOperator{\Res}{Res}
\DeclareMathOperator{\Isoc}{Isoc}
\DeclareMathOperator{\Lie}{Lie}
\newcommand\ps[1]{{\llbracket#1\rrbracket}}

\makeatletter
\newcommand{\exterior}[1]{\mathop{\mathpalette\exterior@{#1}}}
\newcommand{\exterior@}[2]{%
  \raisebox{\depth}{%
  \fontsize{\sf@size}{0}%
  \m@th
  $\ifx#1\displaystyle\textstyle\else#1\fi\bigwedge$}%
  ^{\mspace{-2mu}#2}%
  \kern-\scriptspace
}
\makeatother

\def\presuper#1#2%
  {\mathop{}%
   \mathopen{\vphantom{#2}}^{#1}%
   \kern-\scriptspace%
   #2}

\newcommand{\BC}{{\operatorname{BC}}}

\newcommand\BCint{\presuper\BC\int}
\newcommand\Vint{\presuper{{\operatorname{V}}}\int}
\newcommand\cint{\presuper{{\operatorname{c}}}\int}

\newcommand{\mynewtheorem}[2]{%
  \newtheorem{#1}[subsubsection]{#2}%
  \newtheorem*{#1*}{#2}%
  \newtheorem{#1sub}[subsubsection]{#2}} 

\mynewtheorem{thm}{Theorem}
\mynewtheorem{cor}{Corollary}
\mynewtheorem{conj}{Conjecture}
\mynewtheorem{lem}{Lemma}
\mynewtheorem{prop}{Proposition}
\mynewtheorem{clm}{Claim}

\mynewtheorem{theorem}{Theorem}
\mynewtheorem{corollary}{Corollary}
\mynewtheorem{conjecture}{Conjecture}
\mynewtheorem{lemma}{Lemma}
\mynewtheorem{proposition}{Proposition}
\mynewtheorem{claim}{Claim}

\theoremstyle{definition}
\mynewtheorem{definition}{Definition}
\mynewtheorem{example}{Example}
\mynewtheorem{remark}{Remark}
\mynewtheorem{notn}{Notation}
\mynewtheorem{eg}{Example}
\mynewtheorem{question}{Question}

\theoremstyle{remark}
\mynewtheorem{rem}{Remark}

\begin{document}

\begin{abstract}
We reformulate the theory of $p$-adic iterated integrals on semistable curves using the unipotent log rigid fundamental group. This fundamental group carries Frobenius and monodromy operators whose basic properties are established. By identifying the Frobenius-invariant subgroup of the fundamental group with the fundamental group of the dual graph, we characterize Berkovich--Coleman integration, which is path-dependent, as integration along the Frobenius-invariant lift of a path in the dual graph. Vologodsky's path-independent integration theory which was previously described using a monodromy condition can now be identified as Berkovich--Coleman integration along a combinatorial canonical path arising from the theory of combinatorial iterated integration as developed by the first-named author and Cheng. 
\end{abstract}

\maketitle

\tableofcontents

\section{Introduction}


\subsection{$p$-adic integration}

The theory of line integrals on $p$-adic curves, introduced by Coleman \cite{Coleman:Annals}, produces locally analytic functions that vanish on points of number-theoretic interest. Specifically, the Chabauty method (as summarized in, say \cite{Serre:Lectures}) recognizes the image of torsion (resp., rational points under certain conditions) in the Jacobian as zeroes of (resp., linear projections of) the Lie group-theoretic $p$-adic logarithm. In fact, given a curve $X/\C_p$ with Abel--Jacobi map $\iota_p\colon X\to J$, one identifies $(\Lie J)^\vee$ with $\Omega^1(X)$, the vector space of $1$-forms. Given $\omega\in\Omega^1(X)$, the composition
\[\xymatrix{
X(K)\ar[r]^{\iota_p}& J(X)\ar[r]^{\Log}&\Lie J\ar[r]^\omega &\C_p
}\]
(where $\Log$ is $p$-adic logarithm) can be considered as a function $x\mapsto \int_p^x \omega$.
This type of integration (here called {\em abelian integration}) is global and depends on the Jacobian. On the other hand, Coleman's method of analytic continuation by Frobenius shows that these integrals could be performed {\em locally} on the curve if it is of good reduction. For bad reduction curves, the Lie group-theoretic integral still makes sense, but it does not agree with the analogue of Coleman integration (here called {\em Berkovich--Coleman integration}) as introduced by Coleman and de Shalit \cite{Coleman-deShalit} and fully developed by Berkovich \cite{Berkovich:integration}. For Berkovich--Coleman integration theory, one pieces together locally analytic Coleman functions on good reduction curves and rigid analytic annuli, but unfortunately, the resulting functions may not be single-valued. Instead, one obtains an integration theory depending on the homology class of paths in the dual graph of the curve (equivalently, paths in the Berkovich analytification of the curve). However, it is possible to compare the two notions of integral as was done by Stoll \cite{Stoll:uniform} who analyzed functions on annuli, by Besser--Zerbes \cite{BZ:Vologodsky} who made use of $p$-adic height pairings \cite{Besser:heights}, and by the first author with Rabinoff and Zureick--Brown \cite{KRZB} who studied the tropical Abel--Jacobi map. The comparison between these integration theories was employed in \cite{Stoll:uniform,KRZB} to obtain uniform bounds on rational and torsion points on curves under particular hypotheses.

The $p$-adic analogue of Chen's theory of iterated integrals \cite{Chen:iterated} (see also \cite{Hain:Bowdoin}) was employed in Minhyong Kim's non-abelian Chabauty method \cite{Kim:original} where rational points are obtained as the zeroes of locally analytic functions defined locally by iterated integration. Here, one has an expression  $\int \omega_1\omega_2\dots\omega_r$ where $\omega_i\in \Omega^1(X)$ and sets $F_1=\int \omega_1$, $F_2=\int F_1\omega_2$, \dots, $F_r=\int F_{r-1}\omega_r$. The integral is defined to be $F_r$.
An important challenge in making Kim's work uniform is to define these integrals globally. Now, a change of perspective is necessary. Instead of studying the integral $\int \omega$, one solves the differential equation $ds=\omega$. For iterated integrals, one solves the differential system for functions $F_1,\dots,F_r$
\[dF_0=0,\ dF_1=F_0\omega_1,\ dF_2=F_1\omega_2,\ dF_r=F_{r-1}\omega_r.\]
One can rewrite this system as solving the parallel transport equation $dF-\omega F=0$ on a trivial bundle equipped with connection matrix 
\[\omega=
\left[\begin{array}{cccccc}
0&0&0&\dots&0&0\\
\omega_1&0&0&0&\dots&0\\
0&\omega_2&0&0&\dots&0\\
0&0&\omega_3&0&\dots&0\\
0&0&0&\ddots&\ddots&0\\
0&0&\dots&0&\omega_r&0
\end{array}\right].\]
The notion of solving for parallel transport from $a\in X(K)$ to $b\in X(K)$ can be interpreted as finding an element of $\pi_1^{\dR,\un}(X;a,b)$, the Tannakian fundamental torsor of the category of unipotent bundles with integrable connection on $X$. In \cite{Besser:Coleman}, Besser reformulated Coleman integration on good reduction curves as finding a Frobenius-invariant path from $a$ to $b$. Through an analysis of weights on this fundamental group (which carries a Frobenius action), the path was found to be unique. In the bad reduction case, because of a non-trivial Frobenius-invariant subspace of the first cohomology, this is no longer the case. 
Instead, Vologodsky \cite{Vologodsky} imposed a monodromy condition on Frobenius-invariant paths to pick out a unique path from $a$ to $b$. This gives a unique integral depending only on endpoints. For single integrals, Vologodsky integration coincides with abelian integration. For iterated integrals, it is more mysterious.

In this paper, we reformulate Berkovich--Coleman integration and make Vologodsky integration explicit. We have chosen to use the language of log schemes and the unipotent log rigid fundamental group, i.e.,~ the Tannakian fundamental group of unipotent overconvergent isocrystals on a log curve $(X,M)$. The underlying scheme $X$ will be proper, so we do not have to make use of ovcerconvergence; indeed, our work fits into the {\em log convergent cohomology} and {\em analytic cohomology} frameworks of Shiho \cite{Shiho1,Shiho2}. Here, the coefficient objects are overconvergent isocrystals that can be given a filtration
\[\cE=\cE^0\supset \cE^1\supset \dots \supset \cE^{n+1}=0\]
where $\cE^i/\cE^{i+1}\cong \mathbf{1}^{\oplus n_i}$. Our arguments are inspired by rigid analytic geometry, and in order to reduce to the cases of  good reduction curves and annuli, we develop the theory of log basepoints, an enlargement of the theory of tangential basepoints. We will now describe the main results of our paper working backwards from our results on integration to the foundational work that they require.

\subsection{Berkovich--Coleman and Vologodsky integration}

Let $k$ be a field of finite characteristic; $W=W(k)$, the ring of Witt vectors; $V$, a totally ramified extension of $W$ with uniformizer $\pi$; and $K$ the field of fractions of $V$. 
Let $(X,M)$ be a proper log smooth curve defined over the standard log point $(S,\N)$ (with $S=\Spec k$) with dual graph $\Gamma$. Pick a lift of $X$ to a log smooth formal curve $(\cP,L)$ over $(\Spf V,N)$ where $N$ is the log structure induced by $\Z_{\geq 0}\to V$ given by $e_\pi\mapsto \pi$. Write $\cX$ for the tube $]X[_{\cP}$ (which coincides with the rigid generic fiber of $\cP$ when $X$ is proper).  Let $a$ and $b$ be points of $\cX(K)$ with induced fiber functors $F_a$ and $F_b$. Let $\pi^{\rig,\un}_1((X,M);F_a,F_b)$ be the log rigid unipotent fundamental torsor of $(X,M)$ from $F_a$ to $F_b$.  Let $\omega_1,\dots,\omega_n\in\Omega^1$ be $1$-forms on $\cX$.  For $p\in \pi^{\rig,\un}_1((X,M);F_a,F_b)$, in Section~\ref{s:integrationalongpaths},  we define the integral
\[\int_{p,a}^b \omega_1\dots\omega_r\]
in terms of parallel transport.
For added flexibility, we enlarge the notion of basepoints to {\em log basepoints equipped with lifts}.  
The main result of that section is the following:
\begin{theorem}
The integration theory $\int_{p,a}^b \omega_1\dots\omega_r$ has the following properties:
\begin{enumerate}
    \item (Proposition~\ref{l:multilinear}) multilinearity in $1$-forms;
    \item (Proposition~\ref{p:concatenation}) concatenation: for $p\in \pi_1^{\rig,\un}((X,M);F_a,F_b)$ and $q\in \pi_1^{\rig,\un}((X,M);F_b,F_c)$
\[\int_{pq,a}^c \omega_1\dots \omega_r=\sum_{k=0}^r\int_{p,a}^b \omega_{1}\dots \omega_k\int_{q,b}^c \omega_{k+1}\dots\omega_{r};\]
     \item (Proposition~\ref{p:functoriality})
functoriality: for a morphism of proper log smooth frames (as defined in Section~\ref{s:logrigidisocrystals}) 
    \[f\colon ((X,M_X),(Y,M_Y),(\cP,L))\to ((X',M_{X'}),(Y',M_{Y'}),(\cP',L'))\]
where $(X,M_X)$ and $(X',M_{X'})$ are weak log curves,
  \[\int_{p} f^*\omega'_1\dots f^*\omega'_r=
  \int_{f_*p} \omega'_1\dots\omega'_r;\]
  \item (Lemma~\ref{l:integrationbyparts}) integration by parts; and
  \item (Proposition~\ref{p:symmetrization}) symmetrization:
\[\sum_{\sigma\in S_r} \int_{p,a}^b \omega_{\sigma(1)}\dots\omega_{\sigma(r)}=\prod_{i=1}^r \int_{p,a}^b\omega_i.\]
\end{enumerate}
\end{theorem}

Here, we find it convenient to work with {\em weak log curves} which are a generalization of proper log smooth curves  that allows singularities coming from annular ends of the curve. 

Berkovich--Coleman integration is obtained by restricting the above integration theory on $(X,M)$ to Frobenius-invariant paths which are described using the fundamental group of the dual graph $\Gamma$. Let $F_a,F_b$ be log basepoints of $(X,M)$ anchored at vertices $\overline{a}, \overline{b}\in V(\Gamma)$ corresponding to components of $X$, respectively. There is a {\em specialization map} (Section~\ref{ss:specializationmap})
\[\on{sp}\colon \pi_1^{\rig,\un}((X,M);F_a,F_b)\to \pi_1^{\un}(\Gamma;\overline{a},\overline{b}).\]
Once we pick a branch of $p$-adic logarithm by mandating a value for $\Log(\pi)$ (our choice will be an indeterminate $\ell$), we can define a Frobenius automorphism $\varphi$ on $\pi^{\rig,\un}_1((X,M);F_a,F_b)$. The Frobenius-invariant paths can be identified by using the following result (where the subscript $\ell$ means tensoring with $K[\ell]$):
\begin{proposition}[Proposition~\ref{prop:Frob-invariants}]
Let $(X,M)$ be a weak log curve with dual graph $\Gamma$. Let $F_a,F_b$ be fiber functors induced by log points anchored at $\overline{a},\overline{b}\in V(\Gamma)$, respectively. Then, the specialization  map yields an isomorphism:
\[\on{sp}\colon \pi_1^{\rig,\un}((X,M);F_a,F_b)_{\ell}^{\varphi}\cong \pi_1^{\un}(\overline{\Gamma};\overline{a},\overline{b}).\]
\end{proposition}
This is a natural generalization of the identification of the Frobenius invariants of rigid cohomology in terms of the cohomology of the dual graph \cite{CI:Frobandmonodromy}.
Now, we can define Berkovich--Coleman integration: let $F_a$ and $F_b$ be fiber functors based at log points with lifts (Definition~\ref{d:liftoflogpoint}) $a$ and $b$; let $\overline{p}\in \pi_1^{\un}(\overline{\Gamma};\overline{a},\overline{b})$; we define
\[\BCint_{\overline{p},a}^b \omega_1\dots\omega_r=\int_{p,a}^b \omega_1\dots\omega_r\]
where $p$ is the Frobenius-invariant path specializing to $\overline{p}$.
This definition is a natural generalization of Besser's characterization of Coleman integration as parallel transport along the Frobenius-invariant path.
We show that our definition of Berkovich integration coincides with the usual one. It inherits multilinearity, concatenation, functoriality, integration by parts, and symmetrization properties from the integration theory on $(X,M)$.

Vologodsky defined a path-independent integration theory by making use of the Frobenius and monodromy operators, $\varphi$ and $N$, acting on $\Pi((X,M);F_a,F_b)$, the completed dual coalgebra to the ring of functions on $\pi^{\rig,\un}((X,M);F_a,F_b)$. Specifically, he defines a canonical element $p_{\on{Vol}}$ in $\Pi((X,M);F_a,F_b)$, characterized by the following three properties:
\begin{enumerate}
    \item $p_{\on{Vol}}=1\bmod \mathscr{I}_a$
    \item $\varphi(p_{\on{Vol}})=p_{\on{Vol}}$
    \item $N^r(p_{\text{Vol}})\in W_{-r-1}$ for all $r>0$.
\end{enumerate}
where $W$ is the weight filtration and $\mathscr{I}_a$ is the augmentation ideal.
The Vologodsky integral is defined as 
\[\Vint_a^b \omega_1\dots \omega_r\coloneqq\int_{p_{\on{Vol}},a}^b \omega_1 \dots \omega_r.\] 

We describe Vologodsky integration explicitly by applying Berkovich--Coleman integration to the {\em combinatorial canonical path} 
\[\overline{p}_{\overline{a}\overline{b}}\in \pi_1^{\un}(\Gamma;\overline{a},\overline{b}).\]
This combinatorial canonical path is defined in terms of combinatorial integration as developed in \cite{ChengKatz}.
Specifically, there is a vector space $\Omega^1(\Gamma)$ of tropical $1$-forms on the dual graph $\Gamma$. Given a path $\overline{p}$ between vertices in $\Gamma$, and $\eta_1,\dots,\eta_r\in \Omega^1(\Gamma)$, we define combinatorial iterated integral
\[\cint_{\overline{p}} \eta_1\dots\eta_r.\]
in Section~\ref{s:combint}.
If $\Pi(\Gamma;\overline{a},\overline{b})$ denotes the completed groupoid module of $\pi_1(\Gamma;\overline{a},\overline{b})$ with respect to the augmentation ideal $\mathscr{I}_{\overline{a}}$, and $T_n\Omega$ denotes the length $n$ truncated tensor algebra on a particular subspace $\Omega^1(\overline{\Gamma})\subset \Omega^1(\Gamma)$, combinatorial iterated integration induces a perfect pairing
\[\cint\colon \Pi(\Gamma; \overline a, \overline b)/\mathscr{I}_{\overline a}^{n+1}\otimes T_n\Omega\to K.\]
The combinatorial canonical path $p_{\overline{a}\overline{b}}\in \Pi(\Gamma; \overline a, \overline b)$ is the unique element inducing the truncation map $T_n\Omega\to T_0\Omega=K$ for all $n$.
\begin{theorem}
  The Vologodsky integral is equal to the Berkovich--Coleman integral along the combinatorial canonical path:
  \[\Vint_a^b \omega_1\dots\omega_r=\BCint_{\overline{p}_{\overline{a}\overline{b}},a}^b \omega_1\dots\omega_r.\]
\end{theorem}

To illustrate our results, we provide new formulas and work out Berkovich--Coleman and Vologodsky integrals on a Tate elliptic curve.
The combinatorial canonical path was used to give a description of the unipotent Kummer map in work of Betts--Dogra \cite{BD:Kummer}.

\subsection{Log rigid fundamental groups of curves}

We develop foundational results about the unipotent log rigid fundamental group of curves and log points over the standard log point $(S,\N)$ (where $S=\Spec k$) thickened to the standard formal point $(\Spf V,N)$. 

We make use of alternative descriptions of the fundamental group and fundamental torsor coming from universal objects \cite{Hadian} and the Deligne--Goncharov construction \cite{DG:groupes}. We introduce log basepoints, a new class of fiber functors generalizing tangential basepoints and basepoints attached to residue discs by finding $\Log$-analytic sections (as introduced in \cite{Coleman-deShalit}) of a suitable pullback of a unipotent isocrystal. 

\subsection{Frobenius and monodromy operators on the fundamental group}

There is a considerable technical obstacle to defining the relative Frobenius and monodromy operators on the fundamental group of a log curve. The relative Frobenius is not an endomorphism of a log curve over the standard log point because 
it acts non-trivially on the log structure and thus does not commute
with the structure morphism.  We resolve this problem by adapting Faltings's solution in the case of crystalline cohomology \cite{Faltings:crystalline}: 
embedding the curve into a family over the formal disc $\Spf V\ps{t}$ where it is the fiber over $t=\pi$; one redefines cohomology as the
pushforward from that family to the disc and takes the fiber over
origin; then, Frobenius does act. In our setting, we adapt  Faltings's embedding by a certain modification of the log structure that we call $\pi$-$t$ base-change. By means of a homotopy exact
sequence, we reinterpret 
$\pi_1^{\rig,\un}((X,M);F_a,F_b)$ as
the kernel of homomorphism of Tannakian fundamental groups of weakly unipotent overconvergent $F$-isocrystals on a weak log curve $(X_t,M_t)$ over a log point $(S_t,\N_t)$ corresponding to a standard log disc:
\[\pi_1^{\rig,\un}((X,M);F_a,F_b)_{\ell}\cong \ker\left(\pi_1^{\rig,\varphi,\un(f_t,\nr)}((X_t,M_t);\tilde{F}_a,\tilde{F}_b)_{\ell}\to\pi_1^{\rig,\varphi,\un}((S_t,\N_t);\tilde{F}_a,\tilde{F}_b)_{\ell}\right).\]
Here, our arguments are indebted to \cite{Lazda:rational} and \cite{CPS:Logpi1}. 
We thus obtain a Frobenius automorphism and a monodromy operator
\[\varphi,N\colon \Pi((X,M);F_a,F_b)_{\ell}\to \Pi((X,M);F_a,F_b)_{\ell},\]
and explain their relation to their cohomological analogues.

Later in Section~\ref{ss:monodromycomparison}, we will give a combinatorial description of the action of the power of the monodromy operator $N^n$ on the truncation $\Pi((X,M);F_a,F_b)/\mathscr{I}_a^{n+1}$. This description is the unipotent analogue of the combinatorial formula for the monodromy pairing in Picard--Lefschetz theory (see e.g.,~\cite{Coleman:monodromy,SGA7-1}) and the main technical tool in our description of Vologodsky integration.
To a overconvergent isocrystal $\cE$ of index of unipotency $n$, by considering the residues of the connections induced on $\cE^i/\cE^{i+2}$, one attaches matrices of tropical $1$-forms
\[\eta_i\in \Omega^1(\Gamma)\otimes \on{Hom}(K^{n_i},K^{n_{i+1}}).\]
\begin{proposition} (Proposition~\ref{p:combinatorialmonodromy})
Let $F_a, F_b$ be fiber functors on $(X,M_X)$ attached to log points and anchored at components $\overline{a},\overline{b}\in V(\Gamma)$. Let $p$ be the Frobenius-invariant lift of  $\overline{p}\in\Pi(\Gamma;\overline{a},\overline{b})$. Let $\cE$ be a unipotent overconvergent isocrystal of index of unipotency $n$. The morphism $N^np\colon K^{n_0}=F_a(\cE/\cE^1)\to F_b(\cE^n)=K^{n_n}$ is multiplication by
\[n!\cint_{\overline{p}} \eta_1\dots\eta_n\]
\end{proposition}

\subsection{Acknowledgments}
An enormous debt is owed to Amnon Besser. His paper on Coleman integration through the Tannakian formalism \cite{Besser:Coleman} informed the perspective of our work, and our discussions with him were tremendously helpful. Many of the technical arguments in our work were inspired by those of Chiarellotto, Di Proietto, and Shiho in \cite{CPS:Logpi1}.
We would also like to thank Alexander Betts, Kyle Binder, Patrick Brosnan, Bruno Chiarellotto, and Kiran Kedlaya for valuable conversations.
Gratitude is due to Margaret Lanterman for allowing us to ask questions about log structures.
\part{Graphs and combinatorial integration}

\section{Fundamental groups of graphs}
For our  applications, graphs will have closed and half-open edges with two or one endpoints, respectively. 
Correspondingly, we will modify Gersten's definition of graphs \cite{Gersten:intersections}. Here $\circ$ plays the role of a missing vertex. 

\begin{definition}
A {\em graph} $\Gamma$ is a finite non-empty set with an following data:
\begin{enumerate}  
    \item an involution $\overline{\phantom{x}}\colon \Gamma\to \Gamma$,
    \item a distinguished element $\circ$ that belongs to the fixed point set of $\overline{\phantom{x}}$ which is denoted $\overline{V}(\Gamma)$,
    \item a function $i\colon\Gamma\to \overline{V}(\Gamma)$ satisfying the following conditions: for all $x\in\Gamma$, $i(i(x))=i(x)$; for all $x\in\overline{V}(\Gamma)$, $i(x)=x$; and if $i(x)=i(\overline{x})=\circ$, then $x=\circ$.
\end{enumerate} 
The {\em vertex set} is $V(\Gamma)\coloneqq \overline{V}(\Gamma)\setminus\{\circ\}$. The {\em directed edge set} is $\vec{E}(\Gamma)\coloneqq \Gamma\setminus \overline{V}(\Gamma)$. The set of edges is 
\[E(\Gamma)\coloneqq\{ \{e,\overline{e}\}\mid e\neq \overline{e} \}.\]
An edge $\{e,\overline{e}\}$ with either $i(e)=\circ$ or $i(\overline{e})=\circ$ is said to be {\em half-open}. Otherwise, an edge is {\em closed}.
\end{definition}

The function $i$ takes an edge to its initial point while the involution reverses the orientation of each edge.
We will identify a graph with its topological realization where edges are homeomorphic to $[0,1]$ or $[0,1)$ according to whether they are closed or half-open. A graph is {\em proper} if it has only closed edges. For a graph $\Gamma$, let $\overline{\Gamma}$ be the graph obtained by discarding half-open edges.

The star $\on{St}_v(\Gamma)$ of a vertex in a graph $\Gamma$ is the set of directed edges
\[\on{St}_v(\Gamma)\coloneqq \{e\in \vec{E}(\Gamma)\mid i(e)=v\}.\]
A {\em morphism of graphs} $f\colon \Gamma\to\Gamma'$ is a function $f\colon \Gamma\to\Gamma'$ such that
\begin{enumerate}
    \item $f(\overline{x})=\overline{f(x)}$,
    \item $f(\circ)=\circ$, and
    \item if $i(e)\neq\circ$ then $i(f(e))=f(i(e))$.
\end{enumerate}
Let the uncontracted edge set be
\[\operatorname{NC}(\Gamma,f)\coloneqq\{e\in\Gamma\mid \overline{f(e)}\neq f(e)\}.\]
The morphism $f$ is a {\em submersion of graphs} if for all vertices $v\in V(\Gamma)$, $\on{St}_v(\Gamma)\cap \operatorname{NC}(\Gamma,f)\to \on{St}_{f(v)}(\Gamma')$ is surjective. If the above function on stars is a bijection and $f$ is injective on vertices and closed directed edges, we say that $f$ is a {\em weak embedding}.
Note that weak embeddings can map a half-open edge to a closed edge, a half-open edge, or to the image of its closed endpoint.

We can define {\em log homology} of a graph with coefficients in a field $K$. Let the compactification $\Gamma'$ be obtained from $\Gamma$ by adding distinct endpoints to the half-open edges. Let $\partial\Gamma'$ be the set of those new endpoints, and define 
\[H_n^\diamond(\Gamma)\coloneqq H_n(\Gamma',\partial\Gamma';K).\]
Log homology is a special case of Borel–Moore homology and is contravariant under weak embeddings.

\begin{remark} \label{r:contracttocore}
The natural contraction of half-open edges $\varrho\colon \Gamma\to\overline{\Gamma}$ corresponds to 
an inclusion of pairs
$(\overline{\Gamma},\varnothing)\to(\Gamma',\partial\Gamma')$ that induces 
an inclusion $H^{\diamond}_1(\overline\Gamma)\to H_1^{\diamond}(\Gamma)$
whose image is the subgroup of cycles supported on closed edges.
\end{remark}

Given two vertices $a,b$ on a graph $\Gamma$, we let $\pi_1(\Gamma; a,b)$ be the set of homotopy classes of paths from $a$ to $b$. For a field $K$, we let $K[\pi_1(\Gamma; a,b)]$ be the $K$-vector space generated by $\pi_1(\Gamma;a,b)$. For $a,b,c$ on $\Gamma$, the usual concatenation  
\[\pi_1(\Gamma; a,b)\times \pi_1(\Gamma; b,c)\to \pi_1(\Gamma; a,c)\]
induces a map 
\[K[\pi_1(\Gamma; a, b)]\otimes K[\pi_1(\Gamma; b,c)]\to K[\pi_1(\Gamma; a,c)].\]
The group algebra $K[\pi_1(\Gamma, a)]\coloneqq K[\pi_1(\Gamma; a,a)]$ has an augmentation map 
\[\epsilon\colon K[\pi_1(\Gamma, a)]\to K\]
sending every path to $1$ whose kernel is called the {\em augmentation ideal} $\mathscr{I}_{a}$. 

The ring $K[\pi_1(\Gamma, a)]$ acts naturally on the left on $K[\pi_1(\Gamma; a,b)]$ for any $b$ via concatenation, making $K[\pi_1(\Gamma; a,b)]$ into a free left $K[\pi_1(\Gamma, a)]$-module of rank $1$. Let $\Pi(\Gamma; a,b)$ be the completion of $K[\pi_1(\Gamma; a,b)]$ with respect to the $\mathscr{I}_a$-adic filtration. 

The vector space $K[\pi_1(\Gamma; a,b)]$ naturally has the structure of a coalgebra, with comultiplication defined via 
\[\Delta(p)=p\otimes p\] for 
$p\in\pi_1(\Gamma; a,b)$; this operation extends to $\Pi(\Gamma; a,b)$, giving it the structure of a (topological) coalgebra and making $\Pi(\Gamma, a)$ a cocommutative Hopf algebra.
Then, $\Spec \Pi(\Gamma,a)^\star$ is the pro-unipotent completion 
$\pi^{\un}_1(\Gamma,a)$ over $K$, where $\star$ is the topological dual \cite{vezzani2012pro}. 
This group and the analogously-defined pro-unipotent completion $\pi^{\un}_1(\Gamma;a,b)$ also arise from the Tannakian formalism.

\begin{remark}
One could also complete at the $\mathscr{I}_b$-adic filtration, defined via the right action of $K[\pi_1(\Gamma,b)]$; the two filtrations are the same.
\end{remark}

\section{Combinatorial iterated integrals} \label{s:combint}

\subsection{Definitions}
We will review combinatorial iterated integrals beginning with single integration as described in \cite{BakerFaber,MikhalkinZharkov}

\begin{definition}[Tropical $1$-forms]
A tropical $1$-form is a function $\eta\colon \vec{E}(\Gamma)\to K$ from the set of directed edges to $K$ such that
\begin{enumerate}
    \item $\eta(\overline{e})=-\eta(e)$, and
    \item $\eta$ satisfies the harmonicity condition: for each $v\in V(\Gamma)$,
\[\sum_e \eta(e)=0\]
where the sum is over edges adjacent to $v$ directed away from $v$ (i.e.,~such that $i(e)=v$).
\end{enumerate} 
Let $\Omega^1(\Gamma)$ denote the space of tropical $1$-forms on $\Gamma$.
\end{definition}
There is an isomorphism $\Omega^1(\Gamma)\to H_1^\diamond(\Gamma)$ given by 
\[\eta\mapsto \sum_e \eta(e)e.\]

{\em Single combinatorial integration} is the map $\cint\colon C_1(\Gamma;K)\otimes \Omega^1(\Gamma) \to K$ characterized by
\begin{enumerate}
    \item For a closed edge $e$ of $\Gamma$, considered as an element of $C_1(\Gamma;K)$, $\int_e \omega=\omega(e)$; and 
    \item for each $\omega\in\Omega^1(\Gamma)$, the map $C\mapsto \int_C \omega$ is linear on $C_1(\Gamma;K)$
\end{enumerate}
We may also interpret combinatorial integration as giving maps
\[\cint\colon H_1(\Gamma;K)\otimes \Omega^1(\Gamma) \to K,\quad \cint\colon \Omega^1(\Gamma)\to H_1(\Gamma;K)^\vee.\]

If $\Gamma$ is proper, and one interprets $\Omega^1(\Gamma)$ as $H_1(\Gamma;K)$, combinatorial integration is induced by the Kronecker pairing on edges (making edges into a orthonormal basis), \[C_1(\Gamma;K)\otimes C_1(\Gamma;K)\to K.\] This is the usual cycle pairing \[H_1(\Gamma;K)\otimes H_1(\Gamma;K)\to K\] which gives a combinatorial description of monodromy on the cohomology of a regular semistable family of curves over a discrete valuation ring where $\Gamma$ is the dual graph of the closed fiber (see e.g.,~\cite{Coleman:monodromy}). 
Because the Kronecker pairing, and hence the cycle pairing is nondegenerate, combinatorial integration is a perfect pairing.

Observe that a weak embedding of graphs $\iota\colon \Gamma\to\Gamma'$ induces a pullback $\iota^*\colon \Omega^1(\Gamma')\to \Omega^1(\Gamma)$. This in turn yields a change-of-variables formula: for $C\in H_1(\Gamma)$,
\[\cint_C \iota^*\eta = \cint_{\iota(C)} \eta.\]

We describe the iterated extension of this pairing and its basic properties as  established in \cite[Section~3]{ChengKatz}.
Let $T\Omega=\bigoplus_{k=0}^\infty \Omega^1(\Gamma)^{\otimes k}$ denote the tensor algebra on $\Omega^1(\Gamma)$, 
$T_n\Omega=\bigoplus_{k=0}^n \Omega^1(\Gamma)^{\otimes k}$ be its truncation for $n\in \Z$, and $\hat{T}\Omega=\varprojlim T_n\Omega$ be its completed algebra.
Let $\mathcal{P}(\overline{\Gamma})$ denote the set of paths in the graph $\overline{\Gamma}$. Here, paths are viewed as an ordered list of closed oriented edges.

\begin{definition}[Combinatorial iterated integrals]
Combinatorial iterated integration is a collection of maps for nonnegative $n$, 
\[\cint\colon \mathcal{P}(\overline{\Gamma})\otimes \Omega^1(\Gamma)^{\otimes n}\to K\]
characterized by
\begin{enumerate}
    \item For $n=0$, the map is identically equal to $1$,
    \item For a directed edge $e$, 
 \[  \cint_e \eta=\eta(e),\quad
    \cint_e \eta_1\dots\eta_n=\frac{1}{n!}
     \prod_{i=1}^n \cint_e \eta_i.\]
    \item \label{i:combconcatenation} Let $p_1, p_2$ be paths in $\overline{\Gamma}$ such that the terminal point of $p_1$ is the initial point of $p_2$. Then
    \[\cint_{p_1p_2} \eta_1\cdots \eta_n=\sum_{i=0}^n \left(\cint_{p_1}\eta_1\cdots\eta_i\right)\left( \cint_{p_2}\eta_{i+1}\cdots\eta_n\right).\]
\end{enumerate}
\end{definition}

Instead of using paths in $\overline{\Gamma}$, we may treat combinatorial iterated integrals as defined on homotopy classes of paths in $\Gamma$ with endpoints in $V(\Gamma)$.

\subsection{Properties}
We restate \cite[Proposition~3.5]{ChengKatz}:

\begin{proposition}[Symmetrization relation]\label{prop:shuffle-formula}
Let $S_n$ be the symmetric group on $n$ letters. Then
\[\sum_{\sigma\in S_n}\cint_p\eta_{\sigma(1)}\cdots \eta_{\sigma(n)}=\prod_{i=1}^n \cint_p\eta_i.\]
\end{proposition}

Moreover, combinatorial iterated integration obeys the  more general {\em shuffle formula}: for positive integers $k$ and $\ell$ and a path $p$ in $\Gamma$
\[
    \cint_p \eta_1 \cdots \eta_k \cint_p \eta_{k + 1} \cdots \eta_{k + \ell} = \sum_{\sigma \in \operatorname{Sh}(k,\ell)} \cint_p \eta_{\sigma(1)} \cdots \eta_{\sigma(k + \ell)}.
\]
where $\operatorname{Sh}(k,\ell)$ is the set of $(k,\ell)$-shuffles.
An element $p\in\Pi(\Gamma;a,b)$ with $\epsilon_a(p)=1$ is group-like (i.e.,~ $\Delta(p)=p\otimes p$) if and only if it satisfies the shuffle formula for all choices of $k,\ell,\eta_i$.

For a weak embedding of graphs $\iota\colon\Gamma\to\Gamma'$, a path $p$ in $\Gamma$ and $\eta_1\dots\eta_n\in\Omega^1(\Gamma')^{\otimes n}$, there is a change-of-variables formula,
\[\cint_p \iota^*(\eta_1\dots\eta_n)=\cint_{\iota(p)} \eta_1\dots\eta_n.\]
By applying \eqref{i:combconcatenation} and using inclusion-exclusion, we have 
\begin{proposition}\cite[Theorem~3.8]{ChengKatz} \label{c:uppertriangular} 
Let $p_1,\dots,p_r$ be loops in $\Gamma$ based at $a$ and let $p$ be a path in $\Gamma$ with initial point $a$. Let $\eta_1,\ldots,\eta_r$ be tropical $1$-forms. Then
 \[   \cint_{(p_1 - 1) \cdots (p_r - 1)p} \eta_1 \cdots \eta_n =
      \begin{cases}
        0 & r > n, \\
        \left(\cint_{p_1} \eta_1\right) \dots 
        \left(\cint_{p_r} \eta_n\right) & r = n.
      \end{cases}
\]
\end{proposition}

Consequently, the restriction of combinatorial iterated integration to $T_n\Omega\subset T\Omega$ factors as 
\[
\xymatrix{\cint\colon \Pi(\Gamma; a,b)\otimes T_n\Omega\ar[r]&
\Pi(\Gamma; a,b)/\mathscr{I}_{a}^{n+1}\otimes T_n\Omega\ar[r]&
K.}
\]

The linear extension of combinatorial iterated integrals gives a map
\[\cint\colon \Pi(\Gamma; a, b)\otimes T\Omega\to K.\]
By Corollary~\ref{c:uppertriangular}, the restriction
\[
\cint\colon \mathscr{I}_{a}^n/\mathscr{I}_{a}^{n+1} \otimes T_k\Omega\to K\] 
is the zero map for $k<n$. For $k=n$, there is a factorization
\[\xymatrix{
\cint\colon \mathscr{I}_{a}^n/\mathscr{I}_{a}^{n+1} \otimes T_n\Omega\ar[r]&\mathscr{I}_{a}^n/\mathscr{I}_{a}^{n+1} \otimes \Omega_{\on{trop}}^1(\Gamma)^{\otimes n}\ar[r]&
K
}\] 
where the second arrow is given by 
\[\mathscr{I}_{a}^n/\mathscr{I}_{a}^{n+1} \otimes \Omega_{\on{trop}}^1(\Gamma)^{\otimes n}\cong
H_1(\Gamma;K)^{\otimes n}\otimes \Omega_{\on{trop}}^1(\Gamma)^{\otimes n}\cong\left(H_1(\Gamma;K)\otimes \Omega_{\on{trop}}^1(\Gamma)\right)^{\otimes n}\to K
\]
where the map above is equal to the tensor power of combinatorial single integration. For proper graphs, combinatorial single integration is a perfect pairing, and hence the above pairing is perfect. The upper triangular structure (with non-degeneracy along the diagonal) immediately implies the following (compare \cite[Theorem~3.11]{ChengKatz}):
\begin{proposition}\label{p:perfectpairing}
For a proper graph $\Gamma$, combinatorial iterated integration induces a perfect pairing
\[\cint\colon \Pi(\Gamma; a, b)/\mathscr{I}_{a}^{n+1}\otimes T_n\Omega\to K.\]
\end{proposition} 
  
\begin{proposition}[Combinatorial Canonical Path] \label{p:canonicalpath}
Let $\Gamma$ be a proper graph.
Given $a,b\in V(\Gamma)$, there is a {\em canonical path} $\overline{p}_{ab}\in \Pi(\Gamma;a,b)$ such that
\begin{enumerate}
    \item \[\cint_{\overline{p}_{ab}} 1 = 1\]
    \item for all $n\in \N$ and $\eta_1\dots\eta_n\in \Omega^1(\Gamma)$,
    \[ \cint_{\overline{p}_{ab}} \eta_1\dots\eta_n = 0.\]
\end{enumerate}
The canonical path obeys the shuffle formula, and hence $p_{ab}\in \pi^{\un}_1(\Gamma;a,b)$.
For $a,b,c\in V(\Gamma)$, the canonical path obeys the {\em concatenation property}:
\[\overline{p}_{ac}=\overline{p}_{ab}\overline{p}_{bc}.\]
\end{proposition}

\begin{proof}
By Proposition~\ref{p:perfectpairing}, for any $n$, there is a unique $\overline{p}_n\in \Pi(\Gamma; a, b)/\mathscr{I}^{n+1}$ such that
\[\cint_{\overline{p}_n} 1 = 1, \quad \cint_{\overline{p}_n} \eta_1\dots\eta_k = 0\]
for any $1\leq k\leq n$ and tropical $1$-forms $\eta_1,\dots,\eta_k$.
By the uniqueness of $\overline{p}_n$, under the quotient $\Pi(\Gamma; a, b)/\mathscr{I}^{n+2}\to \Pi(\Gamma; a, b)/\mathscr{I}^{n+1}$, $\overline{p}_{n+1}$ is taken to $\overline{p}_n$. By $\mathscr{I}_{a}$-adic completeness of $\Pi(\Gamma; a, b)$,
$\overline{p}_{ab}$ is the inverse limit of the $\overline{p}_n$'s.

The shuffle relation follows from definitions. The concatenation property follows from the concatenation property of iterated integrals.
\end{proof}

For a non-proper graph $\Gamma$ and $a,b\in V(\Gamma)$, define the canonical path $\overline{p}_{ab}$ to be the pushforward of the canonical path $\overline{p}_{ab}$ from $\overline{\Gamma}$. It vanishes when integrated against tropical $1$-forms in $\Omega^1(\overline{\Gamma})\subset \Omega^1(\Gamma)$.

\begin{remark}\label{r:canonicalfunctorial}
Because morphisms $f\colon\Gamma\to\Gamma'$ do not, in general, induce injections $f^*\colon \Omega^1(\Gamma')\to \Omega^1(\Gamma)$, the canonical path does not enjoy good functoriality properties. 
\end{remark}


Henceforth, paths and vertices will always be named by overlined letters. This will help us refer to paired objects with one in a curve and one in a graph. This overlining does not refer to the reversal of orientation of edges as in the definition of a graph.

\part{The rigid cohomology of log curves and points}

\section{Log rigid isocrystals and cohomology} \label{s:logrigidisocrystals}

We will use the notions of frames and overconvergent isocrystals \cite{Berthelot:cohomologie-rigide,LS:Rigid-cohomology} in the log setting. We recommend \cite{Ogus:logbook} for background on log schemes and \cite{LS:Rigid-cohomology} and \cite[Appendix~A]{LP:Rigidcohomology} for rigid cohomology. A good reference for log rigid cohomology is \cite{GK:Frobandmonodromy}. In all of our cases, the underlying scheme will be proper and is thus equivalent to log convergent cohomology for which \cite{Shiho2} is the supremely useful reference.

Let $W$ be the ring of Witt vectors of $k$, $V$ a finite totally ramified extension of $W$, and $K$ the field of fractions of $V$. Let $S=\Spec k$ be equipped with the log structure $\N$ induced by $\Z_{\geq 0}$ (i.e. generated by $e_\pi$ with morphism of monoids $e_\pi\mapsto 0$). Let $N$ be the log structure on $\Spf V$ induced by $\Z_{\geq 0}$ (i.e.,~generated by $e_\pi$ with $e_\pi\mapsto \pi$ for $\pi$ a uniformizer of $V$). There is an exact immersion  $(S,\N)\to (\Spf V,N)$ given by a closed embedding on schemes and the identity on $\Z_{\geq 0}$. We will work with overconvergent isocrystals thickened along this morphism.

\begin{definition}
A $(\Spf V,N)$-\emph{frame} (frame for short) $((X',M_{X'}),(Y',M_{Y'}),(\cP',L'))$ is a triple consisting of an exact open immersion $(X',M_{X'})\hookrightarrow (Y',M_{Y'})$ of log $(S,\N)$-varieties and an exact closed immersion $(Y',M_{Y'})\hookrightarrow (\cP',L')$ of log formal $(\Spf V,N)$-schemes. A morphism of frames 
\[u\colon ((X',M_{X'}),(Y',M_{Y'}),(\cP',L'))\to ((X'',M_{X''}),(Y'',M_{Y''}),(\cP'',L''))\]
is a triple of morphisms $(X',M_{X'})\to (X'',M_{X''})$, $(Y',M_{Y'})\to (Y'',M_{Y''})$, $(\cP',L')\to (\cP'',L'')$ such that the relevant diagram commutes.

A frame is {\em proper} if $Y'$ is proper. It is {\em log smooth} if there is a log smooth (over $(\Spf V,N)$) open subscheme $(\mathcal{U},L_{\mathcal{U}})\subset (\cP,L')$ containing $X'$.

For an $(S,\N)$-variety $(X,M)$, a \emph{frame over $(X,M)$} is a $(\Spf V,N)$-frame \[((X',M_{X'}),(Y',M_{Y'}),(\cP',L'))\] 
together with an $(S,\N)$-morphism $(X',M')\to (X,M)$. Morphisms of frames over $(X,M)$ are morphisms of frames that commute with the map to $(X,M)$.
\end{definition}

Given a locally closed immersion $\iota\colon X'\hookrightarrow \cP'$ of a $k$-variety into a formal $V$-scheme, we can define the {\em tube}, a rigid analytic space,
\[ ]X'[_{\cP'}=\on{sp}^{-1}(\iota(X'))\]
where $\on{sp}\colon (\cP')^{\an}\to \cP'_k$ is the specialization map. If the tube comes from a frame, we give $]X'[_{\cP'}$ the log structure induced from $(\cP',L)$.
For a frame $((X',M_{X'}),(Y',M_{Y'}),(\cP',L'))$, 
the open immersion $j\colon X'\to Y'$ induces a morphism $j\colon ]X'[_{\cP'}\to ]Y'[_{\cP'}$. Berthelot defined  $j^\dagger$, an exact functor of abelian sheaves on $]Y'[_{\cP'}$. 

\begin{definition} (Compare \cite[Definition~8.1.3]{LS:Rigid-cohomology}) An overconvergent isocrystal $\cE$ on $(X,M)$ over $(\Spf V,N)$ is the following the data:
\begin{enumerate}
    \item For each $(\Spf V,N)$-frame $((X',M_{X'}),(Y',M_{Y'}),(\cP',L'))$ over $(X,M)$, a coherent $j^{\dagger}\cO_{]Y'[_{\cP'}}$-module $E_{\cP'}$,
    \item For each morphism of $(\Spf V,N)$-frames over $(X,M)$ 
    \[((X',M_{X'}),(Y',M_{Y'}),(\cP',L'))\to ((X'',M_{X''}),(Y'',M_{Y''}),(\cP'',L'')),\]
    where $u\colon \cP'\to \cP''$, 
    an isomorphism  
    \[\phi_u\colon u^\dagger E_{\cP''} \to E_{\cP'}\]
    (where $u^\dagger$ is the induced pullback) such that the $\phi_u$'s satisfy the cocycle condition.
\end{enumerate}
A morphism of overconvergent isocrystals $\psi\colon\cE\to\cF$ is a family of compatible morphisms of $j^{\dagger}\cO_{]Y'[_{\cP'}}$-modules
$\psi_{\cP'}\colon E_{\cP'}\to F_{\cP'}$.
Write $\Isoc^\dagger((X,M)/(\Spf V,N))$ for the category of overconvergent isocrystals on $(X,M)$ over $(\Spf V,N)$. We will often suppress ``$(\Spf V,N)$.''
\end{definition}

A morphism $f\colon (X,M_X)\to (Z,M_Z)$ of log $(S,\N)$-schemes induces a pullback functor 
\[f^*\colon \Isoc^\dagger((Z,M_Z)/(\Spf V,N))\to \Isoc^\dagger((X,M_X)/(\Spf V,N)).\]
It is constructed by taking a frame $((X',M_{X'}),(Y',M_{Y'}),(\cP',L'))$ over $(X,M_X)$ to the corresponding frame over $(Z,M_Z)$ whose structure morphism is the composition
\[(X',M_{X'})\to (X,M_X)\to (Z,M_Z).\]

Overconvergent isocrystals form an abelian category. 
The {\em trivial overconvergent isocrystal $\mathbf{1}$} evaluates on any frame $((X',M_{X'}),(Y',M_{Y'}),(\cP',L'))$ as the module  $j^{\dagger}\cO_{]Y'[_{\cP'}}$.
We may define unipotent objects in this category, that is, isocrystals possessing a filtration 
\[\cE=\cE^0\supset \cE^1\supset \dots\supset \cE^{n+1}=0.\]
whose associated gradeds are direct sums of $\mathbf{1}$. 

Henceforth, in all our examples $X$ will be proper, and  we can choose $(Y,M_Y)=(X,M_X)$. In this case, the functor $j^\dagger$ is the identity functor, one does not need to worry about overconvergence, and we will omit ``overconvergent'' when describing isocrystals which will be taken to mean ``convergent isocrystals.'' 
We record the following equivalence of categories:

\begin{theorem} \cite[Section~2.2]{Shiho2} \label{p:equivalenceofcategories}
Let $((X,M_X),(X,M_X),(\cP,L))$ be a proper log smooth frame. Then $\Isoc^{\dagger}((X,M)/(\Spf V,N))$ is equivalent to the category of coherent $\cO_{]X[_\cP}$-modules $E$ equipped with an integrable  log connection 
\[\nabla\colon E\to E\otimes \Omega^1_{(]X[_\cP,L)/((\Spf V)^{\an},N)}
\]
where $\Omega^1_{(]X[_\cP,L)/(S^{\an},N)}$ denotes the sheaf of relative log $1$-forms. This equivalence is canonical for morphisms of proper log smooth frames.
\end{theorem}

Then, $\mathbf{1}$ evaluates on a frame $((X,M_X),(X,M_X),(\cP,L))$ as $\cO_{]X[_{\cP'}}$ equipped with the trivial connection $\nabla=d$.

\begin{definition}
Let $((X,M_X),(X,M_Y),(\cP,L))$ be a log smooth frame.
For an overconvergent isocrystal $\cE$ on $(X,M_X)$, the log rigid cohomology 
\[H^*_{\rig}((X,M_X)/(\Spf V,N),\cE)\]
is defined to be the cohomology of the logarithmic de Rham complex 
$E_\cP\otimes\Omega^\bullet_{(]X[_\cP,L)/((\Spf V)^{\an},N)}$, considered as a complex of sheaves on the admissible site of 
$]X[_{\cP}$.
\end{definition}

We will suppress ``$(\Spf V,N)$'' where it is understood. 
Log rigid cohomology coincides with the {\em analytic cohomology} as defined by Shiho \cite[Section~2.2]{Shiho2} and will be independent of $(\cP,L)$ by \cite[Proposition~2.2.14]{Shiho2}. 

Unipotent isocrystals corresponds to $\cO_{]X[_\cP}$-modules $E$ equipped
with an integrable log connection and a horizontal descending filtration
\[E=E^0\supset E^1\supset \dots \supset E^{n+1}=0\]
whose associated gradeds are isomorphic to direct sums of $(\cO_{(]X[_\cP},d)$. We will make use of the full subcategory consisting of unipotent  isocrystals, which will be denoted by the ``$\un$'' superscript.

By considering the frame $((S, \N),(S,\N),(\Spf V,N))$ and noting that $]S[_{\Spf V}=(\Spf V)^{\an}$ is a point, one immediately obtains the following: 

\begin{corollary} \label{c:isocrystalsonapoint}
The category of unipotent isocrystals on $(S,\N)$ over $(\Spf V,N)$ is equivalent to the category of $K$-vector spaces.
\end{corollary}

A smooth $k$-point $x$ of the underlying scheme $X(k)$ that has the standard log structure (i.e.~the stalk of the relative characteristic sheaf vanishes) can be lifted to a section $x\colon (S,\N)\to (X,M)$. The fiber functor $\cE\to x^*\cE$ allows one to take the fiber over $x$. While unipotent  isocrystals on $(X,M)$ equipped with fiber functor $x^*$ form a Tannakian category, we will define some additional fiber functors in Section~\ref{s:logpoints} using ideas from \cite{Besser:Coleman} and \cite{Vologodsky}.

\section{Log curves}

\subsection{Background on log curves} \label{ss:logcurvesbackground}

We will work with geometrically connected one-dimensional log schemes $(X,M)$ over the following log schemes:
\begin{enumerate}
    \item the standard log point $(S,\N)$ where $S=\Spec k$ and the log structure is induced by $\Z_{\geq 0}\to k$ where $\Z_{\geq 0}$ generated by $e_\pi$ with $e_\pi\mapsto 0$; 
    \item $(S_t,\N_t)$ where $S_t=\Spec k$ and the log structure is induced by $\Z_{\geq 0}^2\to k$ with $\Z_{\geq 0}^2$ generated by $e_t, e_\pi$ with $e_t\mapsto 0$ and $e_\pi\mapsto 0$.
\end{enumerate}
There is a natural log morphism $i\colon (S,\N)\to (S_t,\N_t)$ given on schemes by the identity and on log structures by $e_t\mapsto e_\pi, e_\pi \mapsto e_\pi$. One should think of $i$ as the log analogue of the morphism $\Spf V\to\Spf V\ps{t}$ given by $t\mapsto\pi$. We can make $(S_t,\N_t)$ into an $(S,\N)$-scheme by the morphism $h\colon(S_t,\N_t)\to (S,\N)$ induced by $e_\pi\mapsto e_\pi$. 

We will make use of punctures, nodes, and annuli over $(S,\N)$ and $(S_t,\N_t)$. These structures are also studied in \cite[Section~6]{CPS:Logpi1}. Our log curves will be modeled on the following standard curves over $(S,\N)$:
\begin{enumerate}
    \item The standard line is $(X_0,M_0)$ where $X_0=\Spec k[x_1]$ with $M_0$ induced by the pre-log structure $\Z_{\geq 0}=\langle e_\pi\rangle$ given by $e_\pi\mapsto 0$, and equipped with the map to $(S, \mathbb{N})$ induced by $e_\pi\mapsto e_\pi$;
    \item The standard punctured line is $(X_1,M_1)$ where $X_1=\Spec k[x_1]$ with $M_1$ induced by $\Z_{\geq 0}^2=\langle e_\pi, f\rangle$, where $e_\pi\mapsto 0$ and $f\mapsto x_1$, and equipped with the map to $(S, \mathbb{N})$ induced by $e_\pi\mapsto e_\pi$;
    \item The standard nodal curve is $(X_2,M_2)$ where $X_2=\Spec k[x_1,x_2]/(x_1x_2)$ with $M_2$ induced by $\Z_{\geq 0}^2=\langle f_1,f_2\rangle$ with $f_i\mapsto x_i$, equipped with the map to $(S,\N)$ given by $e_\pi\mapsto f_1+f_2$;
    \item The standard annular line $(X_{2'},M_{2'})$ is the subscheme of $(X_2,M_2)$ cut out by $x_2=0$ with induced log structure, i.e.,~$X_{2'}=\Spec k[x_1]$ with $M_{2'}$ generated by $f_1,f_2$ with $f_1\mapsto x_1$ and $f_2\mapsto 0$, equipped with the map to $(S,\N)$ given by $e_\pi\mapsto f_1+f_2$.
\end{enumerate}

The standard annular line is not log smooth over $(S,\N)$ but belongs to frames that are log smooth. There is a natural closed exact embedding from the standard annular line to the standard node given by mapping to the closed subscheme $x_2=0$.

We will pass from log schemes over $(S,\N)$ to ones over $(S_t,\N_t)$ by a particular base-change.

\begin{definition} \label{d:pitbasechange}
The \emph{$\pi$-$t$ base-change morphism} $h_t\colon (S_t,\N_t)\to (S,\N)$ is given by $e_\pi\mapsto e_t$. For a log scheme $(X,M)$ over $(S,\N)$, we define
\[(X_t,M_t)=(X,M)\times_{(S,\N),h_t} (S_t,\N_t).\]
We define the structure morphism to be the composition
\[\xymatrix{
(X_t,M_t)\ar[r]& (S_t,\N_t)\ar[r]^h& (S,\N).}
\]
\end{definition}

\begin{remark} \label{r:usualbasechange}
The morphism $h\colon (S_t,\N_t)\to (S,\N)$ corresponds  to the projection $\Spf V\ps{t}\to \Spf V$. Base-changing by this morphism is the usual product. The $\pi$-$t$ base-change is not a morphism over $(S,\N)$. It takes semistable curves for which $\pi$ is a deformation parameter for the nodes (i.e.,~modelled on $x_1x_2=\pi$) and replaces them with curves where $t$ is a deformation parameter for nodes (i.e.,~modelled on $x_1x_2=t$). It will be important for defining the monodromy and Frobenius operators on fundamental groups. Observe that base-changing by $i$ is a left inverse to $\pi$-$t$ base-change. We will therefore also write $i\colon (X,M)\to (X_t,M_t)$ for the base-change of $i$ by $(X_t,M_t)\to (S_t,\N_t)$.
\end{remark}

For a log scheme $(X,M)$ over $(S,\N)$ and a closed point $x$, we will say that $x$ is a {\em smooth point}, {\em puncture}, {\em node}, or {\em annular point} if it has an \'etale neighborhood modeled on one the above. 

\begin{definition}
A morphism of fine saturated proper geometrically connected log schemes $f\colon (X,M)\to (S,\N)$ is a {\em weak log curve} if every point is a smooth point, a puncture, a node, or an annular point.
\end{definition}

If a weak log curve does not have any annular points, then it is a log curve. See \cite{Kato:logcurves} for a characterization of such curves.

\begin{remark} \label{r:modifylogstructure}
Given a weak log curve $(X,M)$ over $(S,\N)$, we can modify the log structure at standard annuli and punctures. Specifically, there are morphisms
\begin{enumerate}
    \item $(X_{2'},M_{2'})\to (X_1,M_1)$ induced by the inclusion of pre-log structures $M_1\to M_{2'}$ taking $f\mapsto f_1$;
    \item $(X_1,M_1)\to (X_0,M_0)$ induced by the inclusion of pre-log structures $M_0\to M_1$.
\end{enumerate}
We can replace
any of these log curves by their images by modifying the log structure.
This replacement can be performed \'{e}tale locally and hence on weak log curves. Moreover, it extends to a morphism of log smooth proper frames.
In the rigid analytic picture, the first inclusion corresponds to including an annulus into a punctured disc while the second corresponds to including a punctured disc into a disc. We will refer to a weak log curve thus obtained as a {\em modification}. This is also discussed following the introduction of examples (6.13) and (6.14) in \cite{CPS:Logpi1}.
\end{remark}

Over $(S_t,\N_t)$ we have $(X_{0,t},M_{0,t})$, $(X_{1,t},M_{1,t}), (X_{2,t},M_{2,t}), (X_{2',t},M_{2',t})$ obtained from their $(S,\N)$-analogues by by $\pi$-$t$ base-change.


\begin{definition}
The {\em log dual graph} of a weak log curve $(X,M)$ over $(S,\N)$ is a graph whose vertices correspond to components, whose closed edges correspond to nodal points, and whose half-open edges correspond to annular points and punctures.
\end{definition}

A {\em morphism} of weak log curves $f\colon (X,M_X)\to (Y,M_Y)$ is defined to be the composition of a strict morphism of log curves $(X,M_X)\to (Y,M'_Y)$ (i.e.~$f^*M'_Y\to M_X$ is an isomorphism) and a morphism $(Y,M'_Y)\to (Y,M_Y)$ obtained by a composition of modifications as in Remark~\ref{r:modifylogstructure} at punctures and annular points.
A {\em weak embedding} of weak log curves $f\colon (X,M_X)\to (Y,M_Y)$ is the composition of a strict closed immersion $(X,M_X)\to (Y,M'_Y)$ and modifications $(Y,M'_Y)\to (Y,M_Y)$. A  morphism of weak log curves induces a submersion $\Gamma_X\to \Gamma_Y$ of log dual graphs while a weak embedding induces a weak embedding of log dual graphs. In particular, the weak embedding $(X,M_X)\to (X,M'_X)$ obtained by modifying the log structure by replacing punctures and annular points with smooth points induces the weak embedding $\Gamma_X\to\overline{\Gamma}_X$ in Remark~\ref{r:contracttocore}.

Given a weak log curve $(X,M)$ and a node $p\in X(k)$, we may normalize $X$ at $p$ to produce $f\colon \tilde{X}\to X$. The pullback log structure induces annular points at $f^{-1}(p)$ and a strict morphism $f\colon (\tilde{X},M_{\tilde{X}})\to (X,M)$ of weak log curves which extends to a morphism of log smooth proper frames. The log dual graph $\Gamma_{\tilde{X}}$ is obtained from $\Gamma_X$ by subdividing the edge $e$ corresponding to $p$ and removing the new vertex. There is a weak embedding $\Gamma_{\tilde{X}}\to \Gamma_X$, taking each of the edges arising from the subdivision of $e$ to $e$.

\subsection{Cohomology of log curves} \label{ss:cohomologyoflogcurves}

Let $(X,M)$ be a proper log curve. One can compute log rigid cohomology on a proper log smooth frame
\[((X,M),(X,M),(\cP,L))\]
which exists by standard arguments in deformation theory.
The associated gradeds of the monodromy-weight filtration on $H_{\rig}^1((X,M))$ can be given an explicit description by means of the Steenbrink--Zucker mixed analogue of Steenbrink's spectral sequence  \cite{Steenbrink:Limits}. Specifically, one begins with the double complex $A^{\bullet,\bullet}$ \cite[p.~526]{SteenbrinkZucker}, takes the associated single complex, and writes the spectral sequence induced by the monodromy-weight filtration $M_k$. This also appears with some details in \cite[Section~8.2]{BrosnanElZein}.
The log crystalline analogue of the usual Steenbrink spectral sequence is given by Mokrane \cite{Mokrane}, and once the machinery is developed, adapting it to the Steenbrink--Zucker setting is straightforward. Another approach making use of the Cech spectral sequence is given by Coleman--Iovita \cite[Section~1.1]{CI:Frobandmonodromy} which was generalized to the log crystalline setting in \cite{GK:Cech}. 

Let $(X^{(0)},\tilde{M})$ be the normalization of $X$ with pullback log structure. It will have annular points at the preimages of nodes. Let $(X^{(0)},M^{(0)})$ be obtained from $(X^{(0)},\tilde{M})$ by replacing annuli with punctures as in Remark~\ref{r:modifylogstructure}. Let $X^{(1)}$ and $D$ be the union of the nodes and punctures of $X$, respectively, considered as a disjoint union of copies of $(S,\N)$. 
 The $E_1$-page of the Steenbrink--Zucker spectral sequence is 
\[\begin{array}{ccccc}
E_1^{-1,2}=H_{\rig}^0(X^{(1)}\cup D)(-1)&\rightarrow & E_1^{0,2}=H_{\rig}^2((X^{(0)},M^{(0)})) & & \\
 & & E_1^{0,1}= H_{\rig}^1((X^{(0)},M^{(0)})) & & \\
 & & E_1^{0,0}=H_{\rig}^0((X^{(0)},M^{(0)}))& \rightarrow&
 E_1^{1,0}=H_{\rig}^0(X^{(1)}).
\end{array}\]
The apparent asymmetry of the spectral sequence with respect to $D$ is a consequence of the mixedness of the relevant Hodge structure.
The sequence clearly degenerates at $E_2$. One has
\[E_1^{-i,2}\cong C_i(\Gamma',\partial \Gamma'),\quad E_1^{i,0}\cong C^i(\overline{\Gamma})\]
where $\overline{\Gamma}$ is the dual graph of $X$, $\Gamma$ is the logarithmic dual graph of $(X,M)$, $\Gamma'$ is its compactification, and the $E_1$-differentials is given by the differentials in the relative homology of $(\Gamma',\partial\Gamma')$ and the cohomology of $\overline{\Gamma}$, respectively.
The $E_2$-page of the spectral sequence is
\[\arraycolsep=8pt
\begin{array}{ccc}
H^\diamond_1(\Gamma)(-1) & H^\diamond_0(\Gamma) & \\
 & \bigoplus_{\overline{v}} H_{\rig}^1((X_{\overline{v}},M_{\overline{v}})) & \\
 & H^0(\overline{\Gamma})& H^1(\overline{\Gamma})
\end{array}\]
where $\{(X_{\overline{v}},M_{\overline{v}})\}$ is the set of components of $X^{(0)}$. Observe that $H^i(\overline{\Gamma})\cong H^i(\Gamma)$. 
The top sub-quotient of $H_{\rig}^1((X,M))$, i.e.,~ the quotient, is isomorphic to $H^\diamond_1(\Gamma)(-1)$. The quotient map $\rho\colon H_{\rig}^1((X,M))\to H^\diamond_1(\Gamma)(-1)$ is given by taking residues at nodes and punctures (where log rigid cohomology is interpreted as de Rham cohomology of a tube).
The bottom sub-quotient, i.e.,~the submodule, is isomorphic to $H^1(\overline{\Gamma})$. 

The monodromy operator $N$ on the $E_1$-page is nonzero only from $E_1^{-1,2}$ to $E_1^{1,0}(-1)$ where it is induced by the inclusion $X^{(1)}\hookrightarrow X^{(1)}\cup D$. On the $E_2$-page, it is
\[N\colon H^\diamond_1(\Gamma)(-1)\to H^1(\overline{\Gamma})(-1).\]
The monodromy filtration is thus,
\[M_0=H^1(\overline{\Gamma}),\ M_1=\ker(\rho),\ M_2=H^1_{\rig}((X,M)).\]

The monodromy operator is usually written as the cycle pairing \cite{SGA7-1}
\[H^{\diamond}_1(\Gamma)\otimes H_1(\overline{\Gamma})\to K,\ 
C^\diamond\otimes C\mapsto \sum_e C^\diamond(e)C(e),
\]
but by using the isomorphism between  $H^\diamond_1(\Gamma)$ and $\Omega^1(\Gamma)$, the map $\Omega^1(\Gamma)\otimes H_1(\overline{\Gamma})\to K$ can be rewritten as combinatorial single integration:

\begin{proposition} \label{p:monodromycohomology}
The monodromy map $N$ factors as 
\[\xymatrix{
H_{\rig}^1((X,M))\ar[r]^>>>>>\rho & \Omega^1(\Gamma)\ar[r]^{\cint}& H^1(\overline{\Gamma})\ar[r]& H_{\rig}^1((X,M)).
}
\]
\end{proposition}

\subsection{Log points} \label{ss:logpoints}

In this section, we study the log rigid cohomology of log points, that is, log schemes of the form $(\Spec k,\overline{\M})$ over $(S,\N)$ or $(S_t,\N_t)$. For example, we may take $(\Spec k,\overline{\M}_1)$, $(\Spec k,\overline{\M}_{1,t})$ (resp.~$(\Spec k,\overline{\M}_2)$, $(\Spec k,\overline{\M}_{2,t})$) to be the subschemes of $X_1$, $X_{1,t}$ cut out by $x_1=0$ (resp. $X_2$, $X_{2,t}$ cut out by $x_1=x_2=0$) with induced log structure. For convenience, we will often write $x$ for $\Spec k$.

All of our log points will be induced from a homorphism of monoids $f^*\colon \Z_{\geq 0}=\langle e_\pi\rangle\to \overline{M}=\Z_{\geq 0}^k$. Take the pre-log structure $\overline{M}\to k$ taking each nonzero element of $\overline{M}$ to $0$. Then one obtains a log point $(x,\overline{\M})$ with $x=\Spec k$ and $\overline{\M}=\overline{M}\oplus k^*$.
The structure morphism to $(S,\N)$ is induced by $f^*$, yielding a log point $(x,\overline{\M})\to (S,\N)$.

The log rigid cohomology of $(x,\overline{\M})$ can be computed on a frame of the form \[((x,\overline{\M}),(x,\overline{\M}),(\cP,L))\]
where $\cP=\Spf V\ps{\Z_{\geq 0}^k}/(f^*e_\pi-\pi)$ and $L$ is induced by $\Z_{\geq 0}^k$ with $m\mapsto m$. 
Then, $]x[_{\cP}$, the tube around $x$ in $\cP$, is
\[]x[_{\cP}=\{x\in \cP_K\colon \|m\|<1 \text{ for }m\in M\} \]
where we interpret $m$ as a function on $\cP_K$.
We have the following examples (where unless otherwise noted, $f^*$ is given by $e_{\pi}\mapsto e_{\pi}$):
\begin{enumerate}
    \item $\overline{\M}_0$ induced by $\overline{M}_0\coloneqq\Z_{\geq 0}=\<e_\pi\>$ over $\Z_{\geq 0}$ where $f^*(e_\pi)=e_\pi$, $\cP_0=\Spf V$;
    \item $\overline{\M}_{0,t} $ induced by $\overline{M}_{0,t}\coloneqq \Z_{\geq 0}^2=\<e_\pi,e_t\>$ where $f^*(e_\pi)=e_\pi$, $\cP_{0,t}=\Spf V\ps{t}$;
    \item $\overline{\M}_1$ induced by $\overline{M}_1\coloneqq \Z_{\geq 0}^2=\<e_\pi,f\>$ where $f^*(e_\pi)=e_\pi$, $\cP_1=\Spf V\ps{x_1}$;
    \item $\overline{\M}_{1,t}$ induced by $\overline{M}_{1,t}\coloneqq\Z_{\geq 0}^3=\<e_\pi,e_t,f\>$, where $f^*(e_\pi)=e_\pi$, $\cP_{1,t}=\Spf V\ps{x_1,t}$;
    \item $\overline{\M}_2$ induced by $\overline{M}_2\coloneqq\Z_{\geq 0}^2=\<f_1,f_2\>$ where $f^*(e_\pi)=f_1+f_2$, $\cP_2=\Spf V\ps{x_1,x_2}/(x_1x_2-\pi)$; and
    \item $\overline{\M}_{2,t}$ induced by $\overline{M}_{2,t}\coloneqq\Z_{\geq 0}^3=\langle e_\pi,f_1,f_2\>$ where $f^*(e_\pi)=e_\pi$, $\cP_{2,t}=\Spf V\ps{x_1,x_2,t}$.
\end{enumerate}

\begin{remark}
Observe that $(x,\overline{\M}_{i,t})$ is obtained from $(x,\overline{\M}_i)$ by $\pi$-$t$ base-change and thus possesses a morphism to $(S,\N_t)$.
For $\overline{\M}_i$ and $\overline{\M}_{i,t}$, let $(\cP,L)$ and $(\cP_t,L_t)$ be the formal log schemes obtained
as above. Then $(\cP,L)$ is the closed subscheme cut out by $t=\pi$ from $(\cP_t,L_t)$
\end{remark}

A unipotent isocrystal $\cE$ on $(x,\overline{\M})$ induces a unipotent vector bundle with integrable connection $(E,\nabla)$ on $]x[_{\cP}$. We note the log rigid cohomology of each of the examples by identifying the relevant tubes.

\begin{proposition} \label{p:cohomologyoflogpoint}
In each of the above examples, the log rigid cohomology of $f\colon (x,\overline{\M})\to (S,\N)$ is  (canonically) given by
\[H^*_{\rig}((x,\overline{\M})/(\Spf V,N))\cong \exterior{*}(\overline{M}/f^*\Z_{\geq 0})^{\gp}\otimes K.\]
\end{proposition}

\begin{proof}
For $\overline{M}_0$, $]x[_{\cP}=\Spf V$. For $\overline{M}_{0,t}, \overline{M}_{1}$, $(]x[_{\cP},L)$ is isomorphic to the standard log open disc, i.e.~$\Spf V\ps{t}$ with log structure induced by $\Z_{\geq 0}^2=\<e_\pi,e_t\>$ with $e_\pi\mapsto \pi$, $e_t\mapsto t$. Then the log differentials are generated by $\frac{dt}{t}$, and the cohomology is canonically isomorphic to $\wedge^*(\overline{M}/f^*\Z_{\geq 0})\otimes K$. For $\overline{M}_{1,t}$ and $\overline{M}_{2,t}$, $]x[_\cP$ is isomorphic to the product of two standard log open discs with log differentials generated by $\frac{dx_1}{x_1},\frac{dx_2}{x_2}$ from which the conclusion follows. Now, for $\overline{M}_2$, we have an isomorphism with the annulus
\[]x[_{\cP}\cong \{x_1\mid |\pi|<|x|<1.\}\]
The log differential $\frac{dx}{x}$ generates de Rham cohomology.
\end{proof}

\part{Fundamental groups}

\section{Tannakian fundamental groups}

\subsection{Tannakian categories and fundamental groups}\label{subsection:tannakian-and-fundamental-groups}

In this section, we review Tannakian fundamental groups \cite{Deligne:groupefondamental,Szamuely}.
Let $K$ be a field and $\mathscr{C}$ a essentially small rigid abelian tensor category with unit object $\mathbf{1}_\mathscr{C}$ such that $\on{End}_{\mathscr{C}}(\mathbf{1}_{\mathscr{C}})=K$. Let $F\colon \mathscr{C}\to \on{Vect}_K$ be a fiber functor, i.e.,~a faithful exact $K$-linear tensor functor
so that $\mathscr{C}$ is a \emph{neutral Tannakian category} (see e.g.,~\cite{DM:tannakian} for definitions). Then, given a $K$-algebra $R$, the functor $F\otimes R$ is obtained from the composition 
\[\mathscr{C}\overset{F}{\longrightarrow} \on{Vect}_K\overset{-\otimes R}{\longrightarrow}R\on{-mod}.\]
The main theorem of Tannakian duality \cite[Theorem 2.11]{DM:tannakian} states that the functor \[R\mapsto \on{Aut}^{\otimes}(F\otimes R)\] is representable by a $K$-group scheme $\pi_1(\mathscr{C}, F)$. Moreover for two different fiber functors $F_1, F_2$, the functor \[R\mapsto \on{Isom}^\otimes(F_1\otimes R, F_2\otimes R)\] is represented by a $\pi_1(\mathscr{C}, F_1)$-$\pi_1(\mathscr{C},F_2)$-torsor $\pi_1(\mathscr{C}; F_1, F_2)$ \cite[Theorem 3.2]{DM:tannakian}.

\begin{definition}
We say a Tannakian category $\mathscr{C}$ is {\em unipotent} if every object $E$ of $\mathscr{C}$ admits a filtration $E=E^0\supset E^1\supset \cdots \supset E^{n+1}=0$ (for some $n$) such that each $E^i/E^{i+1}\cong \mathbf{1}_{\mathscr{C}}^{\oplus n_i}$ for some positive integer $n_i$. The maximum value of $n$ for which such a filtration exists (and each quotient is isomorphic to $\mathbf{1}$)
is called the {\em length}. The mimimum value of $n$ is called the {\em index of unipotency}.
\end{definition}
Note that a filtration as above is necessarily of finite length. For a Tannakian category $\mathscr{C}$, write $\mathscr{C}^{\un}$ for the full subcategory of unipotent objects in $\mathscr{C}$. The fundamental torsor $\pi_1(\mathscr{C}^{\un};F_1,F_2)$ is the pro-unipotent completion of 
$\pi_1(\mathscr{C};F_1,F_2)$ as a torsor over $\pi_1(\mathscr{C},F_1)$ or $\pi_1(\mathscr{C},F_2)$

\subsubsection{Hopf algebras} There is a topological cocommutative coalgebra $\Pi(\mathscr{C}; F_1,F_2)$ attached to a unipotent $K$-linear Tannakian category $\mathscr{C}$ with two fiber functors $F_1,F_2$. Let $\mathscr{O}\coloneqq \mathscr{O}_{\pi_1(\mathscr{C}; F_1, F_2)}$ be the coordinate ring of $\pi_1(\mathscr{C}; F_1,F_2)$ --- as the coordinate ring of a (pro-)algebraic variety, it is a commutative algebra. If $F_1=F_2$, it is the coordinate ring of a (pro-)algebraic group, hence a commutative Hopf algebra. In this case, $\Pi(\mathscr{C}; F_1,F_2)$ will be a cocommutative Hopf algebra. The construction of $\Pi(\mathscr{C};F_1,F_2)$ will take a certain (topological) dual of $\mathscr{O}$, meant to imitate the classical formalism of the Mal'cev completion of a discrete group (i.e.,~the completion of the group ring of a discrete group at its augmentation ideal).

We define 
\[\Pi(\mathscr{C}; F_1,F_2)=\varprojlim_{W\subset \mathscr{O}\text{ f.d.}} W^\vee,\]
where  $W$ runs over all finite dimensional subspaces of $\mathscr{O}$. Each $W$ is given the discrete topology and $\Pi(\mathscr{C}; F_1,F_2)$ is given the inverse limit topology (sometimes referred to as the \emph{linearly compact} topology). See e.g.,~ \cite{vezzani2012pro} for a reference on these matters.

The coalgebra $\Pi(\mathscr{C}; F_1,F_2)$ is endowed with a natural augmentation map $\Pi(\mathscr{C}; F_1,F_2)\to K$, dual to the structure map $K\to \mathscr{O}$. We denote its kernel by $\mathscr{I}$. If $F_1=F_2$, $\mathscr{I}$ is an ideal, and its powers form an exhaustive filtration on $\Pi(\mathscr{C}; F_1,F_2)$. For $F_1\not=F_2$, $\pi_1(\mathscr{C}; F_1,F_2)$ is a $\pi_1(\mathscr{C},F_1)$-$\pi_1(\mathscr{C},F_2)$-torsor, and hence $\Pi(\mathscr{C}; F_1,F_2)$ is a free $\Pi(\mathscr{C}, F_1)$-$\Pi(\mathscr{C}, F_2)$-bimodule of rank one. Thus the $\mathscr{I}$-adic filtration on $\Pi(\mathscr{C},F_1)$ (resp.~$\Pi(\mathscr{C}, F_2)$) yields a $\mathscr{I}$-adic filtration on $\Pi(\mathscr{C}; F_1,F_2)$ by left (resp.~right) multiplication --- in fact, one may easily check that these two filtrations agree.
It is a general property of pro-unipotent algebraic groups that $\mathscr{I}/\mathscr{I}^2\cong\pi_1^{\operatorname{ab}}(\mathscr{C},F_1)(K)$ canonically. 

Elements of $\Pi(\mathscr{C}; F_1,F_2)$ give $K$-linear natural transformations $F_1\to F_2$.
The data of $\Pi(\mathscr{C}; F_1,F_2)$ as $F_1,F_2$ range over a collection of fiber functors makes up the data of a \emph{Hopf groupoid} --- see e.g.,~\cite[\S9]{fresse-operads} for details.

\begin{remark}\label{rmk:construction-of-grouplike-elements}
By construction, a $K$-point $p$ of $\pi_1(\mathscr{C}; F_1, F_2)$ (i.e.,~ a map $\mathscr{O}_{\pi_1(\mathscr{C}; F_1, F_2)}\to K$) yields $g_p\in \Pi(\mathscr{C}; F_1, F_2)$ (namely, the image of $1$ under the dual map). Unwinding the duality above shows that this element is \emph{group-like}. Thus, we have, \[\Delta(g_p)=g_p\otimes g_p,\]
where $\Delta$ is the comultiplication on $\Pi(\mathscr{C}; F_1, F_2)$.
\end{remark}
\subsection{Universal objects and unipotent fundamental groups}  \label{ss:unipotentfundamentalgroups}

We summarize a construction of Hadian \cite[Section~2]{Hadian} (see also \cite{AIK}) to characterize the unipotent fundamental group through cohomological data. Let $\mathscr{C}$ be a unipotent neutral Tannakian category such that $\on{Ext}^1(\mathbf{1},\mathbf{1})$ is a finite dimensional $K$-vector space. 
Let $F$ be a fiber functor on $\mathscr{C}$.

\begin{definition}
An {\em $F$-pointed object} is a pair $(V,v)$ where $V$ is an object of $\mathscr{C}$ and $v\in F(V)$. A morphism of $F$-pointed objects $f\colon(V,v)\to(W,w)$ is a morphism $f\colon V\to W$ such that $F(f)(v)=w$. 
\end{definition}

\begin{definition}
A projective system of $F$-pointed objects $\{(E_n,e_n)\}$ such that $E_n$ has index of unipotency $n$ is \emph{universal} if for every pointed object $(V,v)$ with index of unipotency at most $n$, there is a unique morphism $f\colon (E_n,e_n)\to (V,v)$.
\end{definition}

The universal projective system is constructed by iterated extension \cite[Proposition~3.4]{AIK}.
There are objects $\{E_n\}_n$ with $E_0=\mathbf{1}$, and 
an exact sequence of objects
\[
\xymatrix{0\ar[r]&T_n\ar[r]& E_{n+1}\ar[r]& E_n\ar[r]&0}
\]
For an object $V$, write $H^i(V)$ for $\on{Ext}^i(V^\vee,\mathbf{1})$.
Now, 
\[T_n=\on{Ext}^1(E_n,\mathbf{1})^\vee\otimes\mathbf{1}=H^1(E_n^\vee)^\vee\otimes\mathbf{1}\] 
and the extension class in
\[\on{Ext}^1(E_n,T_n)=H^1(E_n^\vee)\otimes H^1(E_n^\vee)^\vee
\]
is given by the identity homomorphism.
There is an alternative description of $T_n$ from \cite[Section~3.6]{AIK}:  $T_n=(R^{(n)})^\vee\otimes \mathbf{1}$ for a vector space $R^{(n)}$ described inductively by  $R^{(0)}=K$, $R^{(1)}=H^1(\mathbf{1})$, and $R^{(n+1)}$ is the kernel of a homomorphism induced by the Yoneda product 
\[R^{(n)}\otimes H^1(\mathbf{1})\to R^{(n-1)}\otimes H^2(\mathbf{1}).\]
The Yoneda product will be equal to the cup product in our examples (see \cite[p.~2614]{AIK}).

There is an isomorphism
\[\on{End}(E_n)\cong F(E_n)\]
given by $f\mapsto f(e_n)$. Composition of endomorphisms induces multiplication on $F(E_n)$.
Set $A_\infty=\varprojlim F(E_n)$. By universality, there is a unique morphism $\Delta_{m,n}\colon E_{m+n}\to E_m\otimes E_n$ taking $e_{m+n}$ to $e_m\otimes e_n$. This induces comultiplication $\Delta\colon A_\infty\to A_\infty\hat{\otimes} A_\infty$ making $A_\infty$ into a cocommutative Hopf algebra.
Set $A_\infty^\star=\on{Hom}_{K, \on{cts}}(A_\infty,K)$, the topological dual as defined in \cite[Definition 6]{vezzani2012pro}.  Then, 
$\pi_1(\mathscr{C},F)\cong \Spec A_\infty^\star.$

\begin{remark}
We may extend the above approach to encompass fundamental torsors. Let $F_1$ and $F_2$ be fiber functors on $\cC$. Let $\{(E^1_n,e^1_n)\}$ and $\{(E^2_n,e^2_n)\}$ be universal $F_1$- and $F_2$-pointed objects. By writing $F_{i,n}$ for the restriction of $F_i$ to objects of index of unipotency at most $n$, the Yoneda Lemma gives a canonical isomorphism
\[F_2(E_n^1)=F_{2,n}(E_n^1)\cong \on{Hom}(F_{1,n},F_{2,n}),\]
the vector space of (not-necessarily tensor compatible) natural transformations.
The collection of vector spaces $H_{n;F_1,F_2} \coloneqq F_2(E_n^1)$
form a projective system with morphism $F_2(E_m^1)\to F_2(E_n^1)$ for $m>n$.
The morphism 
\[F_2(\Delta_{m,n})\colon F_2(E^1_{m+n})\to F_2(E^1_m)\otimes F_2(E^1_n)\]
induces a linear map $\Delta^H_{m,n}\colon H_{m+n;F_1,F_2}\to H_{m;F_1,F_2}\otimes H_{n;F_1,F_2}$.

The projective limits $H_{F_1,F_2}=\varprojlim_n H_{n;F_1,f_2}$
have an obvious composition map 
\[H_{F_1,F_2}\hat{\otimes} H_{F_2,F_3}\to H_{F_1,F_3}.\]
By specializing to $F_1=F_2$ or $F_2=F_3$, we get a left-action of $F_1(E_n)$ and a right-action of $F_2(E_n)$ on $H_{F_1,F_2}.$
There is a comultiplication induced by $\Delta^H_{m,n}$:
\[\Delta\colon H_{F_1,F_2}\to H_{F_1,F_2}\hat{\otimes} H_{F_1,F_2}.\]
In analogy to the above, $\pi_1(\cC;F_1,F_2)\cong \Spec \left(H_{F_1,F_2}^\star\right).$
The construction of the universal bundle yields an alternate construction of $\Pi(\mathscr{C}; F_1,F_2)$ --- namely there is a natural isomorphism, $H_{F_1,F_2}\overset{\cong}{\to} \Pi(\mathscr{C}; F_1, F_2).$
\end{remark}


\subsection{Examples of fundamental groups}

We introduce the unipotent Tannakian categories that will be important in the sequel.

\subsubsection{Fundamental groups of graphs}
\begin{definition}
For a connected graph $\Gamma$, let $\mathscr{C}_\Gamma$ be the Tannakian category of unipotent $K$-local systems with fiber functor given by taking the fiber at a vertex $\overline{a}$. 
\end{definition}


The attached fundamental groups and torsors are $\pi_1^{\un}(\Gamma,\overline{a})$ and $\pi_1^{\un}(\Gamma;\overline{a},\overline{b})$. They coincide with the pro-unipotent completions of the usual  fundamental groups and torsors.

\subsubsection{Log rigid fundamental groups}

Let $(X,M)$ be a geometrically connected log scheme over $(S,\N)$. Let $\sC^{\rig,\un}(X,M)$ be the category of unipotent overconvergent isocrystals on $(X,M)$ over $(\Spf V,N)$. By Corollary~\ref{c:isocrystalsonapoint}, for $(X,M)=(S,\N)$, this is  equivalent to the category of $K$-vector spaces. A section $a\colon (S,\N)\to (X,M)$ induces a pullback functor
\[a^*\colon\sC^{\rig,\un}(X,M)\to \sC^{\rig,\un}(S,\N)\]
and thus a fiber functor on $\sC^{\rig,\un}(X,M)$. Therefore, we may define the fundamental group and torsors $\pi_1^{\rig,\un}((X,M)/(\Spf V,N),a)$ and $\pi_1^{\rig,\un}((X,M)/(\Spf V,N);a,b)$. We will suppress ``$(\Spf V,N)$'' in the sequel. In Section~\ref{s:logpoints}, we will introduce  fiber functors based at points with non-trivial log structures. Observe that $\pi_1^{\rig,\un}$ is functorial for log morphisms over $(S,\N)$.

\subsubsection{de Rham fundamental groups}
The de Rham fundamental group and its comparison to the log rigid fundamental group lies in the background of many of our constructions. Here, we summarize some basic definitions from  \cite{Deligne:groupefondamental}.
Let $X$ be a geometrically connected smooth proper curve over a field $K$, and let $D\subset X$ be a reduced effective divisor. 
\begin{definition}
A \emph{unipotent integrable vector bundle with logarithmic singularities at infinity} on $(X, D)$ is an vector bundle $\mathscr{E}$ on $X$ with integrable logarithmic connection 
\[\nabla\colon \mathscr{E}\to \mathscr{E}\otimes \Omega^1_X(\log D),\]
such that $(\mathscr{E}, \nabla)$ is an iterated extension by trivial  line bundles $(\mathscr{O}_X, d)$. We denote the category of unipotent flat vector bundles with logarithmic singularities at infinity on $(X,D)$ by $\mathscr{C}^{\dR}(X,D)$.
\end{definition}
It is well-known that that the category $\mathscr{C}^{\dR}(X,D)$ is Tannakian. Given a fiber functor $F_1\colon \mathscr{C}^{\dR}(X,D)\to \on{Vect}(K)$ we denote the associated pro-unipotent group scheme by $\pi_1^{\dR}((X,D), F_1)$. If $F_2$ is another fiber functor, we denote the corresponding fundamental torsor by $\pi_1^{\dR}((X,D); F_1,F_2)$

\subsection{The Deligne--Goncharov construction} \label{ss:delignegoncharov}
 The following alternative construction of $\Pi(\sC;F_1,F_2)$ for certain fiber functors works in the log rigid and de Rham settings as well as the topological and \'{e}tale ones not discussed here. Here, a {\em space} will be a proper fine log scheme over $(S,\N)$ or a smooth proper scheme over $\Spec K$.
Let $f\colon X\to T$ be the appropriate structure morphism where $T=(S,\N)$ or $T=\Spec K$. 
Let $\sC$ be the appropriate category of unipotent objects (unipotent isocrystals or unipotent integrable vector bundles). Observe that $\sC(T)$ is equivalent to the category of $K$-vector spaces in either cases. Let $a,b\colon T\to X$ be sections in the appropriate categories which therefore induce fiber functors (also denoted $a,b$) on $\sC(X)$.

Let $\underline{n}$ be the category of finite subsets of $\{0,1, \dots n\}$ with morphisms given by inclusion. For $i\in \{0, 1, \dots, n\}$, let $Y_i\subset X\times_T X \times_T \dots \times_T X$ (where the product is taken $n+2$ times) be the subset $Y_i\coloneqq\{(x_0, \dots, x_{n+1})\mid x_i=x_{i+1}\}.$
Given $J\subset \{0, 1,\dots, n\}$, set  
$Y_J=\bigcap_{j\in J} Y_j.$
For $J_1\subset J_2$, there is a natural inclusion $Y_{J_2}\to Y_{J_1}$, defining a functor $P_X^n$ from the category $\underline{n}^{\on{op}}$ to the category of log schemes. Projection to the first and last factor yield maps 
\[(\pi_0, \pi_{n+1})_J: Y_J\to X^2\]
for each $J$; given $(a,b)\in X^2$, we let $P_{X,a,b}^n$ be the functor sending $J$ to $(\pi_0, \pi_{n+1})_J^{-1}(a,b)$ for all $J$.

Deligne and Goncharov explain how to realize the truncation of the fundamental groupoid module from the cohomology of this diagram of schemes. Indeed, the following is a consequence of the proof of \cite[Proposition~3.4]{DG:groupes} where $\mathbf{1}_\bullet$ refers to the cartesian sheaf on $P_{X, a,b}^n$ whose value on each object is $\mathbf{1}$ (compare also to the main result of \cite{Wojtkowiak:cosimplicial}):
\begin{theorem}\label{t:deligne-goncharov}
There is a canonical isomorphism 
\[\mathbb{H}^{n-1}_\bullet(P_{X, a,b}^n, \mathbf{1}_\bullet)\cong(\Pi(\sC(X); a,b)/\mathscr{I}^n)^\vee.\]
\end{theorem}
One upshot of this canonical isomorphism is that $\Pi(\sC(X); a,b)/\mathscr{I}^n$ immediately obtains all the structure that is present in an ordinary cohomology group; for example, this immediately gives a construction of the $(\varphi, N)$-module structure on $\Pi^{\rig,\un}(X; a,b)/\mathscr{I}^n$ which will be discussed below.

\begin{remark} \label{r:delignegoncharovobject}
The Deligne--Goncharov construction can be used to identify the universal object $E_n$. Indeed, this is how Theorem~\ref{t:deligne-goncharov} is proven. Let $P_{X,a}^n$ be the functor sending $J\subset \{0, 1,\dots, n\}$ to $\pi_0^{-1}(a)\subset Y_J$. The morphism $\pi_{n+1}$ makes $P_{X,a}^n$ into a diagram of schemes over $X$. The pushforward $R^{n-1}\pi_{{n+1}*}\mathbf{1}_\bullet$ is  canonically isomorphic to $E_{n-1}^\vee$. It comes equipped with a homomorphism $(R^{n-1}\pi_{{n+1}*}\mathbf{1}_\bullet)_a\to K$ which is dual to $e_{n-1}\in (E_{n-1})_a$.
This association is functorial: for a morphism $f\colon X_1\to X_2$, by unwinding the construction, there is a morphism $E^1_{n-1}\to f^*E^2_{n-1}$ whose fiber over $a$ takes $e^1_{n-1}$ to $e^2_{n-1}$.
\end{remark}

\section{Log points, fiber functors, and fundamental groups} \label{s:logpoints}

In this section, we study isocrystals on a log point $(x,\overline{\M})$ (as in our examples) using $\Log$-analytic functions as motivated by work of Coleman--de Shalit \cite{Coleman-deShalit} and Berkovich \cite{Berkovich:integration}. This will lead to new fiber functors on $\sC^{\rig,\un}(x,\overline{\M})$ and a computation of the unipotent fundamental group of a log point. The log basepoints defined below were inspired by the residue disc basepoint in \cite[Section~2]{Besser:Coleman}, tangential basepoints in \cite[Section~3]{BesserFurusho}, and sections on annuli in \cite{Vologodsky}. They are related to the basepoints defined in \cite[Section~6]{CPS:Logpi1}.
Let $f\colon (x,\overline{\M})\to(S,\N)$ be one of our examples of log points with $(\cP,L)$ defined as in Subsection~\ref{ss:logpoints}.
A {\em coordinate system on $\overline{M}$} is a subset $B=\{m_1,\dots,m_k\}$ of $\overline{M}$ descending to a basis of $(\overline{M}/f^*\Z_{\geq 0})^{\gp}\otimes K$. 
Write $x_i$ for the coordinate on $\cP$ corresponding to $m_i$.
The module of relative log differentials $\Omega^1_{(]x[_{\cP},L)/((\Spf V)^{\an},N)}$ is free with basis $\frac{dx_1}{x_1},\dots,\frac{dx_k}{x_k}$. 

\subsection{$\Log$-analytic sections}

\begin{definition} Let $\Log(x_1),\dots,\Log(x_k)$ be indeterminates.
A $\Log$-analytic function on $]x[_{\cP}$ is a polynomial in $\Log(x_1),\dots,\Log(x_k)$ whose coefficients are analytic functions on $]x[_{\cP}$. Let $A_{\Log}(]x[_{\cP})$ denote the algebra of $\Log$-analytic functions. For an indeterminate $\ell$, write
\[A_{\Log,\ell}(]x[_{\cP})=A_{\Log}(]x[_{\cP})[\ell]\]

Extend the usual  differential to $d\colon A_{\Log}(]x[_{\cP})\to
\Omega^1_{(]x[_{\cP},L)/((\Spf V)^{\an},N)}\otimes A_{\Log}(]x[_{\cP})$  by mandating $d\Log(x_i)=\frac{dx_i}{x_i}$. It extends to $A_{\Log,\ell}(]x[_{\cP})$ by setting 
$d\ell=0$.
\end{definition}

This definition is inspired by \cite[2.3]{Coleman-deShalit}.
Here, $\ell$ will be interpreted as $\Log(\pi)$, and will be used to fix a branch of logarithm.
The $\Log$-analytic de Rham complex is exact so that any closed $1$-form has a primitive.

\begin{lemma} \label{l:loganalyticsections}
The integrable unipotent  connection $(E,\nabla)$ induced by a unipotent  isocrystal $\cE$ on $(x,\overline{\M})$ has a basis of horizontal $\Log$-analytic
sections on $]x[_{\cP}$.
\end{lemma}

\begin{proof}
By, Proposition~\ref{p:cohomologyoflogpoint}, $H^1_{\dR}((]x[_\cP,L))\cong (\overline{M}/f^*\Z_{\geq 0})^{\gp}\otimes K$ interpreted as log $1$-forms (where $L$ is the induced log structure as in Section~\ref{ss:logpoints}).
By refining the filtration on $\mathscr{E}$ arising from its unipotence, we may suppose that it has $1$-dimensional associated gradeds. 
By induction on the length of the filtration and interpreting the extension group $\on{Ext}^1(\mathbf{1},\mathbf{1})$ as $H^1_{\dR}((]x[_{\cP},L))$, we see that the filtration has a (possibly non-horizontal) splitting 
\[E\cong \bigoplus_{i=0}^{n} \on{gr}^i E\cong \cO^{\oplus n+1};\]
and the connection $1$-form $\omega$ is lower-triangular in the following sense: pick a nonzero (not necessarily horizontal) constant section  $e_k\in \Gamma(\on{gr}^k E)$,
then 
\[\omega e_k=\sum_{i=k+1}^{n} \omega_{ik} e_i\]
for relative log differentials $\omega_{ik}$. For $j=0,\dots,n$, the parallel transport equation $ds_j-\omega s_j=0$ has a basis of solutions given by
\[s_j=\sum s_{ji}e_i\]
where
\[
  s_{ji}=
  \begin{cases}
    0 & \text{if}\ i<j\\
    1 & \text{if}\ i=j\\
    \int\left(\sum_{k=j}^{i-1}\omega_{ik}s_{jk}\right) & \text{if}\ i>j
  \end{cases}
\]
Here, $\int$ refers to formal anti-differentiation of $\Log$-analytic $1$-forms 
which are closed by integrability of the connection.
Our desired basis of horizontal sections is $\{s_0,s_1,\dots,s_n\}$.
\end{proof}

The vector space of such horizontal $\Log$-analytic functions is denoted by $F_B(\cE)$. It gives the \emph{log basepoint fiber functor attached to the basis $B$}, $F_B\colon \cE\mapsto F_B(\cE)$. Write $F_{B,\ell}$ for the functor $F_B\otimes K[\ell]$. It gives $A_{\Log,\ell}(]X[_\cP)$-sections of $\cE$.

Log point fiber functors on $(x,\overline{\M})$ induce those on the log point $(x,\overline{\M}_t)$ obtained by $\pi$-$t$ base-change. Indeed, 
for if $B$ is a coordinate system on $\overline{\M}$, $B\cup \{e_t\}$ is a coordinate system on $\overline{\M}_t$ giving a fiber functor $F_{B\cup \{e_t\}}$

\begin{example}
For $(x,\overline{\M}_2)$, we can define functors $F_{\{f_1\},\ell},F_{\{f_2\},\ell}$. The {\em canonical log path} is the isomorphism of functors
\[\delta_{0}\colon F_{\{f_1\},\ell}\to F_{\{f_2\},\ell}\]
given by the substitution $\Log(x_1)\mapsto\ell-\Log(x_2)$
which is justified because on $]x[_{\cP}$, $x_1x_2=\pi$ (from $f_1+f_2=f^*e_\pi$) and $\Log(\pi)=\ell$. 
Intuitively, the inner and outer boundaries on the annulus $]x[_{\cP}$ should be thought of as $x_1=1$ and $x_2=1$ (which do not belong to $]x[_{\cP}$), and the path goes from one to the other. 

We may replace $\ell$ by an element of $K$ to obtain a non-canonical isomorphism $F_{\{f_1\}}\to F_{\{f_2\}}$ and thus an element of $\pi_1^{\rig,\un}((x,\overline{\M});F_{\{f_1\}},F_{\{f_2\}})$.
\end{example}

\begin{example}
For $(x,\overline{\M}_{2,t})$, there are two analogous fiber functors, $F_{\{f_1,e_t\}}$ and $F_{\{f_2,e_t\}}$, one expressing horizontal sections  using $\Log(x_1)$ and $\Log(t)$, the other using $\Log(x_2)$ and $\Log(t)$.
The {\em canonical log path} $\delta_{0,0}$ between them is induced from the substitution
$\Log(x_1)\mapsto\Log(t)-\Log(x_2)$,
yielding an element of $\pi_1^{\rig,\un}((x,\overline{\M}_{2,t});F_{\{f_1,e_t\}},F_{\{f_2,e_t\}})$.
\end{example}

\begin{remark} \label{r:tangentialbasepoint}
If $\overline{M}\cong f^*\Z_{\geq 0}\oplus \Z_{\geq 0}^{r}$ for $r\geq 1$ (i.e.~when $\overline{M}$ is $\overline{M}_{0,t}$, $\overline{M}_{1}$, $\overline{M}_{1,t}$, or $\overline{M}_{2,t}$), we can define tangential basepoints. The tube $]x[_\cP$ is an open $r$-dimensional polydisc. A coordinate system $B$ gives coordinates on $\cP$. Then, the logarithmic sections can be expressed as polynomials in $\Log(x_1),\dots,\Log(x_r)$ whose coefficients are power series in $x_1,\dots,x_r$. We can specialize $x_i\mapsto 0,\Log(x_i)\mapsto 0$ to get a {\em tangential basepoint} $\tilde{F}$ such that $F(\cE)=E_{0}$, where $0$ is the origin in $]x[_\cP$. This specialization $F_B\to \tilde{F}$ is called the {\em tangential path}.
\end{remark}

%

\begin{definition} \label{d:pushforwardfiberfunctor}
Let $(x,\overline{\M})$ be a log point. let $B$ be a coordinate system on $\overline{M}$. 
Let $(x,\overline{\M}_t)$ be obtained by $\pi$-$t$ base-change. The pushforward $i_*F_{B,\ell}$ by the natural map $i\colon (x,\overline{\M})\to (x,\overline{\M}_t)$, considered as a fiber functor on $\sC^{\rig,\un}(x,\overline{\M}_t)$ will be denoted by $F_{B,\pi}$.
\end{definition}

The fiber functor $F_{B,\pi}$ can be understood explicitly in terms of frames.
Let 
\[((x,\overline{\M}_t),(x,\overline{\M}_t),(\cP_t,L_t))\]
be the frame as above. Let $(\cP,L)$ be the formal subscheme of $(\cP_t,L_t)$ cut out by $t=\pi$ with the induced log structure. Then 
$((x,\overline{\M}),(x,\overline{\M}),(\cP,L))$
is a frame, and there is a natural morphism
\[i\colon ((x,\overline{\M}),(x,\overline{\M}),(\cP,L))\to ((x,\overline{\M}_t),(x,\overline{\M}_t),(\cP_t,L_t))\]
Then, $F_{B,\pi}$ can be interpreted as taking horizontal sections on the fiber over $t=\pi$.

There is a natural path from  $F_{B\cup e_t}\otimes K[\ell]$ to $F_{B,\pi}\otimes K[\ell]$ corresponding to the specialization $t\mapsto \pi$. Because  $\Log$-analytic functions on $(x,\overline{\M}_t)$ involve $\Log(t)$, we will need to choose a value of $\Log(\pi)$. We  make the canonical choice of indeterminate $\ell$ which can be incorporated at the level of pro-algebraic groups. For a pro-algebraic group $H$ over $K$, write $H_\ell$ for $H\times_K \Spec K[\ell]$. If $H$ is the Tannakian fundamental group representing $\on{Isom}^{\otimes}(F_1,F_2)$, for a $K$-algebra $R$, an $R$-point of $H_\ell$ is the data of 
\begin{enumerate}
    \item an element $r\in R$ (i.e.,~the image of $\ell$), and
    \item a tensor isomorphism $F_1\otimes R\overset{\cong}{\to} F_2\otimes R$.
\end{enumerate} 
Here, $r$ will play the role of $\Log(\pi)$. Here, we may abuse notation and consider the case where $R=R'[\ell]$ for some ring $R'$ and the element $r$ is taken to be $\ell$.

\begin{definition} \label{d:specializationpath}
Let $B$ be a coordinate system on $\overline{M}$. 
Then $B\cup \{e_t\}\subset \overline{M}_t$ is a coordinate system on $\overline{M}_t$.
The {\em specialization path} is the $R$-point of $\pi^{\rig,\un}_1((x,\overline{\M}_t);F_{B\cup \{e_t\}},F_{B,\pi})_{\ell}$ considered as the
isomorphism of functors on $\sC^{\rig,\un}(X_t,M_t)$,
$F_{B\cup \{e_t\}}\otimes R\to F_{B,\pi}\otimes R$ 
given by the specializations
\[t\mapsto \pi,\ \Log(t)\mapsto \ell.\]
\end{definition}

The above fiber functors on log points can be used to define fiber functors on weak log curves. 

\begin{definition} \label{d:fiberfunctorlogpoint}
Let $\iota\colon (x,\overline{\M})\to (X,M)$ be an $(S,\N)$-morphism from a log point to a weak log curve. Let $B$ be a coordinate system on $\overline{M}$.
We call $\iota_*F_B$ the {\em fiber functor attached to $\iota$ with respect to $B$}.  When $\iota$ clear, we will write $F_B$.
\end{definition}

\begin{remark}
For a frame around an annular point $(x,\overline{\M}_{2'})$, $f_1\in \overline{M}_{2'}$ maps to the coordinate $x_1$ while $f_2$ maps to $0$. For that reason, it will be most natural to take the fiber functor with respect to $B=\{f_1\}$. Similarly, for the $\pi$-$t$ base-change of an annular point, $(x,\overline{\M}_{2',t})$ over $S_t$, we will take $B=\{e_t,f_1\}$.
\end{remark}

For an $(S,\N)$-morphism $\iota\colon (x,\overline{\M}_0)\to (X,M)$ with target a smooth point of $X(k)$ with relatively trivial log structure, we have the usual fiber functor coming from pulling back by $\iota$ and evaluating on the frame
\[((x,\overline{\M}_0),(x,\overline{\M}_0),(\Spf V,N)).\] 
This is equal to $F_B$ for $B=\{\varnothing\}$.

Following \cite[(2.4)]{Besser:Coleman}, it will be useful to define a fiber functor attached to a residue disc around a smooth $k$-point of $X$.

\begin{definition} \label{d:smoothpoint}
Let $L$ be the log structure on $\Spf V\ps{x_1}$ induced by $\Z_{\geq 0}\to V\ps{x_1}$ with $e_\pi\mapsto \pi$.
Consider the frame
\[((x,\overline{\M}_0),(x,\overline{\M}_0),(\Spf V\ps{t},L)).\] 
The {\em residue disc fiber} functor $F$ on $(x,\overline{\M}_0)$ is given by evaluating a unipotent isocrystal on the above frame and taking horizontal sections on the tube $]0[_\cP$ about the origin in $\Spf V\ps{x_1}$.

Given an $(S,\N)$-morphism  $\iota\colon (x,\overline{\M}_0)\to (X,M)$, the residue disc fiber functor attached to $\iota$ is given by $\iota_*F$.
We may also call $\iota_*F$ the residue fiber functor attached to $\iota(x)\in X(k)$.
\end{definition}

Given a morphisms of frames
\[a\colon ((x,\overline{\M}_0),(x,\overline{\M}_0),(\Spf V,N))
\to ((x,\overline{\M}_0),(x,\overline{\M}_0),(\Spf V\ps{x_1},L)),
\] 
we may define a path from the residue disc fiber functor $F$  to the fiber functor $a^*$ by evaluating analytic sections at the $K$-point in $]0[_{\cP}$ corresponding to $a$. Given morphisms $a_1,a_2$ as above, we can compose such paths to produce a {\em parallel transport path} $a_1^* \to \iota_*F \to a_2^*$ between points in the same residue disc.

\begin{definition}
Let $F$ be a fiber functor attached to a log point of $(X,M)$. If $F$ is induced from a smooth point, puncture, or annular point on an irreducible component $X_{\overline{v}}$, we say $F$ is {\em anchored} at $X_{\overline{v}}$. If $F$ is attached to a nodal point and is of the form $F_{\{f_i\}}$, we say $F$ is {\em anchored} at the component $X_{\overline{v}}$ if $f_i$ maps to a uniformizer on $X_{\overline{v}}$ at the nodal point.
\end{definition}

The following (also observed for log de Rham cohomology in \cite[Subsection~6.6]{CPS:Logpi1}) 
will be useful for computing monodromy:
\begin{lemma} \label{l:aphomotopyequivalence}
Let $(X,M)$ be a weak log curve whose underlying scheme is a smooth curve. Let $F_1,F_2$ be fiber functors attached to  smooth points, punctures, or annular points anchored at $X$. Let $(X,M')$ be obtained from $(X,M)$ by replacing annular points by punctures as in Remark~\ref{r:modifylogstructure}. 
Then the map $g\colon (X,M)\to (X,M')$ induces a commutative diagram whose horizontal maps are isomorphisms
\[\xymatrix{
\pi_1^{\rig,\un}((X,M);F_1,F_2)\ar[r]^{g_*}\ar[d]_{i_*}& \pi_1^{\rig,\un}((X,M');g_*F_1,g_*F_2)\ar[d]_{i_*}\\
\pi_1^{\rig,\un}((X_t,M_t);i_*F_1,i_*F_2)\ar[r]^>>>>{g_{t*}}&\pi_1^{\rig,\un}((X_t,M'_t);(i\circ g)_*F_1,(i\circ g)_*F_2).
}\]
\end{lemma}

The most natural choices for $F_1$ and $F_2$ will be fiber functors attached to an annular point. In this case $g_*F_{\{f_1\}}=F_{\{f\}}$. Indeed, this is a consequence of the morphism  $j\colon (x,\overline{\M}_{2'})\to (x,\overline{\M}_1)$ described in Remark~\ref{r:modifylogstructure} satisfying $j_*F_{\{f_1\}}=F_{\{f\}}$.

\begin{proof}
Note that the tube around $(X,M)$ is obtained from that of $(X,M')$ by removing a closed disc around each puncture.
By 
a Mayer--Vietoris argument using the isomorphism between the log de Rham cohomology of an annulus and a punctured disc, the analogue of this theorem is true with the fundamental group replaced by the log analytic cohomology groups of Shiho \cite{Shiho2}. Because log analytic cohomology is isomorphic to log convergent cohomology \cite[Corollary~2.3.9]{Shiho2}, which is computed on a site, by identifying $H^i_{\rig}((X,M),\cE^\vee)$ with $\on{Ext}^i_{(X,M)}(\cE,\mathbf{1})$, we have an isomorphism of extension groups of log convergent isocrystals, $\on{Ext}^i_{(X,M)}(g^*\cE,\mathbf{1})\cong \on{Ext}^i_{(X,M')}(\cE,\mathbf{1})$. By an induction on the length of unipotent objects, $g^*\colon \sC^{\rig,\un}(X,M')\to \sC^{\rig,\un}(X,M)$ is an equivalence of categories. The argument for $(X_t,M_t)$ is analogous.
\end{proof}

\subsection{Fundamental groups of log points}

\begin{lemma} \label{l:pi1logpoint} Let $(x,\overline{\M})$ be a log point over $(S,\N)$. Let $B$ be a coordinate system on $\overline{M}$. Then 
\[\pi_1^{\rig,\un}((x,\overline{\M}),F_B)\cong \on{Hom}((\overline{M}/f^*\Z_{\geq 0})^{\gp},\Ga).\]
This isomorphism is functorial for log points over $(S,\N)$.

For a unipotent isocrystal $\cV$, the $K$-points of the fundamental group act on $F_B(\cV)$ by 
\[h\in \on{Hom}((\overline{M}/f^*{\Z_{\geq 0}})^{\gp},\Ga)(K)\mapsto[\Log(x_i)\mapsto \Log(x_i)+h(m_i)]\]
\end{lemma}

\begin{proof}
By Proposition~\ref{p:cohomologyoflogpoint}, the log rigid cohomology $H^*_{\rig}((x,\overline{\M}))$ is isomorphic to $\exterior{*}U$ where $U\coloneqq(\overline{M}/f^*\Z_{\geq 0})^{\on{gp}}\otimes K$. 
By the recipe in \cite[Section~3]{AIK} as summarized in Subsection~\ref{ss:unipotentfundamentalgroups}, $\cE_n$ is a trivial bundle on $]X[_{\cP}$ with fiber equal to the truncated symmetric algebra on $U^\vee$, 
\[\on{Sym}^{\leq n} U^\vee=\bigoplus_{i=0}^n \on{Sym}^i U^\vee.\]
Indeed, we can show $R^{(n)}=\on{Sym}^n U$ by induction: $R^{(1)}=U$, 
\[R^{(n+1)}=\ker \left(\on{Sym}^n U\otimes U\to \on{Sym}^{n-1}U\otimes \exterior{2} U\right)\]
where the homomorphism is given by the wedge product of the last two entries.
By the description of $\on{Ext}(\mathbf{1},\mathbf{1})$ as $H^1_{\rig}((x,\overline{\M}))$, we can characterize the connection on $\cE_n$. 
Let $U^\vee$ act on $\on{Sym}^{\leq n}U^\vee$ by multiplication giving a homomorphism $U^\vee\to \End(\on{Sym}^{\leq n}U^\vee)$. Then, the image of the identity element under $U\otimes U^\vee\to U\otimes \End(\on{Sym}^{\leq n}U^\vee)$ gives a connection $1$-form, which can be written in coordinates as
$\omega=\sum_i \frac{dx_i}{x_i}\otimes m_i^\vee.$
The connection, $\nabla=d-\omega$ has a horizontal $\Log$-analytic section 
\[e_n=\exp\left(\sum_i \Log(x_i)m_i^\vee\right).\]
Therefore, there is an isomorphism
\[\on{Sym}^{\leq n} U^\vee\to F_B(\cE_n),\quad
y\mapsto y\exp\left(\sum_i \Log(x_i)m_i^\vee\right).\]

Hence, $A_\infty=\varprojlim F_B(\cE_n)\cong \widehat{\on{Sym}}^* U^\vee$
is the formal power series ring on $U^\vee$.
The comultiplication $A_\infty\to A_\infty\hat{\otimes} A_\infty$ is induced from the morphism
$\cE_{n+m}\to \cE_{n}\otimes \cE_m$
taking $e_{n+m}\mapsto e_n\otimes e_m$. By horizontality, this morphism must take $m^\vee\in U^\vee$ to $m^\vee\otimes 1+1\otimes m^\vee$, hence elements of $U^\vee$ are primitive.
Now,
\[\pi_1((x,\overline{\M}),F_B))\cong\Spec(A_\infty^\star)
\cong\Spec(\on{Sym}^* U)
\cong\on{Hom}((\overline{M}/f^*\Z_{\geq 0})^{\gp},\Ga).
\]
Its $K$-points correspond to group-like elements of $A_\infty$ which are of the form $\exp(m^\vee)$ for $m^\vee\in U^\vee$. The action of such a point $h\in \on{Hom}((\overline{M}/f^*\Z_{\geq 0})^{\gp},K)$ on $F_B(\cE_n)$ is
\begin{align*}
h\cdot y\exp\left(\sum \Log(x_i)m_i^\vee\right)
&=y\exp\left(\sum \Log(x_i)m_i^\vee\right)\exp\left(\sum h(m_i) m_i^\vee\right)\\
&=y\exp\left(\sum (\Log(x_i)+h(m_i))m_i^\vee\right).
\end{align*}

For $\cV$, a unipotent isocrystal, 
the action of $\pi_1((x,\overline{\M}),F_B)$ on $F_B(\cV)$ is inherited from that on $(\cE_n,e_n)$ by  universality.
\end{proof}

\begin{remark} \label{r:fiberfunctorsonlogdisc}
The above arguments show that the $K$-points of $\pi_1((S_t,\N_t),F_{\{e_t\}})\cong \Ga$ are given by the automorphism
$\gamma_u$ of $F_{\{e_t\}}$ induced by the substitution
\[\gamma_u=[\Log(t)\mapsto \Log(t)+u]\]
for $u\in K$. 
These elements can be related to the description of monodromy on cohomology coming from the Gauss--Manin connection (see \cite{Faltings:crystalline} and \cite[Section~2.1]{CI:Hidden}). Indeed, let $\cE$ be a unipotent  isocrystal on $(S_t,\N_t)$. It induces a unipotent vector bundle with connection $(E,\nabla)$ on the disc $(\Spf V\ps{t})^{\an}$. The $\Log$-analytic sections of $E$, which form a vector space $F_B(\cE)$, can be written as tuples of $\Log$-analytic functions as in the proof of Lemma~\ref{l:loganalyticsections}.
There is a natural isomorphism between $F_B$ and $\tilde{F}$ given by the tangential path $t\mapsto 0$, $\Log(t)\mapsto 0$ as in Remark~\ref{r:tangentialbasepoint}. By conjugating with the tangential path, we see that $\gamma_u$ induces an automorphism of $E_0$, the fiber of $E$ over the origin of $(\Spf V\ps{t})^{\an}$.

The action of $\pi_1((S_t,\N_t),F_{\{e_t\}})$ on $F_B(\cE)$ induces a homomorphism from the Lie algebra of $\Ga$ to the endomorphism algebra of the vector space $F_B(\cE)$. Let $D$ be the image of a generator of $\Ga$. Again, pick a non-horizontal trivialization of $E$ so $\nabla$ has connection matrix $\omega$. Let $\Res_0(\nabla)$ is the entry-wise residue of the connection matrix. 
Because $\Ga$ acts component-wise on sections by the substitution above, $D$ acts component-wise on $\Log$-analytic functions as the derivation given by $t\mapsto 0$, $\Log(t)\mapsto 1$. 
We can write the action of $D$ on $\tilde{F}(\cE)$ more explicitly by putting $E$ into standard form analogous to \cite[II.1.17]{Deligne:regularsingular} (compare \cite[p.~185]{CI:Hidden}). Let $\underline{E}_0$ be the trivial bundle on $(\Spf V\ps{t})^{\an}$ with fiber $E_0$. Equip $\underline{E}_0$ with the connection $\nabla_0$ given by $\omega_0=\Res_0(\nabla)\frac{dt}{t}$. By the proof of  Lemma~\ref{l:loganalyticsections}, $(E,\nabla)\cong (\underline{E_0},\nabla_0)$ by an isomorphism that restricts to the identity on the fiber over the origin. Then, the $\Log$-analytic sections of $(\underline{E_0},\nabla_0)$ are 
\[\exp(\Res_0(\nabla)\Log(t))e\]
for $e\in E_0$ (where we note that exponentiation is well-defined because $\Res_0(\nabla)$ is nilpotent).
Now, $D$ acts on such a section by the following formal differentiation
\begin{align*}
D(\exp(\Res_0(\nabla)\Log(t))e)&=\frac{\partial}{\partial u}(\exp(\Res_0(\nabla)(\Log(t)+u))e)|_{u=0}\\
&=\Res_0(\nabla)\exp(\Res_0(\nabla)\Log(t))e.
\end{align*}
Consequently, the action of $D$ on $\tilde{F}(\cE)=E_0$ is given by multiplication by 
$\Res_0(\nabla)$.
\end{remark}

\begin{example} \label{e:pi1annulus}
For $i=1,2$,  $\pi^{\rig,\un}_1((x,\overline{\M}_2),F_{\{f_i\}})\cong \Ga$ with its $K$-points given by $\{\gamma_{i,u}\}_u$ for $u\in \Ga(K)$ with
\[\gamma_{i,u}= \left[\Log(x_i)\mapsto \Log(x_i)+u\right]\]
After tensoring with $K[\ell]$, we see that $\pi_1((x,\overline{\M}_2);F_{\{f_1\}},F_{\{f_2\}})_{\ell}$ is the torsor
over each group with $K[\ell]$-points  given by
\[\{\delta_u\}_{u\in (\Ga)_{\ell}(K[\ell])},\ \delta_u=[\Log(x_1)\mapsto \ell-\Log(x_2)+u].\]
The action of $\pi_1((x,\overline{\M}_2),F_{f_1})_{\ell}$ on the left and $\pi_1((x,\overline{\M}_2),F_{f_2})_{\ell}$ on the right gives 
$\gamma_{1,u_1}\delta_{u}\gamma_{2,u_2}=\delta_{u+u_1-u_2}.$
\end{example}

\begin{example} \label{e:pi1t-annulus}
In analogy with the above, one can show for $i=1,2$, \[\pi_1((x,\overline{\M}_{2,t}),F_{\{f_i,e_t\}})\cong \Ga\times\Ga\]
and its $K$-points are given by 
\[\gamma_{i,(u,u_t)}= \left[\begin{aligned}
\Log(x_i)&\mapsto \Log(x_i)+u\\
\Log(t)&\mapsto \Log(t)+u_t.
\end{aligned}
\right]
\]
The $K$-points of 
$\pi_1((x,\overline{\M}_{2,t});F_{\{f_1,e_t\}},F_{\{f_2,e_t\}})$ form a torsor
\[\{\delta_{(u,u_t)}\}_{(u, u_t)\in \Ga(K)\times\Ga(K)}\]
where 
\[\delta_{(u,u_t)}=
\left[
\begin{aligned}
\Log(x_1)&\mapsto \Log(t)-\Log(x_2)+u\\
\Log(t)     &\mapsto \Log(t)+u_t
\end{aligned}
\right]\]
The action of $\pi_1((x,\overline{\M}_{2,t}),F_{\{f_1,e_t\}})$ on the left and $\pi_1((x,\overline{\M}_{2,t}),F_{\{f_2,e_t\}})$ on the right gives
\[\gamma_{1,(u_1,u_{t_1})}\delta_{(u,u_u)}\gamma_{2,(u_2,u_{t_2})}=\delta_{(u+u_1-u_2+u_{t_2},u_t+u_{t_1}+u_{t_2})}.\]
\end{example}

\subsection{K\"{u}nneth formula}

The unipotent fundamental group of a Cartesian product is the direct product of the unipotent fundamental groups of the factors.

\begin{proposition} \label{p:kunnethpi1}
Let $(X_1,M_1)$ and $(X_2,M_2)$ be proper log schemes over $(S,\N)$. Let 
\[(X,M)=(X_1,M_1)\times_{(S,\N)} (X_2,M_2)\]
with projections $q_i\colon X\to X_i$. 
Let $F_a,F_b$ be fiber functors on $(X,M)$. Then,
\[\pi^{\un}_1((X,M);F_a,F_b)\cong \pi^{\un}_1((X_1,M_1);q_{1*}F_a,q_{1*}F_b)\times \pi^{\un}_1((X_2,M_2);q_{2*}F_a,q_{2*}F_b).\]
\end{proposition}

\begin{proof}
By applying the K\"{u}nneth formula for $\on{Ext}^i$ in the construction of the universal bundle, we obtain the isomorphism of projective systems
\[\cE_{X_1\times X_2}\cong q_1^*\cE_{X_1}\otimes q_2^*\cE_{X_2}\]
with filtration induced by that on each tensor factor. Moreover, 
\[(e_{X_1,n}\otimes e_{X_2,n})\in F_b((q_1^*\cE_{X_1}\otimes q_2^*\cE_{X_2})_n)\]
satisfies the properties for a system of basepoints when $e_{X_i,n}\in (q_{i*}F_b)(\cE_{X_i,n})$ is a system of basepoints.
The conclusion follows from the construction of the unipotent fundamental group.
\end{proof}

\subsection{Specialization map} \label{ss:specializationmap}

Let $(X,M)$ be a weak log curve with log dual graph $\Gamma$, and let $F_a,F_b$ be fiber functors on $\sC^{\rig,\un}(X,M)$ attached to log points anchored on irreducible components $X_{\overline{a}}$ and $X_{\overline{b}}$, corresponding to $\overline{a},\overline{b}\in V(\Gamma)$. We will define a {\em specialization map}
\[\pi_1^{\rig,\un}((X,M);F_a,F_b)\to \pi_1^{\un}(\Gamma;\overline{a},\overline{b}).\]
The specialization map is induced by the cospecialization functor of Tannakian categories, $\on{sp}^*\colon \sC^{\un}(\Gamma)\to \sC^{\rig,\un}(X,M)$ which we describe below. An object in $\sC^{\un}(\Gamma)$ is a unipotent local system of $K$-vector spaces on $\Gamma$, and can therefore be represented by the following data: 
\begin{enumerate}
    \item a nonnegative integer $N$ and
    \item for each closed directed edge $e=\overline{v}\overline{w}$, a linear isomorphism $T_e\colon K^N\to K^N$ given by a lower-triangular matrix with $1$'s on the diagonal such that $T_{\overline{e}}=T_e^{-1}$.
\end{enumerate}
 We can define a unipotent isocrystal on $(X,M)$ as follows: on each $X_{\overline{v}}$, set $\cF_{\overline{v}}=\mathbf{1}^{\oplus N}$; to each node $X_e$ corresponding to $e=\overline{v}\overline{w}$, we attach the morphism $\cF_{\overline{v}}|_{X_e}\to \cF_{\overline{w}}|_{X_e}$ induced by $T_e$. By proper descent \cite{Lazda:descent}, this defines a (unipotent) isocrystal $\cF$. The isocrystal can be defined explicitly by picking a log smooth frame $((X,M),(X,M),(\cP,L))$. On $]X_{\overline{v}}[_{\cP}$, take the module with trivial connection $(\cO_{]X_{\overline{v}}[_{\cP}}^{\oplus N},d)$ and
glue the modules on $]X_e[_{\cP}\ \subseteq\ ]X_{\overline{v}}[_{\cP}\cap ]X_{\overline{w}}[_{\cP}$ by the constant linear map $T_e$.  This gives a unipotent $\cO_{]X[}$-module with a unipotent integrable connection, thus a unipotent isocrystal on $(X,M)$.

Let $F_a$ is a fiber functor on $\sC^{\rig,\un}(X,M)$ corresponding to a log point anchored at a component $X_{\overline{a}}$, it is straightforward to see that the fiber functor $F_a\circ \on{sp}^*$ on $\sC(\Gamma)$ is canonically isomorphic to the fiber functor taking the fiber at $\overline{a}$.

We summarize the discussion as follows:
\begin{definition}[Specialization map]
Let $(X,M)$ be a weak log curve over $(S,\N)$ with log dual graph $\Gamma$. Let $F_a,F_b$ be fiber functors attached to log basepoints of $(X,M)$ anchored at components $X_{\overline{a}}, X_{\overline{b}}$, respectively. The {\em specialization map}
\[\on{sp}\colon \pi_1^{\rig,\un}((X,M);F_a,F_b)\to \pi_1^{\un}(\Gamma;\overline{a},\overline{b}).\]
is induced by the cospecialization functor $\on{sp}^*$.
\end{definition}

We may also view the specialization map as a homomorphism
\[\Pi((X,M);F_a,F_b)\to \Pi(\Gamma;\overline{a},\overline{b}).\]

The specialization map is functorial for morphisms of weak log curves: such a morphism $f\colon (X,M_X)\to (Y,N_Y)$ induces a submersion $\overline{f}\colon\Gamma_{(X,M)}\to\Gamma_{(Y,N)}$ of log dual graphs  giving a commutative diagram
\[\xymatrix{
\pi_1^{\rig,\un}((X,M);F_a,F_b)\ar[r]^>>>>>{f_*}\ar[d]^{\on{sp}}&\pi_1^{\rig,\un}((Y,N);f_*F_a,f_*F_b)\ar[d]^{\on{sp}}\\
\pi_1^{\rig,\un}(\Gamma_{(X,M)};\overline{a},\overline{b})\ar[r]^>>>>>{\overline{f}_*}&\pi_1^{\rig,\un}(\Gamma_{(Y,N)};\overline{f}(\overline{a}),\overline{f}(\overline{b})).
}\]

\subsection{The weight filtration} \label{ss:weightfiltration}

For a weak log curve $(X,M)$ and basepoints $F_a,F_b$ attached to log points, there is a weight filtration on $\pi_1^{\rig,\un}((X,M);F_a,F_b)$.
Let $(X,M')$ be the log curve obtained from $(X,M)$ obtained by replacing annular points and punctures with smooth points as in Remark~\ref{r:modifylogstructure}. Let $f\colon (X,M)\to (X,M')$ be the induced morphism. By functoriality, there is a natural map 
\[f_*\colon \Pi((X,M), F_a)\to\Pi((X,M'),F_a).\]

We set 
\[W_0=\Pi((X,M);F_a),\ W_{-1}=\mathscr{I},\  W_{-2}=\mathscr{I}^2+\on{ker}(f_*),\] 
and inductively define 
\[W_{-n}=\sum_{p+q=n, p,q>0} W_{-p}\cdot W_{-q} \text{ for } n>2.\]
This weight filtration induces the weight filtration on $\Pi((X,M),F_a)$-module $\Pi((X,M);F_a,F_b)$.
\begin{remark}
Under the identification $\mathscr{I}/\mathscr{I}^2\cong \pi_1^{\rig, \un}((X,M))^{\on{ab}}(K)=H^1_{\rig}((X,M))^\vee$ described in Section \ref{subsection:tannakian-and-fundamental-groups}, the weight filtration agrees with the usual weight filtration on $H^1_{\rig}((X,M))^\vee.$ The weight filtration on $\Pi((X,M),F_a)$ is characterized by this property together with multiplicativity. Recall that in the classical situation, the weight filtration is induced by the pole order filtration on the log de Rham complex at the points corresponding to punctures. This extends to our setting by the use of the residue map on the log de Rham complex of $]X[_P$ at punctures and annular points. This is in distinction with the monodromy-weight filtration which makes use of the pole filtration at nodes as well. The weight filtration on $H_{\rig}^1(X,M)$ is characterized by
\begin{multline*}
    W_0H_{\rig}^1((X,M))=0,\ W_1H_{\rig}^1(X,M)=\on{Im}\left(f^*\colon H_{\rig}^1((X,M'))\to H_{\rig}^1((X,M))\right),\\
    W_2H_{\rig}^1((X,M))=H_{\rig}^1((X,M))).
\end{multline*}

To justify the above description of the weight filtration on $\Pi((X,M),F_a)$, note that we have
\[W_{-3}(\mathscr{I}/\mathscr{I}^2)=0,\ W_{-2}(\mathscr{I}/\mathscr{I}^2)=(\mathscr{I}^2+\ker(f_*))/\mathscr{I}^2,\ 
W_{-1}(\mathscr{I}/\mathscr{I}^2)=\mathscr{I}/\mathscr{I}^2.\]
Using the weight filtration on $(\mathscr{I}/\mathscr{I}^2)^\vee=\on{Hom}(\mathscr{I}/\mathscr{I}^2,K)$ where $K$ is equipped with the weight filtration that is nonzero in nonnegative degrees,   we obtain
\begin{multline*}
W_{0}((\mathscr{I}/\mathscr{I}^2)^\vee)=0,\ W_{1}((\mathscr{I}/\mathscr{I}^2)^\vee)\cong \on{Im}\left(f^*\colon H_{\rig}^1((X,M'))\to H_{\rig}^1((X,M))\right),\\ 
W_{2}((\mathscr{I}/\mathscr{I}^2)^\vee)=(\mathscr{I}/\mathscr{I}^2)^\vee\cong H_{\rig}^1(X,M).
\end{multline*}
Indeed, the identification of $W_{1}((\mathscr{I}/\mathscr{I}^2)^\vee)$ is justified by recognizing it as the vector space of linear maps $\mathscr{I}/\mathscr{I}^2\to K$ vanishing on $\ker(f_*)$. This, in turn, is canonically identified with $H_{\rig}^1((X,M'))/\ker(f^*)\cong \on{Im}\left(f^*\colon H_{\rig}^1((X,M'))\to H_{\rig}^1((X,M))\right)$.
\end{remark}
\begin{remark}
The Deligne--Goncharov construction in Section~\ref{ss:delignegoncharov} yields $\Pi((X,M);a,b)/\mathscr{I}^n$ as the cohomology of a diagram of schemes; hence it admits a weight filtration. It can be shown \cite[Remark 3.10]{betts-litt} that this definition agrees with the above definition when $F_a$ and $F_b$ are attached to smooth points of $(X,M)$ by interpreting $\Pi((X,M);F_a,F_b)/\mathscr{I}^n$ as $\cE_{n-1}$ as in Remark~\ref{r:monodromyexplicit}. This definition can be extended to arbitrary log basepoints by the argument in Theorem~\ref{t:weightmonodromy}.
\end{remark}

\part{Frobenius and Monodromy on log curves}

\section{Homotopy exact sequence}

In this section, we will introduce a homotopy exact sequence analogous to the fibration exact sequence for a family over a punctured disc. This exact sequence will be used to define the Frobenius and monodromy operators on the fundamental group of a weak log curve. To describe the terms in the exact sequence, we must introduce overconvergent $F$-isocrystals. See \cite{Chiarellotto:weights} for a treatment in the non-log setting. The first step is defining the Frobenius pullback of overconvergent isocrystals in the log rigid setting.

\subsection{Definition of Frobenius pullback on overconvergent isocrystals}

Let $q=p^a$ for a positive integer $a$.
For a log scheme $(X,M)$ defined over $k=\mathbf{F}_q$, set the the absolute Frobenius $\varphi^{\abs}\colon (X,M)\to (X,M)$ to be the absolute $q$-power Frobenius on the underlying scheme together with multiplication by $q$ on $M$. 

The definition of Frobenius pullback in the log setting is more subtle than that of the non-log case where the relative Frobenius is a $k$-morphism to the Frobenius twist:
\[\varphi\colon X\to X^{(q)}=X\times_{k,\varphi^{\abs}} k\]
where the base-change is by $q$-power Frobenius $\varphi^{\abs}\colon\Spec k\to \Spec k$.
One can define the Frobenius pullback on overconvergent isocrystals in the non-log setting by the following steps: evaluate on a frame $\cP$ to obtain a vector bundle with overconvergent connection; base-change this bundle by a lift of Frobenius $\sigma\colon W(k)\to W(k)$ to obtain a bundle on $\cP^{\sigma}$ and thus an overconvergent isocrystal on $X^{(q)}$; and then pull back by relative Frobenius $\varphi\colon X\to X^{(q)}$ \cite[Section~8.3]{LS:Rigid-cohomology}. In the $p$-adic integration literature \cite{CI:Frobandmonodromy,Besser:Coleman}, one uses the isomorphism between $X$ and $X^{(q)}$ instead of twisting by $\sigma$ and obtains a {\em linear Frobenius} instead of a semi-linear Frobenius. We will follow this convention.  

In the log setting, because of the non-trivial action of $\varphi^{\abs}$ on the base $(S,\N)$, $(X^{(q)},\M^{(q)})$ can be different from $(X,\M)$, and it's not obvious how to relate their frames. We will remedy this problem by only studying log schemes $(X_t,M_t)$ arising from $\pi$-$t$ base-change. 
As will be described in Remark~\ref{r:faltingsfrob}, this idea originates in the work of Faltings \cite{Faltings:crystalline}.

For a log scheme $f\colon (X,M)\to (S,\N)$, we define the relative Frobenius $\varphi$ by the following commutative diagram
\[\xymatrix{(X,M)\ar[dr]\ar[r]^>>>>\varphi \ar@/^1.5pc/[rr]^{\varphi^{\abs}}& (X^{(q)},M^{(q)})\ar[r]^>>>>>{p_1}\ar[d]& (X,M)\ar[d]\\
\ &(S,\N)\ar[r]^{\varphi^{\abs}}&(S,\N)}\]
where $(X^{(q)},M^{(q)})=(X,M)\times_{(S,\N),\varphi^{\abs}} (S,\N)$ and $p_1$ is projection to the first factor.

We begin with the following observation:
\begin{lemma} \label{l:twisteddisc}
  The Frobenius twist $(S_t^{(q)},\N_t^{(q)})\coloneqq (S_t,\N_t)\times_{(S,\N),\varphi^{\abs}} (S,\N)$
  is isomorphic to $(S_t,\N_t)$ over $(S,\N)$. Moreover, there is a morphism
  \[h_t'\colon (S_t^{(q)},\N_t^{(q)})\to (S,\N)\]
  such that the following diagram commutes
  \[\xymatrix{
  (S_t,\N_t)\ar[dr]^{h_t}\ar[rr]^{\cong} & &(S_t^{(q)},\N_t^{(q)})\ar[dl]_{h_t'}\\
  & (S,\N)& 
  }\]
  where $h_t$ is given in Definition~\ref{d:pitbasechange}.
\end{lemma}

\begin{proof}
  Recall that $(S_t,\N_t)$ has log structure $\N_t=\N^2$ generated by $e_\pi,e_t$ with structure morphism to $(S,\N)$ given by $\N\to \N_t$ with $e_\pi\mapsto e_\pi$.
  Now, $S_t^{(q)}=\Spec k$ and $\N_t^{(q)}=\N^2$ generated by $e^{(q)}_\pi,e^{(q)}_t$. The structure morphism $(S_t^{(q)},\N_t^{(q)})\to (S,\N)$ is given by $e_\pi\mapsto e_\pi^{(q)}$ (and the relative Frobenius $(S_t,\N_t)\to (S_t^{(q)},\N_t^{(q)})$ is induced by $e_\pi^{(q)}\mapsto e_\pi$ and $e_t^{(q)}\mapsto qe_t$).
  So, $(S_t,\N_t)$ and $(S_t^{(q)},\N_t^{(q)})$ are isomorphic by $e_t\mapsto e_t^{(q)}, e_\pi\mapsto e_\pi^{(q)}$ over $(S,\N)$. The morphism $h_t'$ is induced from the homomorphism of monoids $e_\pi\mapsto e_{t}^{(q)}$. 
\end{proof}

\begin{lemma}
  Let $f\colon (X,M)\to (S,\N)$ be a morphism of log schemes defined over $\mathbf{F}_q$. Consider the $\pi$-$t$ base-change 
$(X_t,M_t)=(X,M)\times_{(S,\N),h_t} (S_t,\N_t)$
and morphism $f_t\colon (X_t,M_t)\to (S_t,\N_t)$. Then, there is a canonical $(S,\N)$-isomorphism between $(X_t,M_t)$ and its Frobenius twist. 
\end{lemma}

\begin{proof}
Note $(X_t,M_t)$ has Frobenius twist 
\begin{align*}
    (X^{(q)}_t,M^{(q)}_t) &= (X_t,M_t) \times_{{h,(S,\N)},\varphi^{\on{abs}}} (S,\N)\\
    &= ((X,M)\times_{(S,\N),h_t} (S_t,\N_t)) \times_{{h,(S,\N)},\varphi^{\on{abs}}} (S,\N)\\
    &= (X,M)\times_{(S,\N),h_t} ((S_t,\N_t)) \times_{{h,(S,\N)},\varphi^{\on{abs}}} (S,\N))\\
    &= (X,M)\times_{(S,\N),{h_t'}} (S_t^{(q)},\N_t^{(q)})\\
    &\cong (X,M)\times_{(S,\N),h_t} (S_t,\N_t)\\
    & = (X_t,M_t)
\end{align*}
where the isomorphism comes from Lemma~\ref{l:twisteddisc}.
\end{proof}

Putting everything together, we interpret $\varphi$ to be the composition given by the top row in the following diagram:
  \[\xymatrix{(X_t,M_t)\ar[r]&(X_t^{(q)},M_t^{(q)})\ar[d]^{f_t^{(q)}}\ar[rr]^{\cong} & &(X_t,M_t)\ar[d]^{f_t}\\
    &(S_t^{(q)},\N_t^{(q)})\ar[dr]
  \ar[rr]^{\cong} & &(S_t,\N_t)\ar[dl]\\
  & &(S,\N)& 
  }\]

\begin{remark} \label{r:frobliftinterp}
  This approach can be understood in terms of lifts.  Let $(\cQ,N_{\cQ})$ be the {\em log disc} where $\cQ=\Spf V\ps{t}$ is given the log structure $N_{\cQ}$ induced by $\Z_{\geq 0}^2$ with $e_t\mapsto t$, $e_\pi\mapsto \pi$. Let $(\cP,L)$ be a formal scheme over $(\cQ,N_{\cQ})$ whose closed fiber is $(X_t,M_t)$.
  Then $(\cP,L)$ is also a lift of $(X_t^{(q)},M_t^{(q)})$. 
  A lift of relative Frobenius (which may only exist locally) is a morphism 
  \[(\cP,L)\to (\cP,L)\]
  extending $\varphi$. In particular, this extends the map $(\cQ,N_{\cQ})\to(\cQ,N_{\cQ})$ given by $t\mapsto t^q$.
  Because the map fixes the origin of $\cQ$, it will be natural to take fibers over points above the origin. 
\end{remark}

\begin{definition}
An {\em $F$-isocrystal} $\cV$ on $(X_t,M_t)$ is an  isocrystal equipped with an $F$-structure, i.e.,~an isomorphism $\cV\cong \varphi^*\cV$. A morphism of  $F$-isocrystals is a morphism of the underlying  isocrystals intertwining $F$-structures. Let  $\sC^{\rig,\varphi}(X_t,M_t)$ denote the category of $F$-isocrystals on $(X_t,M_t)$.

For a subcategory of $\sC^{\rig}(S_t,\N_t)$, let the superscript ``$\nr$'' denote the full subcategory of isocrystals having nilpotent residues when evaluated on their respective frames. 

For a morphism $g\colon (X,M_X)\to (Z,M_Z)$, the category of {\em relatively unipotent $F$-isocrystals}, $\sC^{\rig,\varphi,\un(g)}(X,M_X)$ is the full subcategory of $\sC^{\rig,\varphi}(X,M_X)$ whose objects have a filtration (in $\sC^{\rig,\varphi}(X,M_X)$) whose associated gradeds are isomorphic to  $F$-isocrystals of the form $g^*\cV$ where $\cV$ is an $F$-isocrystal on $(Z,M_Z)$. 
For a morphism $f_t\colon (X_t,M_t)\to (S_t,\N_t)$, write the superscript ``$\un(f_t,\nr)$'' for the full subcategory of relatively unipotent isocrystals whose associated gradeds are isomorphic to $g^*\cV$ where $\cV$ has nilpotent residues on $(S_t,\N_t)$. 
\end{definition}

\begin{lemma} Let $f\colon (X,M)\to (S,\N)$ be a weak log curve with $\pi$-$t$ base-change $f_t\colon (X_t,M_t)\to (S,\N_t)$. Then any object of  $\sC^{\rig,\varphi,\un(f_t,\nr)}(X_t,M_t)$ is unipotent as an object of $\sC^{\rig}(X_t,M_t)$.
\end{lemma}

\begin{proof}
It suffices to show the unipotency of an object $f_t^*\cV$ of $\sC^{\rig,\varphi,\un(f_t,\nr)}(X_t,M_t)$ for $\cV\in \on{Obj}\sC^{\rig,\varphi}(S_t,\N_t)$. The $F$-isocrystal $\cV$ on $(S_t,\N_t)$ has nilpotent residues and 
as observed in \cite[Lemma~4.2.3]{Kedlaya:aws}, Dwork's trick in combination with \cite[Lemma~3.6.2]{Kedlaya:unipotence} shows that its evaluation on the log disc is unipotent. 
\end{proof}

\subsection{Homotopy exact sequence}

Let $\pi_1^{\rig,\varphi,\un(f_t,\nr)}((X_t,M_t),F)$ denote the fundamental group attached to $\sC^{\rig,\varphi,\un(f_t,\nr)}(X_t,M_t)$. Then, the forgetful
functor composed with restriction induced by $(X,M)\to (X_t,M_t)$,
\[\sC^{\rig,\varphi,\un(f_t,\nr)}(X_t,M_t)\to \sC^{\rig,\un}(X_t,M_t)\to  \sC^{\rig,\un}(X,M)\] 
induces morphisms of fundamental torsors
\[\pi_1^{\rig,\un}((X,M);F_1,F_2)\to \pi_1^{\rig,\un}((X_t,M_t);F_1,F_2)\to \pi_1^{\rig,\varphi,\un(f_t,\nr)}((X_t,M_t);F_1,F_2).\] 
We will understand the image of this composition below.

\begin{proposition} \label{p:homotopyexactsequence} Let $f\colon (X,M)\to (S,\N)$ be a weak log curve. Let $(X_t,M_t)$ be the $\pi$-$t$ base-change which induces $i\colon (X,M)\to (X_t,M_t)$ and $f_t\colon (X_t,M_t)\to (S_t,\N_t)$.
Let $F$ be a fiber functor on $(X,M)$ (which we will identify with its image on $(X_t,M_t)$ and $(S_t,\N_t)$). Then there is an exact sequence of fundamental groups
\[\xymatrix{0\ar[r]&\pi^{\rig,\un}_1((X,M),F)\ar[r]&\pi^{\rig,\varphi,\un(f_t,\nr)}_1((X_t,M_t),F)\ar[r]&\pi_1^{\rig,\varphi,\nr}((S_t,N_t),F)\ar[r]&0}.\]
\end{proposition}

\begin{proof}
This theorem is the log version of the main results of \cite[Section~3]{Lazda:rational} specialized to the situation where the base is the log disc, combined with some ideas from  \cite{CPS:Logpi1} which treats the de Rham case. It amounts to verifying the exactness criterion in \cite[Appendix~A]{esnault-hai-sun}. 
We explain how to adapt the above results to our setting.
Let $(\cQ,N_{\cQ})$ be the log disc as in Remark~\ref{r:frobliftinterp}.
Then $((S_t,\N_t),(S_t,\N_t),(\cQ,N_{\cQ}))$ is a proper log smooth frame. By deformation theory, we can 
pick a proper formal log smooth curve $(\cP,L)$ over
$(\cQ,N_{\cQ})$ around 
$(X_t,M_t)$. 
There are proper log smooth frames $((X_t,M_t),(X_t,M_t),(\cP,L))$ and 
$((X,M),(X,M),(\cP_\pi,L_\pi))$
where the subscript $\pi$ denotes the fiber over $t=\pi$.

Now, \cite[Section~3]{Lazda:rational} constructs the exact sequence in the non-log rigid setting using overconvergent $F$-isocrystals. 
The arguments hold in our situation once one has the existence of higher pushforwards of convergent $F$-isocrystals which is a consequence of the comparison theorem between log convergent and log de Rham cohomology  \cite[Corollary~2.34]{Shiho:relative}.
Then, one may employ a Katz--Oda style argument as in \cite{CPS:Logpi1} to restrict to connections with nilpotent residues.
\end{proof}

\begin{remark} \label{r:universalbundle}
The crux of the above proof (following \cite[Section~3]{Lazda:rational}) is the construction of the relative analogue of the universal object of index of unipotency $n$, $W_n$ on $]X_t[_{\cP}$. It is a relatively unipotent vector bundle with integrable connection with nilpotent residues that occurs as the evaluation of an $F$-isocrystal. It restricts to the universal object $E_n$ on $]X[_{\cP_\pi}$, which is the fiber over $t=\pi$. Moreover, given a section $\sigma\colon (S_t,\N_t)\to (X_t,M_t)$ lifting to $\tilde{\sigma}\colon ]S_t[_{\cQ}\to ]X_t[_{\cP}$, we further constrain $W_n$ to ensure that $\tilde{\sigma}^*W_n\to \tilde{\sigma}^*W_0=\mathbf{1}$ has a splitting. 
This splitting is the {\em base section}. 
This construction is universal in the sense that given a relatively unipotent bundle $V$ with connection of index of unipotency at most $n$ and morphism over $]S_t[_{\cQ}$, $s\colon\mathbf{1}\to \tilde{\sigma}^*V$, there is a unique morphism $W_n\to V$ taking the base section to $s$.
The construction of $W_n$ is analogous to that of $E_n$.  
Specifically, it fits into an exact sequence
\[\xymatrix{
0\ar[r]&U_n\ar[r]&W_{n+1}\ar[r]&W_n\ar[r]&0.}
\]
where $U_n=f_t^*(R^1f_{t*}W_n^\vee)^\vee$. 
By induction, $W_{n+1}$ is the evaluation of an $F$-isocrystal. 
\end{remark}

For log points, the homotopy exact sequence can be verified directly in the unipotent case.

\begin{lemma}\label{l:fundamental-exact-sequence}
Let $f\colon(x,\overline{\M})\to (S,\N)$ be a log point where $\overline{M}$ is $\overline{M}_0$, $\overline{M}_1$, or $\overline{M}_2$. Let $B$ be coordinate system on $\overline{M}$. Then, there is an exact sequence
\[\xymatrix{0\ar[r]&\pi^{\rig,\un}_1((x,\overline{\M}),F_B)\ar[r]&\pi^{\rig,\un}_1((x,\overline{\M}_t),F_{B})\ar[r]&\pi^{\rig,\un}_1((S_t,\N_t),F_B)\ar[r]&0}.\]
\end{lemma}

\begin{proof}
There is a sequence of morphisms of monoids
\[\xymatrix{
\overline{M}&\overline{M}_t\ar[l]&(\Z_{\geq 0})^2\ar[l]
}\]
where $(\Z_{\geq 0})^2\to\overline{M}_t$ is given by including $e_t$ and $e_\pi$, and $\overline{M}_t\to \overline{M}$ is given by identifying $e_\pi$ and $e_t$. This induces an exact sequence of abelian groups
\[\xymatrix{
0&(\overline{M}/f^*\Z_{\geq 0})^{\gp}\ar[l]
&(\overline{M}_t/f^*\Z_{\geq 0})^{\gp}\ar[l]&((\Z_{\geq 0})^2/\Z_{\geq 0})^{\gp}\ar[l]&0\ar[l]
 }\]
The conclusion follows from the functorial identification in Lemma~\ref{l:pi1logpoint}. 
\end{proof}


We now consider the homotopy exact sequence for fundamental torsors.
Let $1\to H\to G\to L\to 1$ be an exact sequence of (pro-)algebraic groups over a field,  $T_1$ a left $H$-torsor, $T_2$ a left $G$-torsor, $f\colon T_1\to T_2$ an $H$-equivariant map (for the induced left $H$-action on $T_2$), and $g: T_2\to L$ a $G$-equivariant map (for the induced left $G$-action on $L$). Then we have the following: 

\begin{lemma}\label{l:torsorfact}
The map $f$ is a closed embedding and the map $g$ exhibits $L$ as the quotient of $T_2$ by $H$.
\end{lemma}
\begin{proof}
The statement may be checked after replacing the field with a finite extension, so we may assume the torsors in question have points and are hence trivial. Now, the first statement follows from the fact that $H \to G$ is a closed embedding, and the second follows from the isomorphism $L=G/H$.
\end{proof}
In this situation we refer to the sequence 
\[T_1\overset{f}{\longrightarrow}T_2 \overset{g}\longrightarrow L\]
as an \emph{exact sequence of torsors}.
From the homotopy exact sequence for fundamental groups and the above lemma one obtains the following:

\begin{corollary} \label{c:exactsequenceoftorsors}
Let $f\colon (X,M)\to (S,\N)$ be a weak log curve. Let $F_1,F_2$ be fiber functors on $(X,M)$ whose image under $f_*$ are equal. Then the sequence 
\[
\xymatrix{\pi^{\rig,\un}_1((X,M);F_1,F_2)\ar[r]^>>>>>{i_*}&\pi_1^{\rig,\varphi,\un(f_t,\nr)}((X_t,M_t);F_1,F_2)\ar[r]^{f_{t*}}&\pi_1^{\rig,\varphi,\nr}((S_t,
\N_t);F_1,F_2)}\]
is an exact sequence of torsors. The torsor analogue of Lemma~\ref{l:fundamental-exact-sequence} holds as well.
\end{corollary}


\section{Frobenius action on fundamental groups} \label{ss:frobaction}

In this section, we define the Frobenius action on the log rigid fundamental torsor and study Frobenius-invariant paths. As noted in the previous section, the main obstacle is the difficulty in comparing $(X,M)$ and $(X^{(q)},M^{(q)})$ .
We remedy this problem by relating the fundamental group of $(X,M)$ to that of $(X_t,M_t)$ following Lazda \cite[Section~3.2]{Lazda:rational}. By the homotopy exact sequence,
\[\pi^{\rig,\un}_1((X,M);F_1,F_2)\cong \ker\left(\pi_1^{\rig,\varphi,\un(f_t,\nr)}((X_t,M_t);F_1,F_2)\to\pi_1^{\rig,\varphi,\nr}((S_t,
\N_t);F_1,F_2)\right)\]
which does carry a Frobenius.
To obtain Frobenius-invariant basepoints, we will need to make use of  tangential basepoints (Remark~\ref{r:tangentialbasepoint}) and specialization paths (Definition~\ref{d:specializationpath}). The latter necessitates a choice of $\Log(\pi)$, and we take an indeterminate $\ell$, tensoring the fundamental torsor with $K[\ell]$ (which will be denoted by the subscript ``$\ell$'').

\subsection{Definition of the Frobenius operator}

Suppose that $(X,M)$ is defined over $k=\mathbf{F}_q$ for $q=p^a$. Consequently, we obtain a canonical isomorphism identifying $(X_t,M_t)$ and $(X_t^{(q)},M_t^{(q)})$. The relative Frobenius is a morphism
\[\varphi\colon (X_t,M_t)\to (X_t^{(q)},M_t^{(q)}).\]
To define the Frobenius on the fundamental group, we  need to introduce Frobenius invariant basepoints on $(X_t,M_t)$.
A log point $(x,\overline{\M})\to (X,M)$ defined over $k$ induces a log point $(x,\overline{\M}_t)\to (X_t,M_t)$ Given a fiber functor $F_B$ on $(x,\overline{\M})$, we may define $F_{B\cup \{e_t\}}$ on $(x,\overline{\M}_t)$. 
This, in turn, induces a tangential basepoint $\tilde{F}$ on $(X_t,M_t)$ as in Remark~\ref{r:tangentialbasepoint}. 
By employing the specialization and tangential paths, we obtain isomorphisms of fiber functors on $(x,\overline{\M}_t)$
\[\xymatrix{
F_{B,\pi}\otimes K[\ell]\ar[r]^>>>>>\cong & F_{B\cup \{e_t\}}\otimes K[\ell]\ar[r]^>>>>>\cong &\tilde{F}\otimes K[\ell]. }\]

The tangential basepoint is fixed by Frobenius in the following sense:
\begin{lemma} \label{l:canonicalfrobfiberfunctor}
  There is a canonical isomorphism $\varphi_*\tilde{F}\cong \tilde{F}$.
\end{lemma}

\begin{proof}
  By functoriality of Frobenius, it suffices to consider $\tilde{F}$ as a fiber functor on $(x,\overline{\M}_t)$. Let 
  \[((x,\overline{\M}_t),(x,\overline{\M}_t),(\cP_t,L_t))\]
  be the frame produced from $(x,\overline{\M}_t)$ as in Section~\ref{ss:logpoints}. 
  Write $(x,\overline{\M}_t^{(q)})$ for the Frobenius twist of $(x,\overline{\M}_t)$.
  Now, $\overline{M}_t=(q\circ f_t)^*{\Z_{\geq 0}}\oplus \overline{M}$, and so $\cP_t\cong \Spf V\ps{\overline{M}}.$
  Consequently, the $q$-power map on the coordinates of $\Spf V\ps{\overline{M}}$ is a lift $\varphi_{\cP}$ of relative Frobenius that gives a morphism of frames:
  \[\xymatrix{
  (x,\overline{\M}_t)\ar[r]\ar[d]_{\varphi}&(x,\overline{\M}_t)\ar[r]\ar[d]_{\varphi}&  (\cP_t,L_t)\ar[d]_{\varphi_{\cP}}\\
  (x,\overline{\M}^{(q)}_t)\ar[r]&(x,\overline{\M}^{(q)}_t)\ar[r]&  (\cP_t,L_t).
  }\]
  A unipotent isocrystal $\cE$ on $(X_t,M_t)$ induces a unipotent vector bundle with integrable log connection on 
  \[]x[_{\cP_t}\cong ]0[_{\Spf V\ps{\overline{M}}}.\]
  There is a canonical isomorphism between $F_{B\cup \{e_t\}}(\cE)$ and $\varphi_*F_{B\cup \{e_t\}}(\cE)=F_{B\cup\{e_t\}}(\varphi^*\cE)$ given by pullback by $\varphi_{\cP}^*\colon s\mapsto s\circ\varphi_{\cP}$. Because $t=0$ is a fixed point of the $\varphi_{\cP}$, $\tilde{F}$ is fixed by $\varphi_*$.
 \end{proof}


\begin{definition}
Let $F_1,F_2$ be fiber functor attached to log points on $(X_t,M_t)$. Use the homotopy exact sequence and the specialization and tangential paths to produce isomorphisms
\begin{align*}
    \pi^{\rig,\un}_1((X,M);F_1,F_2)_{\ell}&\cong \ker\left(\pi_1^{\rig,\varphi,\un(f_t,\nr)}((X_t,M_t);F_1,F_2)_{\ell}\to\pi_1^{\rig,\varphi,\nr}(S_t,
\N_t);F_1,F_2)_{\ell}\right)\\
    &\cong \ker\left(\pi_1^{\rig,\varphi,\un(f_t,\nr)}((X_t,M_t);\tilde{F}_1,\tilde{F}_2)_{\ell}\to\pi_1^{\rig,\varphi,\nr}(S_t,
\N_t);\tilde{F}_1,\tilde{F}_2)_{\ell}\right)\\
\end{align*}
By identifying $(X_t^{(q)},M_t^{(q)})$ and $(X_t,M_t)$ as well as $(S_t,\N_t)$ and $(S_t^{(q)},\N_t^{(q)})$, relative Frobenius $\varphi_*$ induces an action on the kernel above. We define the action on $\ell$ to be $\varphi_*(\ell)=q\ell$,
and obtain a homomorphism
\[\varphi_*\colon \pi^{\rig,\un}_1((X,M);F_1,F_2)_{\ell}\to \pi^{\rig,\un}_1((X,M);F_1,F_2)_{\ell},\] semilinear over the action of $\varphi_*$ on $k[\ell]$. 
\end{definition}

\begin{remark} \label{r:froblogpoint}
We can justify the definition $\varphi_*\ell=q\ell$ by considering the Frobenius action on the log point $(x,\overline{\M})$ with $\pi$-$t$ base-change $(x,\overline{\M}_t)$. A coordinate system $B$ on $\overline{M}$ induces a coordinate system $B\cup \{e_t\}$ on $\overline{M}_t$.

Let $\cE_n$ be the universal bundle on $(x,\overline{\M}_{2,t})$ as described in Lemma~\ref{l:pi1logpoint}: set $U=K^2$; so $\cE_n$ is the trivial bundle attached to $\on{Sym}^{\leq n} U^\vee$ -- which can be interpreted as $K[f_1^\vee,e_t^\vee]/(f_1^\vee,e_t^\vee)^{n+1}$ for indeterminates $f_1^\vee,e_t^\vee$ --
with connection $1$-form
\[\omega=\frac{dx_1}{x_1}\otimes f_1^\vee+\frac{dt}{t}\otimes e^\vee_t.\]
The elements of $F_{f_1,e_t}(\cE_n)\otimes K[\ell]$ are of the form 
\[y\exp\left(\Log(x_i)f_1^\vee+\Log(t)e_t^\vee\right)\]
for $y\in \on{Sym}^{\leq n} U^\vee\otimes K[\ell]$. Under the tangential path such a section  is mapped to $y\in \tilde{F}(\cE_n)\otimes K[\ell]$ while the specialization path is the following isomorphism:  
\[y\exp\left(\Log(x_1)f_1^\vee+\Log(t)e_t^\vee\right)\mapsto
y\exp\left(\ell e_t^\vee\right)\exp\left(\Log(x_1)f_1^\vee\right).\]
Let $\cE_n^{(q)}$ (which can be identified with $\cE_n$) be the universal bundle on $(x,\overline{\M}_t^{(q)})$.
Then, $\varphi^*\cE_n^{(q)}$ has the same underlying bundle as $\cE_n$ but has connection form given by $q\omega$.  The inverse image of $y$ under the tangential path on $\varphi^*\cE_n^{(q)}$, is mapped as follows by the specialization path:
\[y\exp\left(q\Log(x_1)f_1^\vee+q\Log(t)e_t^\vee\right)\mapsto y\exp\left(q\ell e_t^\vee\right) \exp\left(q\Log(x_1)f_1^\vee\right).\]
Therefore, by conjugation with the tangential and specialization paths, the isomorphism $\varphi_*\tilde{F}\to \tilde{F}$ induces an isomorphism
\[\varphi_*F_{\{f_1\},\pi}\otimes K[\ell]\to F_{\{f_1\},\pi}\otimes K[\ell]\]
that takes $\ell\mapsto q\ell$.
\end{remark}

We now compute the action of Frobenius on the fundamental torsor of a log point.

\begin{proposition} \label{p:nodalfrobenius}
Let $(x,\overline{\M}_2)$ be a node. For $u\in(\Ga)(K)$, the path 
\[\delta_u\in \pi_1((x,\overline{\M}_2);F_{\{f_1\},\ell},F_{\{f_2\},\ell})_{\ell}\] 
from Example~\ref{e:pi1annulus} obeys $\varphi_*\delta_u=\delta_{qu}$.
Consequently $\delta_0$ is the unique Frobenius-invariant path.
\end{proposition}

\begin{proof}
We use the notation from Lemma~\ref{l:pi1logpoint} and Remark~\ref{r:froblogpoint}. By conjugating with the specialization path, $\delta_u$ induces an element 
\[\delta_u^t\in\pi_1^{\rig,\un}((X_t,\M_{2,t});F_{\{f_1,e_t\}},F_{\{f_2,e_t\}})_{\ell}.\]
that acts on $F_{\{f_1,e_t\}}(\cE_n)$ by
\[    y\exp\left(\Log(x_1)e_1^\vee+\Log(t)e_t^\vee\right)\\
    \mapsto 
\left(y\exp\left(ue_1^\vee\right)\right)\exp\left((\Log(t)-\Log(x_2))e_1^\vee+\Log(t)e_t^\vee\right)
\]
and acts on $F_{\{f_1,e_t\}}(\varphi^*\cE_n)$ by
\begin{multline*}
    y\exp\left(q\Log(x_1)e_1^\vee+q\Log(t)e_t^\vee\right)\\
    \mapsto 
\left(y\exp\left(que_1^\vee\right)\right)\exp\left((q\Log(t)-q\Log(x_2))e_1^\vee+q\Log(t)e_t^\vee\right)
\end{multline*}
It follows that $\varphi_*\delta_u^t=\delta_{qu}^t$. 
\end{proof}

\begin{remark} \label{r:faltingsfrob}
The use of the fiber functor $\tilde{F}$ is analogous to Faltings's construction of the Frobenius on log crystalline cohomology  \cite{Faltings:crystalline} (which was translated into the log rigid setting in \cite[Section~1.11]{GK:Frobandmonodromy})
Here, the cohomology of $(X,M)$ is defined to be the fiber over the origin of $R^if_{t*}\mathbf{1}$ for $f_t\colon (X_t,M_t)\to (S_t,\N_t)$. By the analogue of Katz--Oda (and arguments from the proof of Proposition~\ref{p:homotopyexactsequence}), $R^if_{t*}\mathbf{1}$ is a unipotent isocrystal. One evaluates $R^1f_{t*}\mathbf{1}$ on the frame $((S_t,\N_t),(S_t,\N_t),(\cQ,L_\cQ))$. The tube $]S_t[_{\cQ}$ is an open disc, and one can take the fiber over its origin. 

The relative Frobenius induces a morphism 
\[\varphi^*R^if^{(q)}_{t*}\mathbf{1}\to R^if_{t*}\mathbf{1}.\]
We can identify $H_{\rig}^i((X,M))$ with the fibers over the origin of these local systems (by use of the specialization and tangential paths after we tensor with $K[\ell]$). Because the origin is fixed by Frobenius, this gives a Frobenius action on log rigid cohomology.

In greater degree of generality, consider the case where $\cW$ is a relatively unipotent $F$-isocrystal on $(X_t,M_t)$ and $\cE=i^*\cW$ is a unipotent  isocrystal on $(X,M)$ (where $i$ is defined in Section~\ref{ss:logcurvesbackground}). Then, the relative Frobenius induces a map
\[\varphi^*R^if^{(q)}_{t*}\cW\to R^if_{t*}\cW.\]
By taking the fiber over the origin, we obtain the action of the Frobenius $\varphi^*$ on $H_{\rig}^*((X,M),\cE)$.
\end{remark}

\begin{remark}
We can unpack the definition of the Frobenius action in light of the universal object construction of the fundamental group at least in the special case where $F$ is a fiber functor attached to a smooth $k$-point $y$ of $(X,M)$. 
Let $\sigma\colon (S_t,\N_t)\to (X_t,M_t)$ be a section with $\sigma(S_t)=y$. 
Let $\cW_n$ be the universal object  on $(X_t,M_t)$ as constructed in Remark~\ref{r:universalbundle} where we pick a base section over $\sigma$.
Under the identification of $(X_t,M_t)$ and $(X_t^{(q)},M_t^{(q)})$,  $\cW_n$ is the relative universal object over 
 \[f_t^{(q)}\colon (X^{(q)}_t,M^{(q)}_t)\to (S^{(q)}_t,\N^{(q)}_t)\cong (S_t,\N_t)\]
(with the induced base section at $\sigma^{(q)}\colon (S_t^{(q)},\N_t^{(q)})\to (X^{(q)}_t,M^{(q)}_t)$).
We know that $i^*\cW_n=\cE_n$. 
The map 
\[\pi_1^{\rig,\varphi,\un(f_t,\nr)}((X_t,M_t),\tilde{F})\to 
\pi_1^{\rig,\varphi,\un(f_t,\nr)}((X_t^{(q)},M_t^{(q)}),\varphi_{*}\tilde{F})\] 
is induced by
\[\tilde{F}(\Phi_n)\colon \tilde{F}(\cW_n)\to \tilde{F}(\varphi^*\cW_n)=(\varphi_{*}\tilde{F})(\cW_n).\]
where $\Phi_n\colon \cW_n\to \varphi^{*}\cW_n$ be the unique map preserving base sections guaranteed by universality.
Lemma~\ref{l:canonicalfrobfiberfunctor} gives an isomorphism $(\varphi_{*}\tilde{F})(\cW_n)\to \tilde{F}(\cW_n)$ which, when composed with $\tilde{F}(\Phi_n)$,
gives an endomorphism of $\tilde{F}(\cW_n)$ corresponding to the Frobenius action on $\pi_1^{\rig,\varphi,\un(f_t,\nr)}((X_t,M_t),\tilde{F})$.
\end{remark}

\begin{remark} \label{r:dgfrob}
The Deligne--Goncharov construction (Section~\ref{ss:delignegoncharov}) explicitly realizes the fiber of $\cE_n$ over $x_2$, which is $\Pi((X,M); F_1,F_2)/\mathscr{I}^{n+1})$
as
$\mathbb{H}^{n}_\bullet(P_{(X,M), x_1,x_2}^{n+1}, \mathbf{1}_\bullet)^\vee$.
 The Frobenius operator is induced by considering an extension of  $P_{(X,M),x_1,x_2}^{n+1}$ to $(X_t,M_t)$. The computation of the hypercohomology by means of higher pushforwards as in \cite[Proposition~3.4]{DG:groupes} constructs $\cE^\vee_n$ and, implicitly, 
 $\cW^\vee_n$ on $(X_t,M_t)$ together with the base section over $x_1$ by reinterpreting rigid cohomology as in Remark~\ref{r:faltingsfrob}.
The relative Frobenius $\varphi$ induces by functoriality, as explained in Remark~\ref{r:delignegoncharovobject}, a morphism $\varphi^*\cW_n^\vee\to \cW_n^\vee$ 
 whose dual is $\Phi_n$. Thus, one can identify the Frobenius operator on the fundamental torsor defined through the Deligne--Goncharov construction with the Frobenius defined above.
\end{remark}

\subsection{Frobenius-invariant paths}

We will study Frobenius-invariant paths on $(X,M)$ using the specialization map to $\Gamma$, the log dual graph of $(X,M)$. We define \[\varphi_*\colon \pi_1^{\un}(\Gamma,\overline{a},\overline{b})_{\ell}\to \pi_1^{\un}(\Gamma,\overline{a},\overline{b})_{\ell}\]
to be the identity on $\pi_1^{\un}(\Gamma,\overline{a},\overline{b})$ and to satisfy $\varphi_*(\ell)=q\ell.$ We may factor the cospecialization map (Subsection~\ref{ss:specializationmap}) as
\[\on{sp}^*\colon \sC^{\un}(\Gamma)\to \sC^{\rig,\varphi,\un(f_t,\nr)}(X_t,M_t)\to\sC^{\rig,\un}(X,M)\]
where the first arrow is defined by descent similarly to the usual cospecialization map and the second arrow is the map $i^*$ from Section~\ref{ss:logcurvesbackground}. Here, the unipotent isocrystal on $(X_t,M_t)$ arising from cospecialization is canonically isomorphic to its Frobenius pullback, making $\on{sp}_*$ into a Frobenius-equivariant homomorphism. Indeed, objects in the image of $\on{sp}^*$ are obtained by gluing together $(\cO_{]X_{\overline{v}}[_{\cP}}^{\oplus N},d)$ on components by constant linear maps. Because we use the linear Frobenius, such objects can be canonically identified with their Frobenius pullbacks on $(X_t,M_t)$.

The following, the main result of this section, is a generalization to the semistable case of the approach of Besser \cite{Besser:Coleman} which identifies Frobenius-invariant paths on smooth curves:

\begin{proposition}\label{prop:Frob-invariants}
Let $(X,M)$ be a weak log curve with log dual graph $\Gamma$. Let $F_a,F_b$ be fiber functors on $\sC^{\rig,\un}(X,M)$ induced by log points anchored at $\overline{a},\overline{b}\in V(\Gamma)$, respectively. Then, the specialization  map yields an isomorphism
\[\on{sp}\colon \pi_1^{\rig,\un}((X,M);F_a,F_b)_{\ell}^{\varphi}\cong \pi_1^{\un}(\Gamma;\overline{a},\overline{b})_\ell^\varphi.\]
\end{proposition}

Because $\pi^{\un}_1(\Gamma;\overline{a},\overline{b})_\ell^\varphi\cong \pi^{\un}_1(\Gamma;\overline{a},\overline{b})$, we will be able to conclude
\[\pi_1^{\rig,\un}((X,M);F_a,F_b)_{\ell}^{\varphi}\cong \pi_1^{\un}(\Gamma;\overline{a},\overline{b}).\]

\begin{lemma}\label{l:free-prounipotent-iso}
Let $G, H$ be pro-unipotent groups over a field $K$ of characteristic zero, with $H$ free. Then if $f\colon G\to H$ is a map inducing an isomorphism $G^{\on{ab}}\to H^{\on{ab}}$, $f$ is an isomorphism.
\end{lemma}

\begin{proof}
This is immediate from \cite[Corollary~2.6]{lubotsky1982cohomology}, as free pro-unipotent groups have the so-called ``lifting property," (see \cite[p.~85]{lubotsky1982cohomology}). 
\end{proof}

The lemmas below are stated for torsors of paths between fiber functors $F_a$ and $F_b$. It suffices to restrict to the case where $F_a=F_b$ as a map between torsors for a given group is necessarily an isomorphism.

\begin{lemma}\label{lem:specialization-abelianization}
The cospecialization map $\on{sp}^*$ induces an isomorphism on abelianizations
\[\on{sp}\colon (\pi_1^{\rig,\un}((X,M);F_a,F_b)_{\ell}^{\varphi})^{ \on{ab}}\cong \pi_1^{\un}(\Gamma;\overline{a},\overline{b})^{\on{ab}}.\]
\end{lemma}
\begin{proof}
It suffices to consider the case $F_a=F_b$.
We first show that the natural map \[\on{Hom}(\pi_1^{\un}(\Gamma,\overline{a}), \mathbb{G}_a)\to \on{Hom}(\pi_1^{\rig,\un}((X,M),F_a)_{\ell}, \mathbb{G}_a)^{\varphi}\] 
induced by $\on{sp}$ is an isomorphism. 
We have canonical identifications \[\on{Hom}(\pi_1^{\un}(\Gamma,\overline{a}), \mathbb{G}_a)=H^1(\Gamma)\]
and 
\[\on{Hom}(\pi_1^{\rig,\un}((X,M),F_a)_{\ell}, \mathbb{G}_a)^\varphi=H_{\rig}^1((X,M))_{\ell}^\varphi\]
(as both sides of the first line classify extensions of the trivial representation of $\pi_1^{\un}(\Gamma,\overline{a})$ by itself, while the second line comes from taking $\varphi$-invariants of the analogous statement for $\pi_1^{\rig,\un}((X,M),F_a)$), so this is equivalently a map 
\[H^1(\Gamma)\to H_{\rig}^1((X, M))^\varphi.\]

One may describe this map explicitly as follows: an element of $H^1(\Gamma)$ corresponds to an extension 
\[0\to \mathbf{1}\to E\to \mathbf{1}\to 0\] 
of $\pi_1(\Gamma,\overline{a})$-representations; its image under the cospecialization map corresponds to a unipotent isocrystal of rank $2$ on $(X,M)$. Explicitly, one takes the gluing datum defining $E$ and turns it into an isocrystal on $(X,M)$ via the recipe of Section \ref{ss:specializationmap}. By the discussion in Subsection~\ref{ss:cohomologyoflogcurves}, this corresponds to Cech cocycles representing an element of $H^1(\Gamma)\cong M_0H_{\rig}^1((X,M))$. The latter group is isomorphic to $H_{\rig}^1((X,M))^\varphi$ as a consequence of the degeneration of the Steenbrink--Zucker spectral sequence from subsection~\ref{ss:cohomologyoflogcurves} at $E_2$. See \cite[Section~5]{Mokrane} for a relevant discussion.

As a consequence of the above, we have that $\pi_1^{\un}(\Gamma, \overline{a})^{\on{ab}}\cong(\pi_1^{\rig,{\un}}((X,M), F_a)^{\on{ab}})_{\ell, \varphi}$. 
That is, we have identified the abelianization of the fundamental group of our graph with the $\varphi$-coinvariants of $\pi_1^{\rig, {\un}}((X,M), F_a)^{\on{ab}}$. 
As $\varphi$ acts semisimply on  $\pi_1^{\rig, \un}((X,M), F_a)^{\on{ab}}$, the coinvariants agree with the invariants. 

We conclude by showing
\[ (\pi_1^{\rig,{\un}}((X,M), F_a)_{\ell}^{\on{ab}})^\varphi\cong 
(\pi_1^{\rig,{\un}}((X,M), F_a)_{\ell}^\varphi)^{\on{ab}}.\]
We compute at the level of Lie algebras. Give the Lie algebra $\mathfrak{g}_a$ of $\pi_1^{\rig, \un}((X,M), F_a)$ the lower central series filtration. Because $\varphi$ acts semisimply on $\mathfrak{g}_a$ \cite[Theorem 1.4]{betts-litt}, there is a  $\varphi$-equivariant isomorphism $\mathfrak{g}_a\cong \on{gr}_\bullet\mathfrak{g}_a$.
Because the Lie algebra $\on{gr}_\bullet\mathfrak{g}_a$ is graded by $\mathbb{Z}_{<0}$, and generated in degree $-1$, the $\varphi$-weights appearing are non-positive. Hence, $\on{gr}_\bullet\mathfrak{g}_a^{\varphi}$ is generated by $\on{gr}_{-1}\mathfrak{g}_a^{\varphi}$.  Indeed, any element of $\on{gr}_\bullet\mathfrak{g}_a^{\varphi}$ can be written as a linear combination of brackets of elements of $\on{gr}_{-1}\mathfrak{g}_a$, but for weight reasons, every element appearing in the linear combination must have $\varphi$-weight zero.

We now claim that $(\on{gr}_\bullet\mathfrak{g}_a^{\varphi})^{\on{ab}}=(\mathfrak{g}_a^{\on{ab}})^{\varphi}$.
Indeed, the left side is computed as 
\[(\on{gr}_\bullet\mathfrak{g}_a^{\varphi})^\text{ab}=\on{coker}([-,-]: \bigwedge^2 \mathfrak{g}_a^{\varphi}\to \mathfrak{g}_a^{\varphi});\]
and by the above,  the commutator map surjects onto $\on{gr}_{\leq -2} \mathfrak{g}_a^{\varphi}$. Consequently, the cokernel is canonically identified with $(\on{gr}_{-1} \mathfrak{g}_a)^{\varphi}=(\mathfrak{g}_a^{\on{ab}})^{\varphi}$.
\end{proof}

\begin{proof}[Proof of Proposition \ref{prop:Frob-invariants}]
We suppose $F_a=F_b$. The specialization map becomes an isomorphism after abelianization by Lemma \ref{lem:specialization-abelianization}. 
Because $\pi_1^{\un}(\Gamma;\overline{a},\overline{b})_\ell^\varphi$ is free pro-unipotent, the conclusion follows from Lemma~\ref{l:free-prounipotent-iso}.
\end{proof}

The isomorphism of Proposition~\ref{prop:Frob-invariants} is functorial with respect to strict morphisms of weak log curves. By the functoriality of their construction, it extends to Hopf groupoids:

\begin{corollary} \label{c:liftinggroupoidmodule}
  Under the hypotheses of Proposition~\ref{prop:Frob-invariants}, the specialization map induces an isomorphism
\[\on{sp}_{\ell}\colon \Pi((X,M);F_a,F_b)_{\ell}^{\varphi}\cong \Pi(\Gamma;\overline{a},\overline{b})_\ell^\varphi.\]
\end{corollary}

Moreover, it respects concatenation of paths and the coalgebra structure. We immediately recover a result of Besser.
\begin{corollary} \cite[Corollary~3.2]{Besser:Coleman}
Let $(X,M)$ be a weak log curve such that the underlying scheme $X$ is smooth. Let $F_a,F_b$ be fiber functors attached to log points anchored on $X$. Then, there is a unique Frobenius-invariant path in $\pi_1^{\rig,\un}((X,M);F_a,F_b)_{\ell}$.
\end{corollary}

\begin{proof}
Because the log dual graph of $(X,M)$ is a vertex, $\pi_1^{\un}(\Gamma,\overline{a})=\{e\}$, and there is a unique Frobenius-invariant path from $F_a$ to $F_b$.
\end{proof}

\section{Monodromy action on fundamental groups}

\subsection{Definition of the Monodromy Operator}
The monodromy operator $N$ on $\Pi((X,M);F_1,F_2)$ will be induced from a conjugation action on $\Pi((X_t,M_t);F_1,F_2)$ by using the homotopy exact sequence. The arguments in this section are similar to those in \cite[Section~2 and Section~6]{CPS:Logpi1}.
Let $f\colon (X,M)\to (S,\N)$ be a weak log curve. 
Let $i\colon (X,M)\to (X_t,M_t)$ be obtained from  $\pi$-$t$ base-change which also induces $f_t\colon (X_t,M_t)\to (S_t,\N_t)$.
Let $F_1,F_2$ be fiber functors on $(X,M)$ such that $(f_t\circ i)_*F_1=(f_t\circ i)_*F_2$ as fiber functors on $(S_t,\N_t)$. 
We have a commutative diagram of torsors
\[\xymatrix{\pi^{\rig,\un}_1((X,M);F_1,F_2)\ar[d]^{f_*}\ar[r]^>>>>>{i_*}
&\pi_1^{\rig,\un}((X_t,M_t);F_1,F_2)\ar[d]^{f_{t*}}\ar[r]
&\pi_1^{\rig,\varphi,\un(f_t,\nr)}((X_t,M_t);F_1,F_2)\ar[d]^{f_{t*}}\\
\pi^{\rig,\un}_1((S,\N);F_1,F_2)\ar[r]^{i_*}
&\pi_1^{\rig,\un}((S_t,\N_t);F_1,F_2))\ar[r]
&\pi_1^{\rig,\varphi,\nr}((S_t,
\N_t);F_1,F_2)}\]
Recall from Remark~\ref{r:fiberfunctorsonlogdisc} that $\pi_1^{\rig,\un}((S_t,\N_t),F_j)(K)\cong \Ga$.

\begin{definition} \label{defn:monodromy-operator}
Let 
$s_j\colon  \pi^{\rig,\un}_1((S_t,\N_t),F_j)\to \pi^{\rig,\un}_1((X_t,M_t),F_j)$
be sections of $f_{t*}$ for $j=1,2$. 
The \emph{monodromy action} $T(s_1,s_2)$ with respect to $(s_1,s_2)$ is the action of $\pi_1((S_t,\N_t),F_j)$ on $\pi^{\rig,\un}_1((X,M);F_1,F_2)$ given by 
\begin{align*}
\pi_1^{\rig,\un}((S_t,\N_t),F_j)&\to\on{Aut}(\pi^{\rig,\un}_1((X,M);F_1,F_2))\\
\gamma_u&\mapsto [\delta\mapsto (s_{1*}\gamma_u^{-1})(i_*\delta)(s_{2*}\gamma_u)].
\end{align*}
The {\em monodromy operator}
$N(s_1,s_2)\in\on{End}(\Pi((X,M);F_1,F_2))$
with respect to $(s_1,s_2)$ is the endomorphism obtained by applying $T(s_1, s_2)$ to the element of the Lie algebra of $\pi_1((S_t, \mathbb{N}_t), F_j)\cong \mathbb{G}_a$ given by $N_0:=\log(\gamma_u)/u$ (which is well-defined and independent of $u$ for $\|u\|<1$). 
\end{definition}

Observe that in the above definition, $(s_{1*}\gamma_u^{-1})(i_*\delta)(s_{2*}\gamma_u)$ is in $\pi^{\rig,\un}_1((X,M);F_1,F_2)$ by Corollary~\ref{c:exactsequenceoftorsors}.
We shall write $N$ for $N(s_1,s_2)$ where $s_1$ and $s_2$ are understood.
The monodromy operator is a derivation on composable paths:
\begin{lemma}
Let $F_1,F_2,F_3$ be fiber functors on $(X,M)$ that are mapped to the same fiber functor on $(S_t,\N_t)$ by $(f_t\circ i)_*$. Let $s_1,s_2,s_3$ be sections
\[s_i\colon  \pi^{\rig,\un}_1((S_t,\N_t),F_j)\to \pi^{\rig,\un}_1((X_t,M_t),F_j).\]
Let $\delta_1\in \pi^{\rig,\un}_1((X,M);F_1,F_2)$,  $\delta_2\in \pi^{\rig,\un}_1((X,M);F_2,F_3)$. Then 
\[N(s_1,s_3)(\delta_1\delta_2)=(N(s_1,s_2)\delta_1)\delta_2+\delta_1(N(s_2,s_3)\delta_2)\]
in $\Pi((X,M);F_1,F_3)$.
\end{lemma}

\begin{proof}
This follows from
\[(T(s_1,s_3)(\gamma_u))(\delta_1\delta_2)=\Big((T(s_1,s_2)(\gamma_u))(\delta_1)\Big)\Big(T(s_2,s_3)(\gamma_u))(\delta_2)\Big).\]
\end{proof}

\begin{remark}
 We can write $N$ more directly. Let $N_0\in\Pi((S_t,\N_t),F_i)$ be the generator of the $1$-parameter subgroup $\{\gamma_u\}_{u\in K}$ described in Definition \ref{defn:monodromy-operator}. Then for $\delta\in \Pi((X,M);F_1,F_2)$, 
\[N\delta=-(s_{1*}N_0)i_*\delta+i_*\delta (s_{2*}N_0).\]
\end{remark}

When $F_j$ is a fiber functor $F_B$ attached a log point $\iota\colon (x,\overline{\M})\to (X,M)$,
a section 
\[s\colon  \pi^{\rig,\un}_1((S_t,\N_t),F_j)\to \pi^{\rig,\un}_1((X_t,M_t),F_j)\]
can be constructed explicitly as follows.
The induced morphism
$\iota_t\colon (x,\overline{\M}_t)\to (X_t,M_t)$
induces a homomorphism fitting into the commutative diagram
\[\xymatrix{
\pi_1^{\rig,\un}((x,\overline{\M}_t),F_B)\ar[r]^{\iota_{t*}}&\pi_1^{\rig,\un}((X_t,M_t),F_B)\ar[d]\\
&\ar@{-->}[ul]^{\sigma} \pi_1^{\rig,\un}((S_t,\N_t),p_*F_B)
}\] 
where we  choose $\sigma$ to be a section of the natural homomorphism (by Lemma~\ref{l:pi1logpoint}),
\begin{multline*}
    \pi_1^{\rig,\un}((x,\overline{\M}_t),F_i)\cong \on{Hom}((\overline{M}_t/(h\circ f_t)^*\Z_{\geq 0})^{\gp},\Ga)\\
    \to \pi_1^{\rig,\un}((S_t,\N_t),f_{t*}F_i)\cong \on{Hom}((\Z_{\geq 0}^2/h^*\Z_{\geq 0})^{\gp},\Ga).
\end{multline*}
To produce $\sigma$, we choose a monoid homomorphism $\overline{M}_t/(h\circ f_t)^*\Z_{\geq 0}\to \Z_{\geq 0}^2/h^*\Z_{\geq 0}$ taking $e_t\mapsto e_t$. Then, set
$s=\iota_{t*}\circ\sigma$.

\begin{definition} \label{d:inducedsections} Let $F_B$ be the fiber functor attached to a log point $\iota\colon(x,\overline{\M}_t)\to (X_t,M_t)$ together with coordinate system $B$ as in Definition~\ref{d:fiberfunctorlogpoint}. There are {\em induced sections} as above where $\sigma$ is constructed from $\overline{M}_t/f^*{\Z_{\geq 0}}\to \Z_{\geq 0}/f^*\Z_{\geq 0}$ given by
\begin{enumerate}
    \item for a smooth point $\overline{M}_{0,t}=\<e_\pi,e_t\>$, $e_t\mapsto e_t$; 
    \item for a puncture $\overline{M}_{1,t}=\<e_\pi,e_t,f\>$, $e_t\mapsto e_t$, $f\mapsto 0$;
    \item for a node $\overline{M}_{2,t}=\<e_\pi,f_1,f_2\>$ with $B=\{f_j,e_t\}$ for $j=1$ or $j=2$, we pick
    $e_t\mapsto e_t$, $f_j\mapsto 0$, $f_{3-j}\mapsto e_t$ (inducing a section $\sigma_j$), and
    \item for an annular point $\overline{M}_{2',t}$ with $B=\{f_1,e_t\}$,  $e_t\mapsto e_t$, $f_1\mapsto0$, $f_2\mapsto e_t$.
\end{enumerate}
\end{definition}

%

\begin{remark}
The sections for the node can be understood from their complex analytic analogues. The node corresponds to the map of analytic spaces 
\[f_t\colon (\C^*)^2\to \C^*,\ (z_1,z_2)\mapsto z_1z_2\]
For a small $\varepsilon>0$, the loop $t\mapsto \varepsilon\exp(2\pi i t)$ in $\C^*$ can lift to $s_j(t)$ where
\[s_1(t)=(1,\varepsilon\exp(2\pi i t)),\ s_2(t)=(\varepsilon\exp(2\pi i t),1)).\]
\end{remark}

\begin{remark} \label{r:monodromyexplicit}
The monodromy operator on the universal bundle $\cE_n$ can be interpreted as the residue of a connection, at least when $F_1$ and $F_2$ are attached to smooth points $x_1,x_2\in X(k)$. We will define $\cW_n$ to be the relative universal object for $(X_t,M_t)\to (S_t,\N_t)$ with base section over $x_1$ as in Remark~\ref{r:universalbundle}. Pick a frame $((X_t,M_t),(X_t,M_t),(\cP,L))$.
Write $(W_n,\nabla)$ for the evaluation of $\cW_n$ on the frame. Lift $x_i$ to sections $\tilde{x}_i\colon \Spf V\ps{t} \to \cP$. 
By construction, we can mandate that the morphism of unipotent vector bundles with connections on $]x_1[_{\Spf V\ps{t}}$ 
\[\tilde{x}_1^*W_n\to\tilde{x}_1^*W_0=\mathbf{1}\]
has a splitting.  
 Let $e_n\in F_1(\cW_n)$ be the image of $1$ under the above splitting,
 viewed as an element of the fiber of $W_n$ over $\tilde{x}_1(\pi)$. 

Because $\cW_n$ pulls back to the universal object $\cE_n$ by the map $i\colon (X,M)\to (X_t,M_t)$, an element $\delta\in \Pi((X,M);F_1,F_2)/\mathscr{I}^{n+1}$ can be interpreted as $\delta(e_n)\in  F_2(\cE_n)$ (and thus an element of the fiber of $W_n$ over $\tilde{x}_2(\pi)$). Now, let $\tilde{F}_j$ be the tangential basepoint (see Remark~\ref{r:tangentialbasepoint}) attached to $\tilde{x}_j$. By conjugating by the specialization and tangential paths at $\tilde{x}_1$ and $\tilde{x}_2$, we view $\delta$ as a map
$\tilde{F}_1(\cW_n)\to \tilde{F}_2(\cW_n)$.
By identifying $e_n$ with its image in $\tilde{F}_1(\cW_n)$, we can interpret $\delta$ as $\delta(e_n)\in \tilde{F}_2(\cW_n)$.
 
 Now, the image of $\gamma_u$ under the monodromy action on $\delta$, 
 \[T(s_1,s_2)(\gamma_u)(\delta)\in \pi_1^{\rig,\un}((X_t,M_t);\tilde{F}_1,\tilde{F}_2)\]
 can be described as 
 \[\left((s_{1*}\gamma_u^{-1})\delta(s_{2*}\gamma_u)\right)(e_n)\in \tilde{F}_2(\cW_n).\]
 Intuitively, the image of $\delta$ under the monodromy operator $N(s_1,s_2)$ is the derivative of this action with
 respect to $u$ at $u=0$.
 We first observe that $(s_{1*}\gamma_{u}^{-1})(e_n)=e_n$. Indeed,  
 $s_{1*}\gamma_u^{-1}$ acts trivially on $\mathbf{1}$ and 
 the base section $e_n$ is in the image of the splitting, $\tilde{x}_1^*W_0\to \tilde{x}_1^*W_n$.
 We thus need only compute $\delta(s_{2*}\gamma_u)$ which can be interpreted as 
 \[\left(\delta(s_{2*}\gamma_u)\right)(e_n)= \left(s_{2*}\gamma_u\right) \left(\delta(e_n)\right)\in \tilde{F}_2(\cW_n).\]
 By Remark~\ref{r:fiberfunctorsonlogdisc}, the derivative of $s_{2*}\gamma_u$ with respect to $u$ at $u=0$ acts on $\tilde{F}_2(\cW_n)$ as $\Res_0(\nabla)$ where $\nabla$ is the connection on $\tilde{x}_2^*W_n$. This is the usual construction of monodromy via the residue of the Gauss--Manin connection. The special case $n=1$ corresponds to the monodromy operator on log rigid cohomology.
\end{remark}


\begin{remark} \label{r:dgmonodromy}
Analogous to the arguments involving the Frobenius operator in Remark~\ref{r:dgfrob}, one can identify the Deligne--Goncharov monodromy operator with the monodromy operator constructed here. Indeed, $\mathscr{W}^\vee_n$ is constructed from higher pushforwards in hypercohomology. The monodromy operator on hypercohomology is obtained by the residue construction as in cohomology, and thus agrees with our approach.
\end{remark}

\subsection{Monodromy action on good reduction curves}

The monodromy action is trivial if the underlying scheme of the weak log curve is smooth.

\begin{lemma} \label{l:trivialmonodromy} Let $(X,M)$ be a weak log curve whose underlying scheme is a smooth curve. Let $F_1,F_2$ be fiber functors attached to smooth points, punctures, or annular points anchored on $X$. Let $s_1,s_2$ be induced sections chosen as in Definition~\ref{d:inducedsections}. Then the monodromy action on $\pi^{\rig,\un}_1((X,M);F_1,F_2)$ is the identity.
\end{lemma}

\begin{proof} 
We begin with the case where there are no annular points on $(X,M)$. Then the $\pi$-$t$ base-change $(X_t,M_t)$ is isomorphic to the usual base-change $(X,M)\times_{(S,\N)} (S_t,\N_t)$ as in Remark~\ref{r:usualbasechange}. The sections $s_1,s_2$ are also obtained by base-change.
Since, by Proposition~\ref{p:kunnethpi1}
\[\pi^{\rig,\un}_1((X_t,M_t);F_{1},F_{2})\cong \pi^{\rig,\un}_1((X,M);F_1,F_2)\times \pi^{\rig,\un}_1((S_t,\N_t),F_1),\]
the monodromy action is trivial.

Given a weak log curve $(X,M)$ over $(S,\N)$, let $f\colon (X,M)\to (X,M')$ be the  inclusion of weak log curves obtained by replacing annular points with punctures as in Remark~\ref{r:modifylogstructure}. Because $f_*F_{\{f_1\}}=F_{\{f\}}$ for log basepoints at an annular point, the result follows from  Lemma~\ref{l:aphomotopyequivalence}.
\end{proof}

\subsection{Monodromy action on a node} \label{ss:monodromynode}

\begin{remark}
For the case of a node, $(x,\overline{\M}_2)$, there are two natural sections corresponding to $B_1=\{f_1\}$ and $B_2=\{f_2\}$,
\[
s_1=\left[
\begin{aligned}
f_1&\mapsto 0\\
f_2&\mapsto e_t\\
e_t&\mapsto e_t
\end{aligned}\right],\quad 
s_2=\left[
\begin{aligned}
f_1&\mapsto e_t\\
f_2&\mapsto 0\\
e_t&\mapsto e_t
\end{aligned}\right].\]
Then $s_{i*}\gamma_s$ is given by
\[
s_{1*}\gamma_u=\left[
\begin{aligned}
\Log(x_1)&\mapsto \Log(x_1)\\
\Log(x_2)&\mapsto \Log(x_2)+u\\
\Log(t)&\mapsto \Log(t)+u
\end{aligned}\right],\quad 
s_{2*}\gamma_u=\left[
\begin{aligned}
\Log(x_1)&\mapsto \Log(x_1)+u\\
\Log(x_2)&\mapsto \Log(x_2)\\
\Log(t)&\mapsto \Log(t)+u
\end{aligned}\right].\]
In the language of Example~\ref{e:pi1t-annulus},
\[s_{1*}\gamma_u=\gamma_{1,(0,u)},\quad s_{2*}\gamma_u=\gamma_{2,(0,u)}.\]
\end{remark}

\begin{lemma} The monodromy action of $\gamma_s$ on $\pi^{\un}_1((x,\overline{\M}_2);F_{\{f_1\}},F_{\{f_2\}})_{\ell}$ 
is given by 
\[((T(s_1,s_2)(\gamma_u))(\delta_v)=\delta_{u+v}\]
where $\delta_v$ is defined in Example~\ref{e:pi1annulus}.
\end{lemma}

\begin{proof}
Because $i_*(\delta_v)=\delta_{(v,0)}$ (where $i\colon (x,\overline{\M}_2)\to (x,\overline{\M}_{2,t})$), by Example~\ref{e:pi1t-annulus},
\[   (T(s_1,s_2)(\gamma_u))(\delta_v)=(s_{1*}\gamma_{-s})\delta_{(v,0)}(s_{2*}\gamma_{u})
=\gamma_{1,(0,-u)}\delta_{(v,0)}\gamma_{2,(0,u)}
=i_*\delta_{u+v}.\]
\end{proof}

The action of the monodromy operator can be described by differentiating the above with respect to $s$. We interpret $\delta_{v}$ as the substitution on  $F_{\{f_1\}}(\cE)$ (where elements are considered as $\Log$-analytic functions) given by 
\[\Log(x_1)\mapsto \ell-\Log(x_2)+v,\]
as in Example~\ref{e:pi1annulus} and write
\[\gamma_u=\exp(uN_0).\]
By matching powers of $u$, we immediately obtain the following:

\begin{lemma} \label{l:monodromyannulus}
Let $D$ be the endomorphism on $F_{\{f_1\}}(\cE)$ (considered as tuples of $\Log$-analytic functions) acting component-wise as the derivation induced by
\[x_1\mapsto 0,\quad \Log(x_1)\mapsto 1.\]
Then, the element $N^k\delta_0\in\Pi((x,\overline{\M}_2);F_{\{f_1\}},F_{\{f_2\}})_{\ell}$ is given by $\delta_0\circ D^k$ (i.e.,~$D^k$ followed by the substitution $\Log(x_1)\mapsto \ell-\Log(x_2))$.
\end{lemma}

It is worthwhile to describe the action of $D$ using the tangential basepoint $\tilde{F}_{\{f_1\}}$ on $(x,\overline{\M}_2)$ and arguments from Remark~\ref{r:fiberfunctorsonlogdisc}.  By virtue of the proof of Lemma~\ref{l:loganalyticsections}, we can evaluate $\cE$ in a frame around $(x,\overline{\M}_2)$. This yields a bundle $E\cong \cO^m$ equipped with unipotent connection $\nabla=d-\omega$  where $\omega$ is a matrix of log $1$-forms whose coefficients are analytic functions (expressed as power series) in $x_1$. Write $\Res_0(\nabla)\in \End(K^m)$ for the coefficient of $\frac{dx_1}{x_1}$ in $\omega$. We have the following straightforward lemma obtained by conjugating by the tangential path:

\begin{lemma} \label{l:monodromyresidue}
  Let $\cE$ be a unipotent isocrystal of index of unipotency $n$. Then, the action of $D^k$ on $\tilde{F}_{\{f_1\}}(\cE)$ is given by multiplying by $\Res_0(\nabla)^k$. In particular, if $k>n$, then $D^k$ acts as zero. 
\end{lemma}

\subsection{Frobenius and monodromy-weight filtrations}
We now define the monodromy filtration following \cite{deligne1980conjecture}.  Namely, we set $M_\bullet$ to be the unique filtration on $\Pi((X, M); F_1, F_2)$ such that $NM_i\subset NM_{i-2}$ and, for all $j,k$, $N^k$ induces an isomorphism $$N^k: \on{gr}^M_{j+k}\on{gr}^W_j \Pi((X, M); F_1, F_2)\overset{\cong}{\to} \on{gr}^M_{j-k}\on{gr}^W_j \Pi((X, M); F_1, F_2).$$
Uniqueness of this filtration is guaranteed by \cite[1.6.13]{deligne1980conjecture}; existence is a consequence of the following theorem which appears as \cite[Theorem 1.3]{betts-litt}.

\begin{theorem}[The weight-monodromy conjecture for $\Pi$] \label{t:weightmonodromy}
Let $F_1$ and $F_2$ be fiber functors attached to log points.
A monodromy filtration on $\Pi((X, M); F_1,F_2)$ exists, and moreover the monodromy filtration on $\on{gr}_j^W\Pi((X, M); F_1,F_2)$ is split by its decomposition into generalized eigenspaces of $\varphi$ with eigenvalues given by $q$-Weil numbers of a particular weight.
\end{theorem}   
\begin{proof}  
The argument is essentially the same as in \cite{betts-litt}; indeed, the argument there is essentially a formal deduction from the weight-monodromy conjecture for $H^1$. While  the result in that paper is stated for fiber functors coming from smooth points $x_1$ and $x_2$, the general case can be reduced to this one. We explain the reduction now.

In \cite{betts-litt}, $\Pi((X, M); x_1, x_2)$ is defined by applying the functor $D_{\on{pst}}$ to the $\mathbb{Q}_p$-pro-unipotent completion of the \'etale fundamental group of the generic fiber of a smooth lift of $X$ to $W(k)$; the equivalence with our definition follows from the construction of Deligne-Goncharov (Section \ref{ss:delignegoncharov}) and the comparison between \'etale and log-crystalline \cite{tsuji1999p} and log-crystalline and log rigid cohomology \cite{Shiho:relative}.

Let $\cX$ be the generic fiber of a smooth lift of $X$ to $W(k)$. First, we give the argument when either $F_1$ or $F_2$ arises from a smooth point $\tilde{x}$ of $\cX$.  We assume $F_1$ arises from a smooth point; the case where $F_2$ arises form a smooth point is identical. In this case we may restrict the diagram in schemes constructed in Section \ref{ss:delignegoncharov} to $\cX\times \{\tilde{x}\}$, yielding a simplicial scheme over $\cX$. (Explicitly, this diagram sends $J$ to $\pi_0^{-1}(x),$ where $\pi_0$ is defined as in Section \ref{ss:delignegoncharov}.) By Remark~\ref{r:delignegoncharovobject}, the $p$-adic \'etale cohomology of this diagram yields a lisse sheaf on $\cX$, which gives by Theorem \ref{t:deligne-goncharov}, upon applying $F_2\circ D_{pst}$, the algebra $\Pi((X,M), \overline x, F_2)/\mathscr{I}^n)^\vee$. 
Now the proof of \cite[Theorem 3.3]{betts-litt} applies verbatim.

Now choose an arbitrary point $\overline x$ of $\cX$; by the previous paragraph, we have that $\Pi((X,M); F_1, \overline x)$ and $\Pi((X, M); \overline x, F_2)$ satisfy the statement of the theorem. But $\Pi((X, M); F_1, F_2)$ is a quotient of $\Pi((X, M); F_1, \overline x)\otimes \Pi(X, M); \overline x, F_2)$, and the conclusion of the theorem is preserved by tensor products and quotients, by \cite[Theorem 2.10]{betts-litt}.
\end{proof}

The monodromy filtration is multiplicative in the sense that $$M_i\cdot M_j\subset M_{i+j}.$$

From Section \ref{ss:unipotentfundamentalgroups} we have a canonical isomorphism for the augmentation ideal $\mathscr{I}\subset \Pi((X,M);F_1,F_2)$,
\[\mathscr{I}/\mathscr{I}^{2} \cong H^1_{\rig}((X,M))^\vee;\]
this identification is $N$-equivariant and composing with the dual of the identification from Proposition \ref{p:monodromycohomology} gives the following statement.

\begin{proposition}\label{prop:harmonic-forms}
There is a canonical identification 
\[M_{-2}(\mathscr{I}(F_1, F_2)/\mathscr{I}(F_1,F_2)^{2}) = H_1^\diamond(\Gamma)^\vee.\]
\end{proposition}
Recall that in the above proposition, $\mathscr{I}(F_1,F_2)^\bullet$ is the filtration on $\Pi((X,M);F_1,F_2)$ induced by the $\mathscr{I}$-adic filtration on $\Pi((X,M), F_1)$, where $\mathscr{I}\subset \Pi((X,M), F_1)$ is the augmentation ideal.

\begin{proposition} \label{p:monodromyidentification}
Under the identification above, the multiplication map 
\[\mathscr{I}(F_1,F_1)^{\otimes n-1}\otimes\mathscr{I}(F_1,F_2)\to \mathscr{I}(F_1,F_2)^n\] induces an isomorphism
\[M_{-2n}(\mathscr{I}(F_1, F_2)^n/\mathscr{I}(F_1,F_2)^{n+1}) \cong \left(H_1^{\diamond}(\Gamma)^{\otimes n}\right)^ \vee.\]
\end{proposition}

\begin{proof}
It suffices to show that the natural map 
\begin{multline*}
M_{-2(n-1)}(\mathscr{I}(F_1,F_1)/\mathscr{I}(F_1,F_1)^2)^{\otimes n-1}\otimes M_{-2}(\mathscr{I}(F_1,F_2)/\mathscr{I}(F_1,F_2)^2)\\
\to M_{-2n}(\mathscr{I}(F_1,F_2)^n/\mathscr{I}(F_1,F_2)^{n+1})
\end{multline*}
is an isomorphism. First note that it is enough to prove this when $F_1=F_2$, as $\Pi((X,M); F_1,F_2)$ is a free module of rank one over $\Pi((X,M); F_1)$ and the left module structure respects the monodromy and augmentation filtrations. Indeed, \[\on{gr}_{\mathscr{I}}\Pi((X,M); F_1, F_2)\coloneqq \bigoplus_n \mathscr{I}(F_1,F_2)^n/\mathscr{I}(F_1,F_2)^{n+1}\] is a free module of rank $1$ over $\on{gr}_{\mathscr{I}}\Pi((X,M),F_1)$ whose degree zero piece is canonically isomorphic to $K$. The isomorphism $\on{gr}_{\mathscr{I}}\Pi((X,M); F_1, F_2)\cong \on{gr}_{\mathscr{I}}\Pi((X,M),F_1)$ induced by sending $1$ to $1$ preserves the monodromy filtration, yielding the desired reduction.

Now, we note that the algebra 
\[\on{gr}_{\mathscr{I}}\Pi((X,M),F_1)\coloneqq \bigoplus_n \mathscr{I}(F_1,F_1)^n/\mathscr{I}(F_1,F_1)^{n+1}\] 
is the free tensor algebra on $\mathscr{I}(F_1,F_1)/\mathscr{I}(F_1,F_1)^2$, modulo the ideal generated by the image of the natural map 
\[H_{\rig}^2((X,M))^\vee\to H_{\rig}^1((X,M))^{\otimes 2,\vee}\cong (\mathscr{I}(F_1,F_1)/\mathscr{I}(F_1,F_1)^2)^{\otimes 2}\]
dual to the cup product. This is a standard fact about pro-unipotent fundamental groups; see e.g.,~\cite[Theorem A.1]{betts-litt} (alternatively, this follows from Section~\ref{ss:unipotentfundamentalgroups}).

Because $H_{\rig}^2((X,M))$ is pure of weight $2$ and has monodromy filtration concentrated in degree $2$, the tensor sub-algebra of $\on{gr}_{\mathscr{I}}\Pi((X,M),F_1)$ generated by $M_{-2}\mathscr{I}(F_1,F_1)/\mathscr{I}^2(F_1,F_1)$ is free, implying the conclusion.
\end{proof}

We will later need a description of a piece of the weight filtration.

\begin{corollary} \label{c:weightsubmodule}
The submodule induced by $W_{-n-1}$ on
\[
M_{-2n}(\mathscr{I}(F_1, F_2)^n/\mathscr{I}(F_1,F_2)^{n+1}) \cong \left(H_1^{\diamond}(\Gamma)^{\otimes n}\right)^\vee\]
is
\[\ker\left((\varrho^*)^{\vee}\colon \left(H_1^{\diamond}(\Gamma)^{\otimes n}\right)^\vee\to \left( H^{\diamond}_1(\overline{\Gamma})^{\otimes n}\right)^\vee \right)\]
where $\varrho\colon \Gamma\to\overline{\Gamma}$ is as in Remark~\ref{r:contracttocore}.
\end{corollary}

\begin{proof}
  Let $(X,M')$ be obtained from $(X,M)$ by replacing annular points and punctures by smooth points as in Remark~\ref{r:modifylogstructure}. Let $f\colon (X,M)\to (X,M')$ be the natural morphism which induces a map on the associated gradeds of cohomology:
\[\xymatrix{  \on{gr}^M_2 H^1_{\rig}((X,M'))\cong H^{\diamond}_1(\overline{\Gamma})\ar[r]^{\varrho^*}&\on{gr}^M_2 H^1_{\rig}((X,M))\cong H_1^\diamond(\Gamma)}\]
as can be seen by examining the spectral sequence in Section~\ref{ss:cohomologyoflogcurves}. From the identification $\on{gr}^W_1 H^1_{\rig}((X,M))=\on{Im}(f^*)$, we can conclude 
\[\on{gr}^W_1\on{gr}^M_2 H^1_{\rig}((X,M))=\on{Im}(\varrho^*)\subset H_1^\diamond(\Gamma)\]
By the definition of the weight filtration in Section~\ref{ss:weightfiltration}, we obtain
\[W_{-2}M_{-2}(\mathscr{I}(F_1,F_2)/\mathscr{I}(F_1,F_2)^2)=\ker((\varrho^*)^\vee).\]
The conclusion follows by multiplicativity of the weight and monodromy filtrations.
\end{proof}

\part{Integration}

\section{Integration along paths} \label{s:integrationalongpaths}

In this section, we define and establish basic properties of iterated integrals of $1$-forms along ``formal sums of paths" in $(X, M)$ --- that is, elements of $\Pi((X,M);F_1,F_2)$. Because we allow basepoints attached to log points, extra care must be taken if we wish to write the values of integrals as elements of $K$.

\subsection{Definition of Integrals}
Let $((X,M_X),(X,M_X),(\cP,L))$ be a proper log smooth frame where $(X,M_X)$ is a weak log curve or a log point over $(S,N)$.
Let $\Omega^1$ denote the vector space of $1$-forms:
\[\Omega^1\coloneqq\Gamma(]X[_\cP,\cO_{]X[_\cP}\otimes\Omega^*_{(]X[_\cP,L)/(\Spf V^{\an},N))}).\]
Let $F_a$ and $F_b$ be fiber functors on $(X,M_X)$ attached to log points. For $p\in \Pi((X,M_X);F_a,F_b)$ and $\omega_1\dots\omega_n\in\Omega^1$, we will define
$\int_{p,a}^b \omega_1\dots\omega_r.$

\begin{definition} \label{d:hodgeconnection} For $\omega_1,\dots,\omega_r\in\Omega^1$, we define a unipotent isocrystal  $\cE_{\omega_1\dots\omega_r}$ on $(X,M_X)$ by specifying a unipotent vector bundle $E_{\omega_1\dots\omega_r}$ (i.e.,~a locally free $\cO_{]X[_\cP}$-module) with integrable connection. The underlying vector bundle is a trivial bundle of rank $r+1$ such that for basis sections $e_0,\dots,e_r$, the connection obeys
\[\nabla e_i=
\begin{cases}
-\omega_i \otimes e_{i+1} &\text{if }0\leq i\leq n-1\\
0 &\text{if }i=n
\end{cases}\]
\end{definition}
This bundle induces an isocrystal on $(X,M_X)$. Formally, 
\[\left(1,\int \omega_1,\int \omega_1\omega_2,\dots,\int \omega_1\dots\omega_r\right)\]
would be a horizontal global section of  $E_{\omega_1\dots\omega_r}$.

\begin{definition} \label{d:liftoflogpoint} Let $(x,\overline{\M})$ be a log point.
Given a log morphism $a\colon(x,\overline{\M})\to (X,M)$, a {\em lift} of $a$ is given by a morphism of frames
\[\xymatrix{
(x,\overline{\M})\ar[r]\ar[d]&(x,\overline{\M})\ar[r]\ar[d]&(\cP_{\overline{M}},L_{\overline{M}})\ar[d]\\
(X,M_X)\ar[r]&(X,M_X)\ar[r]&(\cP,L)}\]
where the top line is given in Section~\ref{ss:logpoints}. We will also denote the morphism of frames by $a$.
\end{definition}

For a fiber functor $F_B$ attached to a coordinate system $B$ of $\overline{M}$ at the log point $(x,\overline{\M})$, and a lift $a$, we may interpret $F_B$ in terms of $\Log$-analytic sections on $]x[_{\cP_{\overline{M}}}$ which should be thought of as functions pulled back from $]X[_{\cP}$ by $a$. When $(x,\overline{\M})$ is a smooth point, a lift is a smooth $K$-point of $]X[_{\cP}$, and $F_B$ is the fiber over that $K$-point where one interprets an  isocrystal on the frame  
$((x,\overline{\M}_0),(x,\overline{\M}_0),(\Spf V,L))$ as a $K$-vector space.

To define integrals, we need a basis for $F_B(\cE_{\omega_1\dots\omega_r})$. 
As per Lemma~\ref{l:loganalyticsections}, elements of $F_B(\cE_{\omega_1\dots\omega_r})$ can be expressed as $(r+1)$-tuples of $\Log$-analytic functions on $]x[_{\cP_{\overline{M}}}$. Their components  are polynomials in $\Log(x_1)$ whose coefficients are power or Laurent series in $x_1$ (depending on $\overline{M}$). Therefore, once we pick a frame, we may speak of the constant terms of a section (i.e.,~the coefficient of $x_1^0\Log(x_1)^0$). The following lemma is a straightforward consequence of the proof of Lemma~\ref{l:loganalyticsections}.

\begin{lemma}
Let $F_B$ be a fiber functor attached to a log point $(x,\overline{\M})$ 
where $\overline{\M}\in\{\overline{\M}_0,\overline{\M}_1,\overline{\M}_2\}$
Let $a\colon (\cP_{\overline{M}},L_{\overline{M}})\to (\cP,L)$ be a lift of the log point.
Then, $a^*\cE_{\omega_{1}\dots\omega_{r}}$ is the unipotent isocrystal on $(x,\overline{\M})$ induced by the bundle $E_{a^*\omega_1\dots a^*\omega_r}$ on $]x[_{\cP_{\overline{\M}}}$. There is a basis (which we call the {\em lifted basis})
\[\{h_0,\dots,h_r\}\subset F_B(\cE_{\omega_1\omega\dots\omega_r})\]
characterized by the equality of their constant terms:
\[\on{const}(h_i)=e_i.\]
For a smooth point, this becomes $h_0=e_0,\dots,h_r=e_r$. 
\end{lemma}

We write $h_{a,i}$ for $h_i$ when $a$ is not clear from the context.  These bases are natural under the following maps:
the horizontal inclusion for $1\leq k\leq r$,
\[\iota_k\colon \cE_{\omega_k\dots\omega_r}\to\cE_{\omega_1\dots\omega_r},\quad e_i\mapsto e_{i+k-1}\]
and the horizontal projection
\[p_k\colon\cE_{\omega_1\dots\omega_r}\to\cE_{\omega_1\dots\omega_k},\quad e_i\mapsto e_k,\]
in the sense that
\begin{enumerate}
    \item $h_0\in F(\cE_\varnothing)$ is $1$,
    \item for all $k$, $\iota_{k*}(h_i)=h_{i+k-1}$ for $0\leq i\leq r-k+1$, and
    \item for all $k$, $p_{k*}(h_i)=h_i$ for $0\leq i\leq k$.
\end{enumerate}

\begin{definition}
Let $F_a,F_b$ be fiber functors attached to log points for which we have chosen lifts $a$ and $b$. 
We define the {\em iterated integral} for $\omega_1,\dots,\omega_r$ and $p\in \Pi((X,M_X);F_a,F_b)$,
\[\int_{p,a}^b \omega_{1}\dots\omega_{r}\]
to be the $h_{b,r}$-component of the image of $h_{a,0}$ under the $K$-linear map
\[p\colon F_a(\cE_{\omega_1\dots\omega_r})\to F_b(\cE_{\omega_1\dots\omega_r}).\]
Here, the integral depends on the fiber functors $F_a,F_b$, but this has been suppressed from the notation.
Given $q\in \pi_1((X, M_X); F_a, F_b)(K),$ we define $$\int_{q, a}^b \omega_1\dots\omega_r:=\int_{g_q,a}^b \omega_1\dots\omega_r,$$ where $g_q\in \Pi((X, M_X); F_a, F_b)$ is the grouplike element defined in Remark \ref{rmk:construction-of-grouplike-elements}.
\end{definition}

\begin{example} \label{e:annulusintegral}
 Let $(x,\overline{\M}_2)$ be the nodal log point, and let $F_a=F_{\{f_1\}}$, $F_b=F_{\{f_2\}}$. 
 Pick $(\cP,L)$ as in Subsection~\ref{ss:logpoints}.
 For $\delta_0\in \pi_1^{\rig,\un}((x,\overline{\M}_2);F_a,F_b)_{\ell}$ as in Example~\ref{e:pi1annulus},
 \[\int_{\delta_0,a}^b \left(\frac{dx_1}{x_1}\right)^r=\frac{\ell^r}{r!}.\]
  Indeed, let $\omega=\frac{dx_1}{x_1}$. We can express $\cE_{\omega^r}$ evaluated on $]x[_{\cP}$ as the trivial bundle with fiber given by the vector space underlying the polynomial ring $K[c]/c^{r+1}$ and connection $1$-form $\omega$ given by 
  \[\frac{dx_1}{x_1}(c\ \cdot\ ) =-\frac{dx_2}{x_2}(c\ \cdot\ )\]
  where $(c\ \cdot\ )$ is the linear transformation given by multiplying by $c$.
  Then 
  \[h_{a,i}=c^i\exp(c\Log(x_1)),\ h_{b,i}=c^i\exp(-c\Log(x_2)).\]
  Now $\delta_0$ is given by the substitution $\Log(x_1)\mapsto \ell-\Log(x_2)$. Consequently,
  \[p(h_{a,0})=\exp(c(\ell-\Log(x_2)))=\exp(c\ell)(\exp(-c\Log(x_2))).\]
  Its $h_{b,r}$-component (that is, the $c^r$ component) is $\ell^r/r!$.
\end{example}

\begin{remark} \label{r:paralleltransport}
One can justify this definition in terms of parallel transport. Suppose that $F_a$ and $F_b$ are attached to smooth points lifting to $K$-points $a$ and $b$ and that
there exists a morphism $s\colon \mathbf{1}\to \cE_{\omega_1\dots\omega_r}$. Evaluated on a frame, this corresponds to a horizontal morphism from a rank $1$ trivial bundle (with trivial connection) to $E_{\omega_1\dots\omega_r}$. Suppose also that $F_a(s)(1)=e_0$ for $1\in F_a(\mathbf{1})$. Then,
\[F_b(s)(1)=F_b(s)(p(1))=p(F_a(s(1)))=\sum_{i=0}^n \left(\int_{p,a}^b \omega_1\dots\omega_i\right)e_i.\]

We can consider the case where $a$ and $b$ are in the same residue disc and treat $s(1)$ as a function of $b$. Then, horizontality implies
\[d\left(\int_{p,a}^b \omega_1\dots\omega_i\right)=\left(\int_{p,a}^b \omega_1\dots\omega_{i-1}\right)\omega_i.\]
This, together with $\int_{p,a}^a \omega_1\dots\omega_i=0$, characterizes the integral on a residue disc.
\end{remark}

\begin{lemma} \label{l:multilinear}
The iterated integral is multilinear in $\omega_1\dots\omega_r$.
\end{lemma}

\begin{proof}
It suffices to show that for any $\omega'_i\in\Omega$ and $c\in K$,
\[\int_{p,a}^b \omega_1\cdots\omega_{i-1}(\omega_i+c\omega'_i)\omega_{i+1}\cdots\omega_r
= \int_{p,a}^b \omega_1\cdots\omega_{i-1}\omega_i\omega_{i+1}\cdots\omega_r+
c\int_{p,a}^b \omega_1\cdots \omega_{i-1}\omega'_i\omega_{i+1}\cdots\omega_r.\]
Write $\omega^+,\omega,\omega'$ for the above $r$-tuples of $1$-forms.
Consider the diagram of vector bundles with integrable connections on $]X[_{\cP}$:
\[\xymatrix{
0\ar[r]&\mathscr{K}\ar[r]\ar[d]^s&
E_{\omega}\oplus E_{\omega'}\ar[r]^c& E_{\omega_1\dots\omega_{i-1}}\ar[r]&0\\
&E_{\omega^+}&&&
}\]
where $c$ is given by
\[((t_0,\dots,t_r),(t'_0,\dots,t'_r))\mapsto (t_0-t'_0,\dots,t_{i-1}-t'_{i-1}),\]
$\mathscr{K}$ is the kernel, and $s$ is defined by
\[((t_0,\dots,t_r),(t'_0,\dots,t'_r))\mapsto (t_0,\dots,t_{i-1},t_i+ct'_i,\dots,t_r+ct'_r).\]
Let $h^+_i,h_i,h'_i$ denote the lifted bases of the respective bundles.
Now, $s((h_0,h'_0))=h^+_0$, and the image of $(h_0,h_0')$ and $h^+_0$ under $p$ have the following coefficients for  their $(h_r,h'_r)$ and $h^+_r$-components:
\[\left(\int_{p,a}^b \omega,\int_{p,a}^b \omega'\right),\quad \int_{p,a}^b \omega^+.\]
Since $s$ takes the former to the later, the conclusion follows.
\end{proof}

In the bases defined above, $p$ on $\cE_{\omega_1\dots\omega_r}$ can be expressed as a matrix 
\[\left[
\begin{array}{ccccc}
1               & 0 & 0 & \dots & 0\\
\int_{p,a}^b \omega_1         & 1               & 0 & \dots & 0\\
\int_{p,a}^b \omega_1\omega_2 & \int_{p,a}^b\omega_2 & 1 & \dots & 0\\
\vdots  & \vdots & \vdots & \ddots & \vdots\\
\int_{p,a}^b \omega_1\dots \omega_r & \int_{p,a}^b \omega_2\dots \omega_r&\int_{p,a}^b \omega_3\dots \omega_r& \cdots & 1
\end{array}
\right]
\]

By multiplying such matrices (and noting the convention that $pq$ means $p$ followed by $q$), we arrive at the following {\em concatenation formula}:

\begin{proposition} \label{p:concatenation} Let $F_a,F_b,F_c$ be fiber functors attached to log points equipped with lifts on $(X,M)$. Let $p\in \Pi((X,M);F_a,F_b)$ and $q\in \Pi((X,M);F_b,F_c)$. Then,
\[\int_{pq,a}^c \omega_1\dots \omega_r=\sum_{k=0}^r\int_{p,a}^b \omega_{1}\dots \omega_k\int_{q,b}^c \omega_{k+1}\dots\omega_{r}.\]
\end{proposition}

\begin{proposition} \label{p:functoriality}
 Let 
    \[f\colon ((X,M_X),(X,M_X),(\cP,L))\to ((X',M_{X'}),(X',M_{X'}),(\cP',L'))\]
    be a morphism of proper log smooth frames where $(X,M_X)$ and $(X',M_{X'})$ are weak log curves. 
    Let $F_a$ and $F_b$ be fiber functors attached to log points equipped with lifts $a$ and $b$, respectively, on $(X,M)$.
    Then $f_*F_a$ and $f_*F_b$ are fiber functors attached to log points equipped with lifts $f\circ a$ and $f\circ b$, respectively. 
    Let $\omega'_1,\dots\omega'_r$ be $1$-forms on $]X'[_{\cP'}$. Let $p\in\Pi((X,M);F_a,F_b)$. Then
  \[\int_{p,a}^b f^*\omega'_1\dots f^*\omega'_r=
  \int_{f_*p,f\circ a}^{f\circ b} \omega'_1\dots\omega'_r.\]
\end{proposition}

\begin{proof}
  This is an immediate consequence of $f^*\cE_{\omega'_1\dots\omega'_r}=\cE_{f^*\omega'_1\dots f^*\omega'_r}$
\end{proof}

We will need the notion of an analytic function vanishing at the fiber functor attached to a log point.

\begin{definition}
  Let $F$ be a fiber functor attached to a log point equipped with a lift $a$. Let $g\in \Gamma(]X[_{\cP},\cO_{]X[_{\cP}})$. We say $g$ {\em vanishes} at $a$ if the morphism $\mathbf{1}\to \cE_{dg}$ given by $e_0\mapsto e_0+ge_1$
  induces a map $F(\mathbf{1})\to F(\cE_{dg})$ taking $1$ to $h_0$.
\end{definition}

If $F$ is the fiber functor attached to a smooth point, this is the usual notion of vanishing at the point. In the case of log points, it means that the constant term vanishes when written as a $\Log$-analytic function after being  pulled back by the lift.

The following  analogue of integration by parts is useful for computing iterated integrals.

\begin{lemma} \label{l:integrationbyparts}
  Let $F_a$ and $F_b$ be fiber functors attached to log points equipped with lifts $a$ and $b$, respectively, and let $p\in \Pi((X,M_X);F_a,F_b)$.
  Let $g\in \Gamma(]X[_{\cP},\cO_{]X[_{\cP}})$ be a function vanishing at $a$. Let $\omega_1\dots\omega_r\in\Omega^1$. For $1\leq i\leq r-1$,
\[
\begin{split}
\int_{p,a}^b (dg)\omega_1\dots\omega_r
&=\int_{p,a}^b (g\omega_1)\omega_2\dots\omega_r,\\
\int_{p,a}^b \omega_1\dots\omega_i(dg)\omega_{i+1}\dots\omega_r
&=\int_{p,a}^b \omega_1\dots\omega_{i}(g\omega_{i+1})\omega_{i+2}\dots\omega_r-\int_{p,a}^b\omega_1\dots\omega_{i-1}(g\omega_{i})\omega_{i+1}\dots\omega_r,\\
\int_{p,a}^b \omega_1\dots\omega_r(dg)
&=g\int_{p,a}^b \omega_1\dots\omega_r-\int_{p,a}^b\omega_1\dots\omega_{r-1}(g\omega_r).
\end{split}
\]\end{lemma}

\begin{proof}
  We prove the middle formula using the techniques of Lemma~\ref{l:multilinear}. The other cases are similar.
  
  Let
  \begin{align*}
  \omega&=\omega_1\dots\omega_i(dg)\omega_{i+1}\dots\omega_r,\\
   \omega^+&=\omega_1\dots\omega_{i}(g\omega_{i+1})\omega_{i+2}\dots\omega_r\\ 
  \omega^-&=\omega_1\dots\omega_{i-1}(-g\omega_{i})\omega_{i+1}\dots\omega_r
  \end{align*}
  
  Consider the diagram where $\mathscr{K}$ is defined as the kernel
  \[\xymatrix{
0\ar[r]&\mathscr{K}\ar[r]\ar[d]^s&
E_{\omega^+}\oplus \cE_{\omega^-}\ar[r]& E_{\omega_1\dots\omega_{i-1}}\ar[r]&0\\
& E_{\omega}&&&
}\]
  and the map $s$ is given by
  \[((t^+_0,\dots,t^+_r),(t^-_0,\dots,t^-_r)\mapsto 
  (t_0,\dots,t_{i-1},t_i,gt_i+t'_i,t_{i+1}+t'_{i+1},\dots,t_r+t'_r).\]
  It is straightforward to verify that $s$ is horizontal and that $s(h^+_0,h^-_0)=h_0$. The conclusion follows as before.
\end{proof}

\subsection{Symmetrization relation}

Now we will prove the symmetrization relation:
\begin{proposition} \label{p:symmetrization} Let $F_a,F_b$ be fiber functors attached to smooth points equipped with lifts $a$ and $b$.
Let $\omega_1,\dots,\omega_r\in\Omega^1$ and $p\in \pi_1((X,M);F_a,F_b)$. Then,
\[\sum_{\sigma\in S_n} \int_{p,a}^b \omega_{\sigma(1)}\dots\omega_{\sigma(r)}=\prod_{i=1}^r \int_{p,a}^b\omega_i.\]
\end{proposition}

This proposition is an immediate consequence of $p$ being group-like, but we will provide a proof for completeness.  Let $\cE^\Omega_r$ be the unipotent isocrystal attached to $E^\Omega_r$, the unipotent bundle with integrable connection on $]X[_{\cP}$ given by the trivial bundle attached to  the truncated tensor algebra $T^{\leq r}\Omega^\vee=\bigoplus_{i=0}^r \left(\Omega^\vee\right)^{\otimes i},$
with connection $1$-form $\omega$ 
given by the image of the identity under the map
\[\on{Hom}(\Omega^\vee,\Omega^\vee)\cong \Omega\otimes \Omega^\vee\to \Omega \otimes \End(T^{\leq r}\Omega^\vee)\]
where the arrow takes $v\in \Omega^\vee$ to the endomorphism given by right-multiplication by $v$.

There is a natural projection $q_{\omega_{i_1}\dots\omega_{i_r}}\colon E_r^\Omega\to E_{\omega_{i_1}\dots\omega_{i_r}}$ induced by 
the block diagonal linear map
\[T^{\leq n}\Omega^\vee=\bigoplus_{i=0}^r \left(\Omega^\vee\right)^{\otimes i}\to \bigoplus_{i=0}^r K=K^{r+1}\] 
given by
\begin{align*}
    \left(\Omega^\vee\right)^0&\to K, & \left(\Omega^\vee\right)^{\otimes m}&\to K\\
    1 &\mapsto 1, & v_1\dots v_m &\mapsto
    \<\omega_{i_1},v_1\>\cdot\dots\cdot\<\omega_{i_m},v_m\>. 
\end{align*}
There is also a horizontal morphism $j\colon E^{\Omega}_r\to \left(E^\Omega_1\right)^{\otimes r}$ given by $j=\sum_{i=1}^r j_i$
where $j_i$ is the algebra extension of 
\[\Omega^\vee\to \left(T^{\leq 1}\Omega^\vee\right)^{\otimes r},\ v\mapsto 1\otimes\dots\otimes 1\otimes v\otimes 1\otimes \dots\otimes 1\]
where  $v$ occurs in the $i$th factor.     
\begin{proof}
Note that $F_a(\cE)$ and $F_b(\cE)$ are the fibers over $a$ and $b$ of $E$, respectively, and the lifted bases are the basis vectors $e_0,\dots,e_n$. 
There is a commutative diagram
\[\xymatrix{
F_a(\cE^\Omega_r)\ar[r]^{F_a(j)}\ar[d]^p&F_a(\cE^\Omega_1)^{\otimes r}\ar[d]^p\\
F_b(\cE^\Omega_r)\ar[r]^{F_b(j)}&F_b(\cE^\Omega_1)^{\otimes r}
}\]
Note that $e_0\in  F_a(\cE^\Omega_r)=T^{\leq r}\Omega^\vee$ maps to $e_0\otimes\dots\otimes e_0$ under $F_a(j)$.
By treating $\omega_1\otimes\dots\otimes\omega_r\in \Omega^{\otimes r}$ as an element in $(F_b(\cE^\Omega_1)^{\otimes r})^\vee$, we
see for $s\in F_b((\cE^\Omega_1)^{\otimes r})$,
\[(\omega_1\otimes\dots\otimes\omega_r)(s)=[e_r](F_b(q_{\omega_1\dots\omega_r})s)
\]
where $[e_r]$ denotes taking the $e_r$ component.
We obtain
\begin{align*}
\prod_{i=1}^r \int_{p,a}^b\omega_i&=(\omega_1\otimes\dots\otimes\omega_r)p(e_0\otimes\dots\otimes e_0)\\
&=(\omega_1\otimes\dots\otimes\omega_r)p(F_a(j)(e_0))\\
&=(\omega_1\otimes\dots\otimes\omega_r)F_b(j)(p(e_0))\\
&=(F_b(j)^*(\omega_1\otimes\dots\otimes \omega_r))p(e_0)\\
&=\sum_{\sigma\in S_r} (\omega_{\sigma(1)}\otimes\dots\otimes\omega_{\sigma(r)})  p(e_0)\\
&=\sum_{\sigma\in S_r} \int_{p,a}^b \omega_{\sigma(1)}\dots\omega_{\sigma(r)}.
\end{align*}
Here, the first equality uses $\Delta(p)=p\otimes p$, which translates into $p$ respecting tensor products.
\end{proof}


\section{Berkovich--Coleman integration}

We now explain our formalism for Berkovich--Coleman integration as integration along Frobenius-invariant paths.
Recall that  $\ell$ is an indeterminate standing in for $\Log(\pi)$, and we write it as a subscript to mean $\otimes K[\ell]$. 

\subsection{Properties of Berkovich--Coleman Integration}
Let $((X,M_X),(X,M_X),(\cP,L))$ be a proper log smooth frame, where $(X,M_X)$ is a weak log curve over $(S,\N)$.
Let $\Omega^1$ denote the vector space of log $1$-forms as in the previous section.
Let $F_a, F_b$ be fiber functors on $(X,M_X)$ attached to log points, anchored at components $\overline{a},\overline{b}\in V(\Gamma)$, and equipped with lifts $a,b$.
Let $\overline{p}\in \Pi(\Gamma; \overline{a}, \overline{b})$. By Proposition~\ref{prop:Frob-invariants}, there is a unique Frobenius-invariant $p\in \Pi((X,M);F_a,F_b)_{\ell} ^\varphi$ specializing to $\overline{p}$. For $\omega_1\dots\omega_r\in\Omega^1$, we define the {\em Berkovich--Coleman integral} to be
\[\BCint_{\overline{p},a}^b \omega_1\dots\omega_r\coloneqq\int_{p,a}^b \omega_1\dots\omega_r.\]

Berkovich--Coleman integration satisfies the natural properties of an integration theory according to the following theorem, an analogue of \cite[Theorem~9.1.1]{Berkovich:integration}. This theorem should be compared to the characterization of integration theories in \cite{KRZB}.

\begin{theorem} \label{t:bcprops} Let $F_a,F_b,F_c$ be fiber functors attached to points anchored on components $\overline{a},\overline{b},\overline{c}\in V(\Gamma)$ and equipped with lifts $a,b,c$. The Berkovich--Coleman integral has the following properties:
\begin{enumerate}
    \item \label{i:mulitlinearity} Berkovich--Coleman integration is multilinear in $1$-forms.
    \item \label{i:concatenation} Let $\overline{p}\in\Pi(\Gamma;\overline{a},\overline{b}), \overline{q}\in \Pi(\Gamma;\overline{b},\overline{c})$. Then
    \[\BCint_{\overline{pq},a}^c \omega_1\dots\omega_r=
    \sum_{k=0}^r\BCint_{\overline{p},a}^b \omega_{1}\dots \omega_{k}\BCint_{\overline{q},b}^c \omega_{k+1}\dots\omega_r.\]
   \item \label{i:functoriality} Let 
    $f\colon ((X,M_X),(X,M_X),(\cP,L))\to ((X',M_{X'}),(X',M_{X'}),(\cP',L'))$
    be a morphism of proper log smooth frames around weak log curves. Let $\omega'_1,\dots\omega'_r$ be $1$-forms on $]X'[_{\cP'}$. Let $\overline{p}\in\Pi(\Gamma;\overline{a},\overline{b})$. Then
  \[\int_{\overline{p},a}^b f^*\omega'_1\dots f^*\omega'_r=
  \int_{\overline{f}_*\overline{p},f\circ a}^{f\circ b} \omega'_1\dots\omega'_r.\]
   \item \label{i:symmetrization} If $F_a$ and $F_b$ are attached to smooth points, then the Berkovich--Coleman integral obeys the symmetrization relation: for $\overline{p}\in \pi_1(\Gamma;\overline{a},\overline{b})$
   \[\sum_{\sigma\in S_r} \BCint_{\overline{p},a}^b \omega_{\sigma(1)}\dots\omega_{\sigma(r)}=\prod_{i=1}^r\left( \BCint_{\overline{p},a}^b\omega_i\right).\]
   
   
   \item \label{i:goodreduction} Let $(X,M)$ be a weak log curve whose underyling scheme is smooth. Let $\overline{p}$ be the unique constant path on $\Gamma$. Then $\BCint_{\overline{p},a}^b \omega_1\dots\omega_r$ is given by integrating along
   the unique Frobenius-invariant path in $\pi_1((X,M);F_a,F_b)_{\ell}^\varphi$.

 \item \label{i:annularexample} Let $\overline{p}=e$ be an edge in the dual graph $\Gamma$, and let 
     $\iota\colon ((x,\overline{\M}_2),(x,\overline{\M}_2),(\cP_2,L_2))\to ((X,M_{X}),(X,M_{X}),(\cP,L))$
    be a morphism of a proper log smooth frame around a node (given as in Subsection~\ref{ss:logpoints}) to one around a weak log curve so that the node corresponds to $e=\overline{a}\overline{b}$. Let $\omega_1,\dots\omega_r$ be $1$-forms on $]X[_{\cP'}$. Then
  \[\int_{\overline{p},\iota_*F_{\{f_1\}}}^{\iota_*F_{\{f_2\}}} \omega_1\dots\omega_r=
  \int_{\delta_0} f^*\omega'_1\dots f^*\omega'_r\]
  where $\delta_0$ is the path in Example~\ref{e:pi1annulus}. 
    \item \label{i:antiderivative} Let $F_a,F_b$ be fiber functors attached to the same smooth $k$-point in $X$, and let $\overline{p}$ be the constant path at $\overline{a}=\overline{b}$. Then $\int_{\overline{p},a}^b \omega_1\dots\omega_r$ is given by 
    classical parallel transport as in Remark~\ref{r:paralleltransport}.
\end{enumerate}
\end{theorem}

\begin{proof}
  Property (\ref{i:mulitlinearity}) follows from Lemma~\ref{l:multilinear}.

  For (\ref{i:concatenation}), note that the Frobenius-invariant lift of $\overline{pq}$ is equal to $pq$. The equality follows from Proposition~\ref{p:concatenation}.
  
  To prove (\ref{i:functoriality}), we note that the Frobenius-invariant lift of $\overline{f}_*\overline{p}$ is $f_*p$ and then apply  Proposition~\ref{p:functoriality}.
  
  The symmetrization relation (\ref{i:symmetrization}) follows from Proposition~\ref{p:symmetrization} together with the observation that if $\overline{p}\in\pi_1(\Gamma;a,b)$ then $p\in \pi^{\rig,\un}_1((X,M);F_a,F_b)_{\ell}^\varphi$. 
 
Property (\ref{i:goodreduction}) follows from definitions.

  For (\ref{i:annularexample}), it suffices to show $\iota_*\delta_0=p$, the Frobenius invariant lift of $\overline{p}=e$. Clearly, $\iota_*\delta_0$ is Frobenius invariant. Moreover, by functoriality of specialization, $\iota_*\delta_0$ must specialize to an element of $\pi_1(\Gamma_X;\overline{a},\overline{b})$ that is in the image of $\pi_1(\Gamma_{(X',M')};\overline{a},\overline{b})\to \pi_1(\Gamma_{(X,M)};\overline{a},\overline{b})$ for any log morphism $(X',M')\to (X,M)$ with the node corresponding to $e$ in its image. By considering the normalization at all nodes except for $e$, $(\tilde{X},M_{\tilde{X}})\to(X,M)$, we see $\on{sp}(\iota_*\delta_0)=\overline{p}$. The conclusion follows from Proposition~\ref{prop:Frob-invariants}.
  
  For (\ref{i:antiderivative}), suppose $a,b\in \cP(K)$ specialize to $a_0\in X(k)$. Set $\cP'=\Spf V\ps{x_1}$, let $(\cP',L')$ 
 be the formal residue disc as in Definition~\ref{d:smoothpoint}, and choose the frame 
  \[((x,\overline{\M}_0),(x,\overline{\M}_0),(\cP',L')).\]
  By \cite[Proposition~2.2]{BL:stable1}, there is a morphism $g\colon\cP'\to\cP$ parameterizing the residue disc around $a_0$ and we may  treat $F_a$ and $F_b$ as fiber functors on $(x,\overline{\M}_0)$. Then, 
  $g^*\cE_{\omega_1\dots\omega_r}=\cE_{g^*\omega_1\dots g^*\omega_r}$.
  We can produce a horizontal morphism
  \[s\colon\mathbf{1}\to E_{g^*\omega_1\dots g^*\omega_r}\]
  with $s(1)_a=e_0$ by iterated anti-differentiation.
  Let $q\in\pi_1^{\rig,\un}((x,\overline{\M}_0);F_a,F_b)$ be 
  the unique (and thus Frobenius-invariant) path. Because $g_*(q)=p$, the 
  conclusion follows from (\ref{i:functoriality}) and 
  Remark~\ref{r:paralleltransport}.
\end{proof}

\begin{remark}
  There is an analogous parallel transport property on nodal log points making use of $\Log$-analytic functions. Together with the above properties, it can be shown to characterize Berkovich--Coleman integration. Here, one uses (\ref{i:concatenation}) and (\ref{i:functoriality}) to reduce to the case of curves with smooth underlying schemes and nodal log points. The former is covered  by (\ref{i:goodreduction}) and the latter by the parallel transport property.
  \end{remark}

\begin{remark}
  One can define locally analytic functions from the above theory of integration by using Besser's language of {\em abstract Coleman functions} \cite[Section~4]{Besser:Coleman}. Besser's functions are equivalent to Coleman's for log curves $(X,M)$ with only smooth points and punctures. The functions arising from our definition of Berkovich--Coleman integration are equivalent to the functions defined in \cite{Coleman-deShalit} and \cite{Berkovich:integration}. Indeed, a path $\overline{p}$ in $\Gamma$ can be decomposed into edges. Consequently, its Frobenius-invariant lift can be decomposed into a concatenation of paths of the following kind:
  \begin{enumerate}
      \item $i_{v*}p_v\in \pi_1^{\rig,\un}((X,M);F_1,F_2)_\ell^\varphi$
      where $p_v\in \pi_1^{\rig,\un}((X_v,M_v);F_1,F_2)_{\ell}^\varphi,\\$
      $i_{\overline{v}}\colon (X_{\overline{v}},M_{\overline{v}})\to (X,M)$
      is the inclusion of a component $X_{\overline{v}}$ of $X$ equipped with the induced log 
      structure, and $F_1$ and $F_2$ are fiber functors attached to log 
      points anchored at $\overline{v}$; and
      \item $i_{e*}\delta_0\in\pi_1^{\rig,\un}((x,\overline{\M}_2);i_{e*}F_{\{f_1\}},i_{\overline{e}*}F_{\{f_2\}})_{\ell}^\varphi$ where $i_e\colon (x,\overline{\M}_2)\to (X,M)$ is the inclusion of a nodal point.
  \end{enumerate}
  We need only verify equivalence for the two kinds of paths by Property  (\ref{i:functoriality}) of Theorem~\ref{t:bcprops}.
  For the first kind, let $j_{\overline{v}}\colon (X_{\overline{v}},M_{\overline{v}})\to (X_{\overline{v}},M'_{\overline{v}})$ be the morphism given by Remark~\ref{r:modifylogstructure} where annular points are replaced by punctures. The conclusion follows from Property (\ref{i:functoriality}) and \cite[Section~5]{Besser:Coleman}. For the second kind, we employ Example~\ref{e:annulusintegral}.
\end{remark}
\begin{remark}
Our use of the bundle $\cE_{\omega_1\dots\omega_n}$ can be interpreted more systematically by considering the Hodge filtration on the de Rham fundamental group. 

Let $(X,M_X)$ be a proper log smooth curve, and let 
$\cX$ be the fiber of $\cP$ over $\Spec K$ and $\cD$ be the divisor on $\cX$ corresponding to punctures on $(\cP,L)$. Let 
 $F_a,F_b$ be fiber functors given by taking the fiber at $a,b\in \cX(K)$. Let $\mathscr{I}$ be the augmentation ideal of the unipotent de Rham module $\Pi^{\dR}((\cX,L); F_a, F_b)$ and $F$ the Hodge filtration. Then there is a natural identification 
 \[F^{r}\left((\Pi^{\dR}((\cX,L); F_a,F_b)/\mathscr{I}^{r+1})^\vee\right)
 \cong H^0(\cX, \Omega^1_X(\log D))^{\otimes r}\]
Indeed, the result is trivial for $r=0$ and for for $r=1$ is immediate from the  identification \[\mathscr{I}/\mathscr{I}^2\cong H^1_{\dR}((\cX,L))^\vee.\] 
The general case follows by multiplicativity. In the language of universal bundles, the functional given by \[\omega_1\otimes\dots\otimes \omega_r\in H^0(\cX, \Omega^1_X(\log D))^{\otimes r}\]
is induced by the projection $\cE_r\to \cE_r^{\Omega}$ followed by the functional $\omega_1\otimes\dots\otimes\omega_{r}$ as constructed in the proof of Proposition~\ref{p:symmetrization}.

Let $\overline{p}\in \Pi(\Gamma, \overline{a}, \overline{b})$. Then, the Berkovich-Coleman integral of $\omega_1\otimes \cdots\otimes \omega_r$ along $\overline{p}$ is defined as follows. Let $p$ be the unique lift of $\overline{p}$ to $\Pi((X,M), F_a, F_b)_{\ell}^{\varphi}$, and let $\on{comp}(p)\in \Pi^{\dR}((\cX,L), F_a,F_b)_{\ell}$ be its image under the comparison map and $\overline{\on{comp}(p)}$ its reduction modulo $\mathscr{I}^{r+1}$. The Berkovich-Coleman integral is equal to the natural pairing
$\langle \overline{\on{comp}(p)}, \omega_1\otimes \cdots\otimes \omega_r\rangle.$

\end{remark}


\begin{example} \label{ex:twicedpunctureP1}
Let $(X,M)$ be $\P_k^1$ with punctures at $0$ and $\infty$. This has a natural frame given by 
\[((X,M),(X,M),(\cP,L))\]
where $\cP=\hat{\P}_V^1$ with $L$ induced by punctures at $0$ and $\infty$ (given 
in an inhomogeneous coordinate $z$). Because $]X[_{\cP}=(\P_K^1)^{\an}$, the log 
rigid cohomology of $(X,M)$ corresponds to the de Rham cohomology of 
$((\P_K^1)^{\an},L)$ which is spanned by $\omega=\frac{dz}{z}$. Consequently, 
$\cE_{\omega^r}$ evaluated on our frame is the trivial bundle with fiber
$K[c]/c^{n+1}$ and connection $1$-form $\frac{dz}{z}(c\ \cdot\ )$ as in Example~\ref{e:annulusintegral}.

Let $F_0$ and $F_\infty$ be the log basepoints at the punctures $0$ and $\infty$ respectively (each given by basis $\{f\}$). Then
$h_0\in F_0(\cE_{\omega^r})$ is given by the function $\exp(\Log(z)c)$.
Let $\overline{p}$ be the constant path in $\Gamma$ and $p\in\pi_1^{\rig,\un}((X,M),F_0,F_\infty)_{\ell}^\varphi$ be its Frobenius-invariant lift.
By a computation directly analogous to Proposition~\ref{p:nodalfrobenius} or by making use of the lift of Frobenius $z\mapsto z^q$, we see that the unique Frobenius-invariant path from $F_0$ to $F_\infty$ must be the substitution $\Log(z)\mapsto -\Log(w)$ where $w=z^{-1}$ is the coordinate at $\infty$. Consequently,
$p(h_{0,0})=\exp(-\Log(w)c)$. Since $h_{\infty,i}=c^i\exp(-\Log(w)c)$, and 
\[\int_{\overline{p},0}^\infty\omega^r=\begin{cases}
1 &\text{if }r=0\\
0 &\text{if }r\geq 1.
\end{cases}\]
\end{example}

\begin{example}
Let $(X',M')$ be the  log curve where $X$ is the nodal elliptic curve coming from identifying $0$ and $\infty$ on $\P^1$, and $M'$ is the induced by the nodal log structure at that point. Let $F_0$ and $F_{\infty}$ be the log base-points attached to the branches containing $0$ and $\infty$, respectively. These base-points possess a lift given by the coordinates $z$ and $w=z^{-1}$ This curve is the closed fiber of a semistable model of the Tate curve $E_\pi=\Gm/\pi^\Z$ and has a regular $1$-form given by $\nu=\frac{dz}{z}$. The log dual graph is a vertex with a loop. Let $\overline{p}$ be
a closed path going around that loop (oriented from the branch of the node near $\infty$ to the branch of the node near $0$). Let $p\in \pi_1((X',M');F_0,F_0)_{\ell}^{\varphi}$ be its Frobenius-invariant lift.

Let $(X'',M'')$ be $\P^1_k$ with annular points at $0$ and $\infty$. The log dual graph is a point. Let $F''_0$ and $F''_\infty$ be the corresponding log base points. By Lemma~\ref{l:aphomotopyequivalence}, we have an Frobenius-equivariant isomorphism with the fundamental group of the curve in Example~\ref{ex:twicedpunctureP1}
\[\pi^{\un}_1((X'',M'');F''_0,F''_\infty)\cong
\pi^{\un}_1((X,M);F_0,F_\infty).\]
Consequently, if $p''$ is the Frobenius-invariant path from $F''_0$ to $F''_\infty$, then 
\[\int_{p'',0}^\infty \nu^r=\begin{cases}
1 &\text{if }r=0\\
0 &\text{if }r\geq 1.
\end{cases}\]

Let $g\colon (x,\overline{\M}_2)\to (X',M')$ be the inclusion of the nodal log point such that $f_1\in \overline{M}_2$ corresponds to the branch of $(X',M')$ at $\infty$ and $f_2\in \overline{M}_2$ corresponds to the branch of $(X',M')$ at $0$. Then 
$\int_{\delta_0,\infty}^0 g^*\nu^r=\ell^r/r!$ by Property (\ref{i:annularexample}) of Theorem~\ref{t:bcprops} and Example~\ref{e:annulusintegral}.
Because $p$ is the composition of $p''$ and $g_*\delta_0$, by the concatenation formula,
\[\BCint_p \nu^r=\frac{\ell^r}{r!}.\]
This is consistent with the period of the Tate curve $E_\pi$ being $\Log(\pi)=\ell$ and the symmetrization formula.
\end{example}

\begin{example} \label{e:longertatecurve}
We can modify the above example by considering a cycle of $m$ $\P^1$'s. It has a lift given by the Tate curve $E_{\pi^m}=\Gm/\pi^{m\Z}$. By composing paths through the $\P^1$'s, the above arguments show that
\[\BCint_{\overline{p}} \nu^n=\frac{m^n\ell^n}{n!}.\]
where $\overline{p}$ is a (suitably oriented) closed path around the dual graph $\Gamma$.
\end{example}

\section{Vologodsky Integration}

\subsection{Definition of the Vologodsky integral}

\begin{definition}
Let $((X,M_X),(X,M_X),(\cP,L))$ be a proper log smooth frame where $(X,M_X)$ is a weak log curve. 
Let $F_a, F_b$ be fiber functors on $(X,M_X)$ attached to log points anchored at components $\overline{a},\overline{b}\in V(\Gamma)$ and equipped with lifts $a$ and $b$. For $\omega_1\dots\omega_r\in\Omega^1$, the 
{\em Vologodsky integral} is
\[\Vint_a^b \omega_1\dots\omega_r\coloneqq\BCint_{\overline{p}_{\overline{a}\overline{b}},a}^b \omega_1\dots\omega_r\]
where $\overline{p}_{\overline{a}\overline{b}}$ is the combinatorial canonical path from $\overline{a}$ to $\overline{b}$ from Proposition~\ref{p:canonicalpath}.
\end{definition}

This notion of integral is path independent and has the following properties:

\begin{theorem}  Let $F_a,F_b,F_c$ be fiber functors attached to points anchored on components $\overline{a},\overline{b},\overline{c}\in V(\Gamma)$ equipped with lifts $a,b,c$. The Vologodsky integral has the following properties:
\begin{enumerate}
    \item \label{i:Vmulitlinearity} Vologodsky integration is multilinear in $1$-forms.
    \item \label{i:Vconcatenation} 
    \[\Vint_a^c \omega_1\dots\omega_r=
    \sum_{k=0}^r\Vint_a^b \omega_{1}\dots \omega_{k}\Vint_b^c \omega_{k+1}\dots\omega_r.\]
    \item \label{i:Vsamecomponent} If $\overline{a}=\overline{b}$, then 
    \[\Vint_a^b \omega_1\dots\omega_r=\BCint_{\overline{p},a}^b \omega_1\dots\omega_r\]
    where $\overline{p}$ is the constant path at $\overline{a}$.
    \item \label{i:Vsymmetrization} If $F_a$ and $F_b$ are attached to smooth points, then the Vologodsky integral obeys the symmetrization relation
      \[\sum_{\sigma\in S_r} \Vint_a^b \omega_{\sigma(1)}\dots\omega_{\sigma(r)}=\prod_{i=1}^r \Vint_a^b\omega_i.\]
 \end{enumerate}
\end{theorem}

\begin{proof}
  All of these are immediate from the analogous properties of Berkovich--Coleman integration except the following: (\ref{i:Vconcatenation}) requires the concatenation property of the combinatorial canonical path; and (\ref{i:Vsymmetrization}) requires the combinatorial canonical path to be group-like. These follow from Lemma~\ref{p:canonicalpath}.
\end{proof}

  Because the combinatorial canonical path is not necessarily functorial under morphisms of graphs as noted in Remark~\ref{r:canonicalfunctorial}, the Vologodsky integral is not necessarily functorial under morphisms of frames. 

\subsection{Monodromy and the comparison with Vologodsky's definition} \label{ss:monodromycomparison}
We now explain an alternative description of Vologodsky iterated integration, which is implicit in \cite{Vologodsky} where single integrals are described in detail.
\begin{proposition}[Vologodsky path]
Let $(X, M), F_a, F_b$ be as in the previous subsection. Then there is a unique element $p_{\text{Vol}}$ of $\Pi((X,M);F_a, F_b)_{\ell}$ such that
\begin{enumerate}
    \item $p_{\text{Vol}}\equiv 1\bmod \mathscr{I}$
    \item $p_{\text{Vol}}$ is fixed by $\varphi$, and
    \item For all $n>0$, $N^n(p_{\text{Vol}})=0 \bmod W_{-n-1}$.
\end{enumerate}
Vologodsky defined $\Vint_a^b \omega_1\dots\omega_r$ to be $\int_{p_{\text{Vol}},a}^b \omega_1\dots\omega_r$.
\end{proposition}

In this section, we will connect our definition with Vologodsky's. This will involve studying the action of the monodromy operator on paths in $\Pi((X,M);F_a,F_b)$.

\begin{definition}
Let $((X,M_X),(X,M_X),(\cP,L))$ be a proper log smooth frame where $(X,M_X)$ is a weak log curve.  Let $\cE$ be a unipotent isocrystal on $(X,M)$
equipped with a filtration
\[\cE=\cE^0\supset \cE^1\supset \dots \supset \cE^{n+1}=0\]
where we pick isomorphisms $\cE^i/\cE^{i+1}\cong \mathbf{1}^{\oplus n_i}$.
Evaluate $\cE$ on the frame to obtain a unipotent vector bundle with integrable connection $(E,\nabla)$ on $]X[_{\cP}$ with filtration $\{E^i\}$. The induced unipotent connection on $E^i/E^{i+2}$ can be interpreted as a matrix of $1$-forms
\[\omega_{i+1}\in\Omega^1\otimes \on{Hom}(E^i/E^{i+1},E^{i+1}/E^{i+2})\cong\Omega^1\otimes \on{Hom}(\cO^{\oplus n_i},\cO^{\oplus n_{i+1}}).\]
For each puncture, node, or annular point $(x,\overline{\M})\to (X,M)$ corresponding to an edge $e$ in the log dual graph,  by \cite[Section~2]{Coleman:RLC}  there is a residue $\eta_i(e)$ 
\[\eta_i(e)=\Res_{(x,\overline{\M})}(\omega_i)\in \on{Hom}(K^{\oplus n_i},K^{\oplus n_{i+1}}).\]
This residue can also be computed as the coefficient of $\frac{dx_1}{x_1}$ as in Section~\ref{ss:monodromynode}. 
By the residue theorem \cite[Proposition~4.3]{Coleman:RLC}, for each component $\overline{a}\in V(\Gamma)$,
\[\sum_{(x,\overline{\M})}\Res_{(x,\overline{\M})}(\omega_i)=0\]
where the sum is over punctures, nodes, and annular points anchored on $\overline{a}$. Consequently, $\eta_i$ is a matrix of tropical $1$-form 
\[\eta_i\in \Omega^1(\Gamma)\otimes \on{Hom}(K^{n_i},K^{n_{i+1}}).\]
The {\em tropical multiform} attached to the above data is the composition
\[\eta_1\dots\eta_n\in\Omega^1(\Gamma)^{\otimes n}\otimes \on{Hom}(K^{n_0},K^{n_n}).\]
\end{definition}


The following is a generalization of the combinatorial description of the monodromy pairing as in Proposition~\ref{p:monodromycohomology}.

\begin{proposition} \label{p:combinatorialmonodromy}
Let $F_a, F_b$ be fiber functors on $(X,M_X)$ attached to log points and anchored at components $\overline{a},\overline{b}\in V(\Gamma)$. Let $p$ be the Frobenius-invariant lift of some $\overline{p}\in\Pi(\Gamma;\overline{a},\overline{b})$. Let $\cE$ be a unipotent isocrystal of index of unipotency $n$. Then,
\begin{enumerate}
    \item the linear transformation $N^kp$  factors as 
\[N^kp\colon F_a(\cE)\to F_a(\cE^0/\cE^{n+1-k})\to F_b(\cE^k)\to F_b(\cE);\]
    \item for $n=k$, the linear map $N^np\colon K^{n_0}=F_a(\cE^0/\cE^1)\to F_b(\cE^n)=K^{n_n}$ in the above factorization is given by multiplication by
$n!\left(\cint_{\overline{p}} \eta_1\dots\eta_n\right)$
where $\eta_1\dots\eta_n$ is the tropical multiform attached to $\cE$.
\end{enumerate}
\end{proposition}

\begin{proof}
We first consider the case where  $\overline{p}=e$ is the edge corresponding to a node, and $F_a=F_{\{f_1\}}$, $F_b=F_{\{f_2\}}$ at that node. Since $p=i_{e*}\delta_0$ where $i_{e}\colon (x,\overline{\M}_2)\to (X,M)$ is the inclusion of the node, we may work on $(x,\overline{\M}_2)$.
By Lemma~\ref{l:monodromyresidue}, because $\cE^{n+1-k}$ has index of unipotency $k-1$, $N^k\delta_0$ is vanishes on $F_{\{f_1\}}(\cE^{n+1-k})$. Consequently, $N^k\delta_0$ factors through $F_{f_1}(\cE/\cE^{n+1-k})$.
Similarly, because $N^k\delta_0\colon F_{\{f_1\}}(\cE/\cE^k)\to F_{\{f_2\}}(\cE/\cE^k)$ is zero, $N^k\delta_0$ factors through $F_{\{f_2\}}(\cE^k)$.
Suppose $n=k$. By Lemma~\ref{l:monodromyresidue}, on $F_{\{f_1\}}(\cE)$, $N^n\delta_0$ is given by multiplication by \[\Res_{(x,\overline{\M}_2)}(\omega_1)\dots\Res_{(x,\overline{\M}_2)}(\omega_n)=\int_e \eta_1\dots\eta_n.\]

In the general case, it suffices to prove the result for $\overline{p}\in\pi_1(\Gamma;a,b)$.
We induct on the length of the path $\overline{p}$. Consider the case where $\overline{p}$ is the constant path at a vertex $\overline{a}=\overline{b}$. Let $(X_a,M_a)$ be the component of $X$ corresponding to the vertex $a$ equipped with the log structure induced from $X$. Then, the points corresponding to nodes on $X$ become annular points on $(X_a,M_a)$. There is a natural log morphism $i\colon (X_a,M_a)\to (X,M)$ and  fiber functors $F'_a$ and $F'_b$ on $(X_a,M_a)$ such that $F_a=i_*F'_a$ and $F_b=i_*F'_b$. Moreover, by the functoriality of the Frobenius action, the Frobenius-invariant path $p'\in\pi_1^{\rig,\un}((X_a,M_a);F'_a,F'_b)$ lifting $\overline{p}$ satisfies $i_*p'=p$. By Lemma~\ref{l:trivialmonodromy} applied to $(X_a,M_a)$, $Np'=0$. By the functoriality of monodromy, $Np=0$.

Now, write $\overline{p}=\overline{p}_1e$ where $e=\overline{v}\overline{b}$ is an edge and $\overline{p}_1$ is shorter than $\overline{p}$. Let $F_{\{f_1\}}$ (resp.~$F_{\{f_2\}}$) be the fiber functor anchored at $\overline{v}$ (resp.~$\overline{b}$) corresponding to the node indexed by $e$. 
Then $p=p_1(i_{e*}\delta_0)q$ where $p_1,\delta_0,q$ are Frobenius-invariant paths with
\begin{enumerate}
    \item $p_1\in\Pi((X,M);F_a,F_{\{f_1\}})_{\ell}^\varphi$ lifting $\overline{p}_1$,
    \item $\delta_0\in\Pi((x,\overline{\M}_2);F_{\{f_1\}},F_{\{f_2\}})_{\ell}^\varphi$ and 
    $i_e\colon (x,\overline{\M}_2)\to (X,M)$ parameterizing the node corresponding to $e$, and
    \item $q\in\Pi((X,M);F_{\{f_2\}},F_b)_{\ell}^\varphi$ lifting the constant path at $\overline{b}$.
\end{enumerate}
Since $Nq=0$, 
\[N^kp=\sum_{i=0}^k \binom{k}{i} (N^ip_1)(N^{k-i}\delta_{0})q.\]
Now, we may factor $N^ip_1$ and $N^{k-i}\delta_0$ as follows:
%
\[    N^ip_1\colon F_a(\cE/\cE^{n+1-(k-i)})\to F_a(\cE/\cE^{n+1-k})\to F_{\{f_1\}}(\cE^i/\cE^{n+1-(k-i)})\to F_{\{f_1\}}(\cE/\cE^{n+1-(k-i)})\]
\[
    N^{k-i}\delta_{0}\colon F_{\{f_1\}}(\cE^i)\to F_{\{f_1\}}(\cE^i/\cE^{n+1-(k-i)}) \to F_{\{f_2\}}(\cE^k)\to F_{\{f_2\}}(\cE).
\]
Therefore, $N^kp$ factors as described.

When $n=k$, by induction, the middle arrow in the first row above, $F_a(\cE/\cE^{n+1-i})\to F_{\{f_1\}}(\cE^i/\cE^{i+1})$
acts as multiplication by
$i!\cint_{\overline{p}_1}\nu_1\dots\nu_i.$
By the special case, the middle arrow in the second row, $F_{\{f_1\}}(\cE^i/\cE^{i+1})\to F_{\{f_2\}}(\cE^n)$
acts as multiplication by
\[\nu_{i+1}(e)\dots\nu_n(e)=(n-i)!\cint_e \nu_{i+1}\dots\nu_n.\]
The conclusion follows from the concatenation formula for combinatorial iterated integrals.

\end{proof}

We expect this formula for $N^np$ to extend to arbitrary elements $p\in\Pi((X,M_X);F_a,F_b)$ by a generalized Seifert--van Kampen theorem \cite{Stix:svk} to decompose elements of $\pi^{\rig,\un}_1((X,M_X);F_a,F_b)$ into paths in components and nodal points.
We have the following conclusion for proper log curves without punctures:

\begin{corollary} \label{c:vanishingmonodromy}
Let $(X,M)$ be a proper log curve without punctures. Let $F_a,F_b$ be fiber 
functors attached to log points anchored at $\overline{a},\overline{b}\in V(\Gamma)$. 
Let $p\in \pi_1^{\rig,\un}((X,M);F_a,F_b)_\ell^\varphi$ be the Frobenius-invariant 
lift of the canonical path 
$\overline{p}_{\overline{a}\overline{b}}\in\pi_1^{\un}(\Gamma;\overline{a},\overline{b})$. 
For $n\in\N$, $N^np\colon F_a(\cE)\to F_b(\cE)$ is the zero map for all unipotent isocrystals of index $\cE$ of unipotency less than or equal to $n$.
\end{corollary}


\begin{theorem} Let $(X,M)$ be a proper log curve. Then, Vologodsky integration satisfies Vologodsky's characterization.
\end{theorem}

\begin{proof}
Write $p\in \pi_1^{\rig,\un}((X,M);F_a,F_b)_{\ell}^\varphi$ for the Frobenius-invariant lift of the combinatorial canonical path $\overline{p}_{\overline{a}{b}}$. It clearly satisfies the first two properties of Vologodsky's characterization.

 If $(X,M)$ does not have any punctures, the third property is a consequence of Corollary~\ref{c:vanishingmonodromy} as 
 \[N^np=0 \bmod  W_{-n-1}=\mathscr{I}^{n+1}.\]
 Now, let $(X,M')$ be obtained from $(X,M)$ by replacing punctures by smooth points as in Remark~\ref{r:modifylogstructure}. Let $f\colon (X,M)\to (X,M')$ be the canonical morphism which induces $\varrho\colon \Gamma_X\to \overline{\Gamma}_X$ as in Remark~\ref{r:contracttocore}.  
 Write $p'\in \pi_1^{\rig,\un}((X,M');f_*F_a,f_*F_b)^\varphi_{\ell}$ for the Frobenius-invariant lift of the canonical path from $\overline{a}$ to $\overline{b}$ on $\overline{\Gamma}_X$.
 By the functoriality of the Frobenius action and the specialization map, $f_*p=p'$.
 Because in $\Pi((X,M');F_a,F_b)_{\ell}$,
 \[f_*(N^np)=N^n(f_*p)=N^np'=0 \bmod \mathscr{I}^{n+1},\]
 and $N^np\in M_{-2n}\Pi((X,M);F_a,F_b)_{\ell}$, we can conclude
 $N^np\in W_{-n-1}$ from Corollary~\ref{c:weightsubmodule}.
\end{proof}


\subsection{A formula for Vologodsky integrals}
We can give a description of Vologodsky integrals in terms of periods, i.e., Berkovich--Coleman integrals along closed loops. 
Some of these arguments emerged out of a discussion with Amnon Besser. The approach using completed tensor algebras was inspired by \cite{Chen:BHformula}.

Let $F_a,F_b$ be a fiber functors attached to log points anchored at $\overline{a},\overline{b}\in V(\Gamma)$ equipped with lifts $a,b$. Combinatorial and Berkovich--Coleman integration induces linear maps to completed tensor algebras:
\[    C_{\overline{a}}^{\overline{b}}\colon \Pi(\Gamma;\overline{a},\overline{b})\to \hat{T}(\Omega^1(\overline{\Gamma})^\vee),\quad
    \BC_a^b\colon \Pi(\Gamma;\overline{a},\overline{b})\to \hat{T}((\Omega^1)^\vee\otimes K[\ell]).
\]
In other words, for $\overline{p}\in\Pi(\Gamma;\overline{a},\overline{b})$, write $C_{\overline{a}}^{\overline{b}}(\overline{p})\in \hat{T}(\Omega^1(\overline{\Gamma})^\vee)$ for the element which when evaluated on $\eta_1\otimes\dots\otimes\eta_n$ gives $\cint_{\overline{p}}\eta_1\dots\eta_n$ and similarly for $\BC_a^b(\gamma)$. As a consequence of concatenation formulas, these maps are groupoid homomorphisms in that
\[C_{\overline{a}}^{\overline{c}}(\overline{p}\overline{q})=C_{\overline{a}}^{\overline{b}}(\overline{p})C_{\overline{b}}^{\overline{c}}(\overline{q}),\quad 
\BC_a^c(\overline{p}\overline{q})=\BC_a^b(\overline{p})\BC_b^c(\overline{q})\]
for $\overline{p}\in\Pi(\Gamma;\overline{a},\overline{b})$ and $\overline{q}\in\Pi(\Gamma;\overline{b},\overline{c})$.
Similarly, we may write $V_a^b\in \hat{T}((\Omega^1)^\vee\otimes K[\ell])$ for Vologodsky integration.
The map $C_{\overline{a}}^{\overline{a}}$ is full-rank as a map between the respective truncations by Proposition~\ref{p:perfectpairing}. For $\overline{p}\in\pi_1(\Gamma;\overline{a},\overline{b})$, $BC_a^b(\overline{p})$ is an invertible element of $\hat{T}((\Omega^1)^\vee\otimes K[\ell])$.

\begin{theorem} \label{t:vologodskyformula}
Let $F_a, F_b$ be fiber functors on $(X,M_X)$ attached to log points anchored at components $\overline{a},\overline{b}\in V(\Gamma)$ and equipped with lifts $a,b$.
Let $\overline{p}$ be a path in $\Gamma$ from $\overline{a}$ to $\overline{b}$. Then,
\[V_{\overline{a}}^{\overline{b}}=\BC_{a,a}\left(
(C_{\overline{a}}^{\overline{a}})^{-1}\left(\frac{1}{C_{\overline{a}}^{\overline{b}}(\overline{p})}\right)\right)\BC_{a,b}(\overline{p})\]
where the reciprocal is taken in $\hat{T}(\Omega^1(\overline{\Gamma})^\vee)$.
\end{theorem}

\begin{proof}
Because 
\[C_{\overline{a}}^{\overline{b}}\left((C_{\overline{a}}^{\overline{a}})^{-1}\left(\frac{1}{C_{\overline{a}}^{\overline{b}}(\overline{p})}\right)\overline{p}\right)
=
C_{\overline{a}}^{\overline{a}}\left((C_{\overline{a}}^{\overline{a}})^{-1}\left(\frac{1}{C_{\overline{a}}^{\overline{b}}(\overline{p})}\right)\right)C_{\overline{a}}^{\overline{b}}(\overline{p})
=
1,\]
$(C_{\overline{a}}^{\overline{a}})^{-1}\left(\frac{1}{C_{\overline{a}}^{\overline{b}}(\overline{p})}\right)\overline{p}$ is the combinatorial canonical path from $a$ to $b$. We apply $\BC$ to it to obtain the Vologodsky integral, and then use the concatenation formula.
\end{proof}

By writing
\begin{equation} \label{eq:canonicalformula} \frac{1}{C_{\overline{a}}^{\overline{b}}(\overline{p})}=\frac{1}{1+(C_{\overline{a}}^{\overline{b}}(\overline{p})-1)}=1-(C_{\overline{a}}^{\overline{b}}(\overline{p})-1)+(C_{\overline{a}}^{\overline{b}}(\overline{p})-1)^2+\dots,
\end{equation}
the above theorem gives a way to compute Vologodsky integrals, which in the case of single integrals becomes the following:

\begin{corollary} \label{c:vsingleint}
Let $F_a, F_b$ be fiber functors on $(X,M_X)$ attached to log points anchored at components $\overline{a},\overline{b}\in V(\Gamma)$ and equipped with lifts $a, b$.
Let $\overline{C}_1,\dots,\overline{C}_h\in H_1(\Gamma;K)$ and 
$\eta_1,\dots,\eta_h\in\Omega^1(\Gamma)$ be dual 
bases with respect to single combinatorial 
integration. Let 
$\overline{\gamma}_1,\dots,\overline{\gamma}_h\in\pi_1^{\un}(\Gamma,a)$ be loops whose homology classes are 
$\overline{C}_1,\dots,\overline{C}_h$, respectively.
Pick a path 
$\overline{p}$ in $\Gamma$ from $\overline{a}$ to $\overline{b}$.
Then,
\[\Vint_a^b\omega=\BCint_{\overline{p},a}^b\omega-\sum_i \left(\BCint_{\overline{\gamma}_i,a}^a \omega\right)\left(\cint_{\overline{p}} \eta_i\right)\]
for every $\omega\in\Omega^1$.
\end{corollary}

\begin{proof}
In the truncation $\Pi(\Gamma,\overline{a})/\mathscr{I}_{\overline{a}}^2$, we have
\[(C_{\overline{a}}^{\overline{a}})^{-1}(1/C_{\overline{a}}^{\overline{b}}(\overline{p}))=1-\sum_i \left(\cint_{\overline{p}}\eta_i\right)\left(\overline{C}_i-1\right).\]
We obtain the conclusion by applying $\BC_a^a$, multiplying on the right by $\BC_a^b(\overline{p})$, and evaluating on $\omega$.
\end{proof}

This Corollary allows us to identify Vologodsky integrals with abelian integrals defined through the Lie group structure on the Jacobian of a curve \cite{Zarhin1996}. Indeed, the formula for abelian integrals in terms of Berkovich--Coleman integrals proved in \cite{KRZB} and rephrased as \cite[Theorem~3.16]{KatzKaya} is identical to the above formula.

\begin{example} \label{ex:vexample}
Let $(X,M)$ be as in Example~\ref{e:longertatecurve}. Let $E=\C_p^*/\pi^{m\Z}$ and $\nu=\frac{dz}{z}.$ Identify $\Gamma$ with $\mathbb{R}/m\Z$ with vertices at every integer.
Let $F_a$ and $F_b$ be fiber functors attached to smooth points $a,b\in E(K)$ which are anchored at $\overline{a},\overline{b}\in V(\Gamma)$.
Choose $\tilde{\overline{a}},\tilde{\overline{b}}\in\Z$ with reduction mod $m$ equal to $\overline{a},\overline{b}$, respectively. Let $\overline{p}$ be the path in $\Gamma$ whose lift to $\tilde{\Gamma}=\R$ linearly interpolates from $\tilde{\overline{a}}$ to $\tilde{\overline{b}}$.
Lift $a,b$ to $\tilde{a},\tilde{b}\in \C_p^*$ with $\val(\tilde{a})=\tilde{\overline{a}}\val(\pi)$ and $\val(\tilde{b})=\tilde{\overline{b}}\val(\pi)$. Then,
$\BCint_{\overline{p},a}^b \nu=\Log(\tilde{b})-\Log(\tilde{a})$
where $\Log$ is the branch with $\Log(\pi)=\ell$. For the tropical $1$-form $\eta=dt$ on $\Gamma$, $\cint_{\overline{p}} \eta = \tilde{\overline{b}}-\tilde{\overline{a}}.$
For the closed loop $\overline{\gamma}$ from $\overline{a}$ to $\overline{a}+m$ in $\Gamma$, $\BCint_{\overline{\gamma},a}^a \nu=m\ell.$
Therefore, by Corollary~\ref{c:vsingleint}
\[\Vint_a^b \nu=\Log(\tilde{b})-\Log(\tilde{a})-\ell(\tilde{\overline{b}}-\tilde{\overline{a}}).\]
By the symmetrization formula,
\[\Vint_a^b \nu^n=\frac{1}{n!}\left(\Log(\tilde{b})-\Log(\tilde{a})-\ell(\tilde{\overline{b}}-\tilde{\overline{a}})\right)^n.\]
\end{example}

\bibliographystyle{plain}
\bibliography{references}

\begin{thebibliography}{10}

\bibitem{SGA7-1}
{\em Groupes de monodromie en g\'{e}om\'{e}trie alg\'{e}brique. {I}}.
\newblock Lecture Notes in Mathematics, Vol. 288. Springer-Verlag, Berlin-New
  York, 1972.
\newblock S\'{e}minaire de G\'{e}om\'{e}trie Alg\'{e}brique du Bois-Marie
  1967--1969 (SGA 7 I), Dirig\'{e} par A. Grothendieck. Avec la collaboration
  de M. Raynaud et D. S. Rim.

\bibitem{AIK}
Fabrizio Andreatta, Adrian Iovita, and Minhyong Kim.
\newblock A {$p$}-adic nonabelian criterion for good reduction of curves.
\newblock {\em Duke Math. J.}, 164(13):2597--2642, 2015.

\bibitem{BakerFaber}
Matthew Baker and Xander Faber.
\newblock Metric properties of the tropical {A}bel-{J}acobi map.
\newblock {\em J. Algebraic Combin.}, 33(3):349--381, 2011.

\bibitem{Berkovich:integration}
Vladimir~G. Berkovich.
\newblock {\em Integration of one-forms on {$p$}-adic analytic spaces}, volume
  162 of {\em Annals of Mathematics Studies}.
\newblock Princeton University Press, Princeton, NJ, 2007.

\bibitem{Berthelot:cohomologie-rigide}
Pierre Berthelot.
\newblock Cohomologie rigide et cohomologie rigide a supports propres,
  premi\`{e}re partie.
\newblock available at
  \texttt{https://perso.univ-rennes1.fr/pierre.berthelot/}, 1996.

\bibitem{Besser:Coleman}
Amnon Besser.
\newblock Coleman integration using the {T}annakian formalism.
\newblock {\em Math. Ann.}, 322(1):19--48, 2002.

\bibitem{Besser:heights}
Amnon Besser.
\newblock {$p$}-adic heights and {V}ologodsky integration.
\newblock {\em J. Number Theory}, 239:273--297, 2022.

\bibitem{BesserFurusho}
Amnon Besser and Hidekazu Furusho.
\newblock The double shuffle relations for {$p$}-adic multiple zeta values.
\newblock In {\em Primes and knots}, volume 416 of {\em Contemp. Math.}, pages
  9--29. Amer. Math. Soc., Providence, RI, 2006.

\bibitem{BZ:Vologodsky}
Amnon Besser and Sarah~Livia Zerbes.
\newblock Vologodsky integration on curves with semi-stable reduction.
\newblock {\em Israel J. Math.}, 253(2):761--770, 2023.

\bibitem{BD:Kummer}
L.~Alexander Betts and Netan Dogra.
\newblock The local theory of unipotent {K}ummer maps and refined {S}elmer
  schemes.
\newblock {\em arXiv:1909.0573}.

\bibitem{betts-litt}
L.~Alexander Betts and Daniel Litt.
\newblock Semisimplicity of the {F}robenius action on {$\pi_1$}.
\newblock In {\em {$p$}-adic {H}odge theory, singular varieties, and
  non-abelian aspects}, Simons Symp., pages 17--64. Springer, Cham, [2023]
  \copyright 2023.

\bibitem{BL:stable1}
Siegfried Bosch and Werner L\"{u}tkebohmert.
\newblock Stable reduction and uniformization of abelian varieties. {I}.
\newblock {\em Math. Ann.}, 270(3):349--379, 1985.

\bibitem{BrosnanElZein}
Patrick Brosnan and Fouad El~Zein.
\newblock Variations of mixed {H}odge structure.
\newblock In {\em Hodge theory}, volume~49 of {\em Math. Notes}, pages
  333--409. Princeton Univ. Press, Princeton, NJ, 2014.

\bibitem{Chen:BHformula}
Kuo-Tsai Chen.
\newblock Integration of paths, geometric invariants and a generalized
  {B}aker-{H}ausdorff formula.
\newblock {\em Ann. of Math. (2)}, 65:163--178, 1957.

\bibitem{Chen:iterated}
Kuo-Tsai Chen.
\newblock Iterated path integrals.
\newblock {\em Bull. Amer. Math. Soc.}, 83(5):831--879, 1977.

\bibitem{ChengKatz}
Raymond Cheng and Eric Katz.
\newblock Combinatorial iterated integrals and the harmonic volume of graphs.
\newblock {\em Adv. in Appl. Math.}, 128:102190, 2021.

\bibitem{Chiarellotto:weights}
Bruno Chiarellotto.
\newblock Weights in rigid cohomology applications to unipotent
  {$F$}-isocrystals.
\newblock {\em Ann. Sci. \'{E}cole Norm. Sup. (4)}, 31(5):683--715, 1998.

\bibitem{CPS:Logpi1}
Bruno Chiarellotto, Valentina Di~Proietto, and Atsushi Shiho.
\newblock Comparison of relatively unipotent log de {R}ham fundamental groups.
\newblock {\em Mem. Amer. Math. Soc.}, 288(1430):v+111, 2023.

\bibitem{Coleman-deShalit}
Robert Coleman and Ehud de~Shalit.
\newblock {$p$}-adic regulators on curves and special values of {$p$}-adic
  {$L$}-functions.
\newblock {\em Invent. Math.}, 93(2):239--266, 1988.

\bibitem{CI:Frobandmonodromy}
Robert Coleman and Adrian Iovita.
\newblock The {F}robenius and monodromy operators for curves and abelian
  varieties.
\newblock {\em Duke Math. J.}, 97(1):171--215, 1999.

\bibitem{CI:Hidden}
Robert Coleman and Adrian Iovita.
\newblock Hidden structures on semistable curves.
\newblock {\em Ast\'{e}risque}, (331):179--254, 2010.

\bibitem{Coleman:Annals}
Robert~F. Coleman.
\newblock Torsion points on curves and {$p$}-adic abelian integrals.
\newblock {\em Ann. of Math. (2)}, 121(1):111--168, 1985.

\bibitem{Coleman:RLC}
Robert~F. Coleman.
\newblock Reciprocity laws on curves.
\newblock {\em Compositio Math.}, 72(2):205--235, 1989.

\bibitem{Coleman:monodromy}
Robert~F. Coleman.
\newblock The monodromy pairing.
\newblock {\em Asian J. Math.}, 4(2):315--330, 2000.

\bibitem{Deligne:groupefondamental}
P.~Deligne.
\newblock Le groupe fondamental de la droite projective moins trois points.
\newblock In {\em Galois groups over {${\bf Q}$} ({B}erkeley, {CA}, 1987)},
  volume~16 of {\em Math. Sci. Res. Inst. Publ.}, pages 79--297. Springer, New
  York, 1989.

\bibitem{DM:tannakian}
P.~Deligne and J.~S Milne.
\newblock {\em Tannakian Categories}, pages 101--228.
\newblock Springer Berlin Heidelberg, 1982.

\bibitem{Deligne:regularsingular}
Pierre Deligne.
\newblock {\em \'{E}quations diff\'{e}rentielles \`a points singuliers
  r\'{e}guliers}.
\newblock Lecture Notes in Mathematics, Vol. 163. Springer-Verlag, Berlin-New
  York, 1970.

\bibitem{deligne1980conjecture}
Pierre Deligne.
\newblock La conjecture de {W}eil. {II}.
\newblock {\em Publications Math{\'e}matiques de l'Institut des Hautes Etudes
  Scientifiques}, 52(1):137--252, 1980.

\bibitem{DG:groupes}
Pierre Deligne and Alexander~B. Goncharov.
\newblock Groupes fondamentaux motiviques de {T}ate mixte.
\newblock {\em Ann. Sci. \'{E}cole Norm. Sup. (4)}, 38(1):1--56, 2005.

\bibitem{esnault-hai-sun}
H\'{e}l\`ene Esnault, Ph\`ung~H\^{o} Hai, and Xiaotao Sun.
\newblock On {N}ori's fundamental group scheme.
\newblock In {\em Geometry and dynamics of groups and spaces}, volume 265 of
  {\em Progr. Math.}, pages 377--398. Birkh\"{a}user, Basel, 2008.

\bibitem{Faltings:crystalline}
Gerd Faltings.
\newblock Crystalline cohomology of semistable curve---the {${\bf
  Q}_p$}-theory.
\newblock {\em J. Algebraic Geom.}, 6(1):1--18, 1997.

\bibitem{fresse-operads}
Benoit Fresse.
\newblock {\em Homotopy of operads and {G}rothendieck-{T}eichm\"{u}ller groups.
  {P}art 1}, volume 217 of {\em Mathematical Surveys and Monographs}.
\newblock American Mathematical Society, Providence, RI, 2017.

\bibitem{Gersten:intersections}
S.~M. Gersten.
\newblock Intersections of finitely generated subgroups of free groups and
  resolutions of graphs.
\newblock {\em Invent. Math.}, 71(3):567--591, 1983.

\bibitem{GK:Frobandmonodromy}
Elmar Grosse-Kl\"{o}nne.
\newblock Frobenius and monodromy operators in rigid analysis, and
  {D}rinfel'd's symmetric space.
\newblock {\em J. Algebraic Geom.}, 14(3):391--437, 2005.

\bibitem{GK:Cech}
Elmar Grosse-Kl\"{o}nne.
\newblock The \v{C}ech filtration and monodromy in log crystalline cohomology.
\newblock {\em Trans. Amer. Math. Soc.}, 359(6):2945--2972, 2007.

\bibitem{Hadian}
Majid Hadian.
\newblock Motivic fundamental groups and integral points.
\newblock {\em Duke Math. J.}, 160(3):503--565, 2011.

\bibitem{Hain:Bowdoin}
Richard~M. Hain.
\newblock The geometry of the mixed {H}odge structure on the fundamental group.
\newblock In {\em Algebraic geometry, {B}owdoin, 1985 ({B}runswick, {M}aine,
  1985)}, volume~46 of {\em Proc. Sympos. Pure Math.}, pages 247--282. Amer.
  Math. Soc., Providence, RI, 1987.

\bibitem{Kato:logcurves}
Fumiharu Kato.
\newblock Log smooth deformation and moduli of log smooth curves.
\newblock {\em Internat. J. Math.}, 11(2):215--232, 2000.

\bibitem{KatzKaya}
Eric Katz and Enis Kaya.
\newblock $p$-adic integration on bad reduction hyperelliptic curves.
\newblock {\em Int. Math. Res. Not. IMRN}, 2020.

\bibitem{KRZB}
Eric Katz, Joseph Rabinoff, and David Zureick-Brown.
\newblock Uniform bounds for the number of rational points on curves of small
  {M}ordell-{W}eil rank.
\newblock {\em Duke Math. J.}, 165(16):3189--3240, 2016.

\bibitem{Kedlaya:unipotence}
Kiran~S. Kedlaya.
\newblock Semistable reduction for overconvergent {$F$}-isocrystals. {I}.
  {U}nipotence and logarithmic extensions.
\newblock {\em Compos. Math.}, 143(5):1164--1212, 2007.

\bibitem{Kedlaya:aws}
Kiran~S. Kedlaya.
\newblock {$p$}-adic cohomology: from theory to practice.
\newblock In {\em {$p$}-adic geometry}, volume~45 of {\em Univ. Lecture Ser.},
  pages 175--203. Amer. Math. Soc., Providence, RI, 2008.

\bibitem{Kim:original}
Minhyong Kim.
\newblock The motivic fundamental group of {$ \P^1\setminus\{0,1,\infty\}$} and
  the theorem of {S}iegel.
\newblock {\em Invent. Math.}, 161(3):629--656, 2005.

\bibitem{Lazda:descent}
Christopher Lazda.
\newblock A note on effective descent for overconvergent isocrystals.
\newblock {\em J. Number Theory}.

\bibitem{Lazda:rational}
Christopher Lazda.
\newblock Relative fundamental groups and rational points.
\newblock {\em Rend. Semin. Mat. Univ. Padova}, 134:1--45, 2015.

\bibitem{LP:Rigidcohomology}
Christopher Lazda and Ambrus P\'{a}l.
\newblock {\em Rigid cohomology over {L}aurent series fields}, volume~21 of
  {\em Algebra and Applications}.
\newblock Springer, [Cham], 2016.

\bibitem{LS:Rigid-cohomology}
Bernard Le~Stum.
\newblock {\em Rigid cohomology}, volume 172 of {\em Cambridge Tracts in
  Mathematics}.
\newblock Cambridge University Press, Cambridge, 2007.

\bibitem{lubotsky1982cohomology}
Alexander Lubotsky and Andy Magid.
\newblock Cohomology of unipotent and prounipotent groups.
\newblock {\em Journal of Algebra}, 74(1):76--95, 1982.

\bibitem{MikhalkinZharkov}
Grigory Mikhalkin and Ilia Zharkov.
\newblock Tropical curves, their {J}acobians and theta functions.
\newblock In {\em Curves and abelian varieties}, volume 465 of {\em Contemp.
  Math.}, pages 203--230. Amer. Math. Soc., Providence, RI, 2008.

\bibitem{Mokrane}
A.~Mokrane.
\newblock La suite spectrale des poids en cohomologie de {H}yodo-{K}ato.
\newblock {\em Duke Math. J.}, 72(2):301--337, 1993.

\bibitem{Ogus:logbook}
Arthur Ogus.
\newblock {\em Lectures on logarithmic algebraic geometry}, volume 178 of {\em
  Cambridge Studies in Advanced Mathematics}.
\newblock Cambridge University Press, Cambridge, 2018.

\bibitem{Serre:Lectures}
Jean-Pierre Serre.
\newblock {\em Lectures on the {M}ordell-{W}eil theorem}.
\newblock Aspects of Mathematics, E15. Friedr. Vieweg \& Sohn, Braunschweig,
  1989.
\newblock Translated from the French and edited by Martin Brown from notes by
  Michel Waldschmidt.

\bibitem{Shiho1}
Atsushi Shiho.
\newblock Crystalline fundamental groups. {I}. {I}socrystals on log crystalline
  site and log convergent site.
\newblock {\em J. Math. Sci. Univ. Tokyo}, 7(4):509--656, 2000.

\bibitem{Shiho2}
Atsushi Shiho.
\newblock Crystalline fundamental groups. {II}. {L}og convergent cohomology and
  rigid cohomology.
\newblock {\em J. Math. Sci. Univ. Tokyo}, 9(1):1--163, 2002.

\bibitem{Shiho:relative}
Atsushi Shiho.
\newblock Relative log convergent cohomology and relative rigid cohomology i.
\newblock {\em arXiv:0707.1742v2}, 2007.

\bibitem{Steenbrink:Limits}
Joseph Steenbrink.
\newblock Limits of {H}odge structures.
\newblock {\em Invent. Math.}, 31(3):229--257, 1975/76.

\bibitem{SteenbrinkZucker}
Joseph Steenbrink and Steven Zucker.
\newblock Variation of mixed {H}odge structure. {I}.
\newblock {\em Invent. Math.}, 80(3):489--542, 1985.

\bibitem{Stix:svk}
Jakob Stix.
\newblock A general {S}eifert-{V}an {K}ampen theorem for algebraic fundamental
  groups.
\newblock {\em Publ. Res. Inst. Math. Sci.}, 42(3):763--786, 2006.

\bibitem{Stoll:uniform}
Michael Stoll.
\newblock Uniform bounds for the number of rational points on hyperelliptic
  curves of small {M}ordell-{W}eil rank.
\newblock {\em J. Eur. Math. Soc. (JEMS)}, 21(3):923--956, 2019.

\bibitem{Szamuely}
Tam\'{a}s Szamuely.
\newblock {\em Galois groups and fundamental groups}, volume 117 of {\em
  Cambridge Studies in Advanced Mathematics}.
\newblock Cambridge University Press, Cambridge, 2009.

\bibitem{tsuji1999p}
Takeshi Tsuji.
\newblock p-adic {\'e}tale cohomology and crystalline cohomology in the
  semi-stable reduction case.
\newblock {\em Inventiones mathematicae}, 137(2):233--411, 1999.

\bibitem{vezzani2012pro}
Alberto Vezzani.
\newblock The pro-unipotent completion, 2012.
\newblock available at
  \texttt{http://users.mat.unimi.it/users/vezzani/Files/Research/prounipotent.pdf}.

\bibitem{Vologodsky}
Vadim Vologodsky.
\newblock Hodge structure on the fundamental group and its application to
  {$p$}-adic integration.
\newblock {\em Mosc. Math. J.}, 3(1):205--247, 260, 2003.

\bibitem{Wojtkowiak:cosimplicial}
Zdzis{\l}aw Wojtkowiak.
\newblock Cosimplicial objects in algebraic geometry.
\newblock In {\em Algebraic {$K$}-theory and algebraic topology ({L}ake
  {L}ouise, {AB}, 1991)}, volume 407 of {\em NATO Adv. Sci. Inst. Ser. C: Math.
  Phys. Sci.}, pages 287--327. Kluwer Acad. Publ., Dordrecht, 1993.

\bibitem{Zarhin1996}
Yuri~G. Zarhin.
\newblock {$p$}-adic abelian integrals and commutative {L}ie groups.
\newblock {\em J. Math. Sci.}, 81(3):2744--2750, 1996.

\end{thebibliography}
\end{document}